\documentclass[10pt,a4paper,twoside]{amsart}
\usepackage{etex}
\usepackage[dvipsnames]{xcolor}
\usepackage{amsmath,amssymb,amsfonts,enumerate,amsthm,graphicx,bm,fancyhdr,pictexwd,dcpic,mathrsfs,verbatim,aligned-overset,tikz-cd}
\usepackage[hidelinks]{hyperref}
\usepackage{hyperref}
\usepackage{amscd}
\usepackage{chngcntr}
\usepackage[all]{xy}

\xyoption{all}
\headsep=1cm \numberwithin{equation}{section}
\newtheorem{thm}{Theorem}[section]
    \theoremstyle{Definition}
    \newtheorem{defi}[thm]{Definition}
    \newtheorem{defis}[thm]{Definitions}
    \theoremstyle{Definition and Remark}
    \newtheorem{defi-rem}[thm]{Definition and Remark}
    \newtheorem{defi-rems}[thm]{Definition and Remarks}
    \newtheorem{defis-rems}[thm]{Definitions and Remarks}
    \newtheorem{defi-Nots}[thm]{Definition and Notations}
    \newtheorem{defi-Not}[thm]{Definition and Notation}
    
    \theoremstyle{Lemma}
    \newtheorem{lem}[thm]{Lemma}

    \newtheorem{rem}[thm]{Remark}
    \theoremstyle{Corollary}
    \newtheorem{cor}[thm]{Corollary}

    \newtheorem{prop}[thm]{Proposition}
    \newtheorem{prop-defi}[thm]{Proposition-Definition}

    \newtheorem{conv}[thm]{Convention}

    \newtheorem{consi}[thm]{Consideration}

\newtheorem{subdefi-rem}[subthm]{Definition and Remark}
\newtheorem{subdefi-rems}[subthm]{Definition and Remarks}
\newtheorem{subdefis-rems}[subthm]{Definitions and Remarks}
\newtheorem{subdefi-Nots}[subthm]{Definition and Notations}
\newtheorem{subdefi-Not}[subthm]{Definition and Notation}

\DeclareMathOperator{\Hom}{Hom} 

\DeclareMathOperator{\Ass}{Ass}

\DeclareMathOperator{\Ext}{Ext}
\DeclareMathOperator{\depth}{depth}
 \DeclareMathOperator{\Tot}{Tot}
\DeclareMathOperator{\grade}{grade}

\newcommand\numberthis{\addtocounter{equation}{1}\tag{\theequation}}

\newcommand{\fm}{\mathfrak{m}}

\newcommand{\fp}{\frak{p}}
\newcommand{\fq}{\frak{q}}
\newcommand{\fa}{\frak{a}}
\newcommand{\fb}{\frak{b}}
\newcommand{\fc}{\frak{c}}

 \newcommand{\Ht}{\text{H}}   
 \newcommand{\Kt}{\text{K}}    
    
 \newcommand{\as}{\text{ass}}  \newcommand{\dime}{{\rm dim}}
\newcommand{\map}{\mathfrak{p}} \newcommand{\mam}{\mathfrak{m}}

 \newcommand{\mn}{\mathbb{N}}
  
\newcommand{\Supp}{\text{Supp}}  

\newcommand{\pd}{\text{pd}} \newcommand{\id}{\text{id}}
\newcommand{\ext}{\text{Ext}}
\newcommand{\homm}{\text{Hom}}

 \newcommand{\im}{\text{im}}

 \newcommand{\tor}{\text{Tor}}

 \newcommand{\Frac}{\text{Frac}} 

\DeclareMathAlphabet{\mathcalligra}{T1}{calligra}{g}{f}

\newcommand{\llar}{-\kern-5pt-\kern-5pt\longrightarrow}
\def\restr{{\kern-1pt\restriction\kern-1pt}}

\usepackage[paper=a4paper,left=30mm,right=20mm,top=25mm,bottom=30mm]{geometry}

\begin{document}

	\title{Reductions towards a characteristic free proof of the  Canonical Element Theorem}
	\author[Ehsan Tavanfar]{Ehsan Tavanfar}
	\address{School of Mathematics, Institute for 	Research in Fundamental Sciences (IPM), P. O. Box: 	19395-5746, Tehran, Iran. }
	\email{tavanfar@ipm.ir and tavanfar@gmail.com}	
	\maketitle
	\markboth{\small \textsc{E. Tavanfar}}{\small  \textsc{{Reductions towards a characteristic free proof of the  Canonical Element Theorem}}}

\date{\today}
\makeatletter{\renewcommand*{\@makefnmark}{}
\footnotetext{(MSC2020)  13D22, 13A30, 13D45, 13H10.}\makeatother}
\makeatletter{\renewcommand*{\@makefnmark}{}
	\makeatother}
\makeatletter{\renewcommand*{\@makefnmark}{}
	\footnotetext{This research was supported by a grant from IPM.}\makeatother}

  \begin{abstract}

We reduce    Hochster's Canonical Element Conjecture  (theorem since 2016)  to a localization problem in a characteristic free way.
We prove the validity of  a new variant of the Canonical Element Theorem  (CET)  and  explain how  a characteristic free    deduction of the new variant from the original CET  would provide  us with a characteristic free proof of the CET. 
  We also  show   that  the  Balanced Big Cohen-Macaulay Module Theorem   can  be settled by a characteristic free proof if   the big Cohen-Macaulayness of  Hochster's modification module can be deduced  from the existence of a maximal Cohen-Macaulay complex in a characteristic free way.
\end{abstract}


\tableofcontents

\section{Introduction}

Let $(R,\fm_R)$ be  a   local ring of dimension $n$. The aim of this  paper is  to provide some results that we hope  can  be   useful for obtaining   a   characteristic free proof of the following   celebrated results in Commutative Algebra, all of which have been shown to be equivalent.

\textbf{Direct Summand Theorem (DST):} If $R$ is module finite over a regular local subring $P$, then the inclusion $P\rightarrow R$ splits in the category of $P$-modules.

\textbf{Monomial Theorem (MT):}  Let   $x_1,\ldots,x_n$ be a system of parameters for $R$. Then  $$\forall\ t\ge 1:\ x_1^t\cdots x_n^t\notin (x_1^{t+1},\ldots,x_n^{t+1}).$$  


\textbf{Canonical Element Theorem  (CE Property):} Let $F_\bullet$ be a minimal free resolution of $R/\mam_R$ and let  $\Kt_\bullet(\mathbf{x},R)$ be the Koszul complex of a system of parameters $\mathbf{x}$ of $R$. Then  $\phi_n\neq 0$ for any lifting  $\phi_\bullet:\Kt_\bullet(\mathbf{x},R)\rightarrow F_\bullet$  of the natural epimorphism $R/\mathbf{(x)}\rightarrow R/\mam_R$.

\textbf{Canonical Element Theorem (CET):} Let $F_\bullet$ be a free resolution of $R/\fm_R$ and consider the natural epimorphism $\pi:F_n\rightarrow \text{syz}_n^R(R/\fm_R)$. 
Let $\eta_R$  be the image of $[\pi]$ under the natural map   $$\phi:\ext^n_{R}\big(R/\fm_R,\text{syz}_n^R(R/\fm_R)\big)\rightarrow \lim\limits_{\overset{\longrightarrow}{m\in \mn}}\ext^n_{R}\big(R/\fm_R^m,\text{syz}_n^R(R/\fm_R)\big)=\Ht^n_{\fm_R}\big(\text{syz}_n^R(R/\fm_R)\big).$$  Then $\eta_R\neq 0$.

\textbf{Improved New Intersection Theorem (INIT):} Let  $F_\bullet:0\rightarrow F_s\rightarrow F_{s-1}\rightarrow \cdots \rightarrow F_0\rightarrow 0$ be a complex of finite free $R$-modules such that $\ell\big(\Ht_i(F_\bullet)\big)<\infty$ for each $i\ge 1$ and $\Ht_0(F_\bullet)$ has a minimal generator whose generated $R$-module has finite length. Then $s\ge n$.

 The  general form of the Direct Summand Theorem states that if $R'$ is a (not necessarily local) Noetherian ring that is a module finite extension of a (not necessarily local) regular ring  $P$, then $P\rightarrow R'$ splits (see \cite[Theorem (6.1)(1)]{HochsterCanonical}). However, this general case is reduced to our aforementioned local case: By arguments in \cite[last paragraph of page 539 and first paragraph of page 540]{HochsterCanonical} one can assume that $P$ is a complete regular local ring and that the  module finite extension $R'$ of $P$ is a domain. Such a module finite extension is necessarily local by \cite[Theorem 7]{CohenOnTheStructure}. 
 
 For more equivalent forms of the MT see \cite{DuttaTheMonomial}, \cite{DuttaTheMonomialII}, \cite{DuttaLocalCohomology}, \cite{OhiDirectSummand}, \cite[Theorem 4.8]{AvramovIyengarNeemanBigCohenMacaulay}, \cite[Theorem (Rob$_2$), Fact (P. Roberts), page 687]{KohLeeSomeRestrictions},  \cite[Theorem 2]{SimonStrookerStiffness}, \cite{StrookerStuckrudMonomialConjecture} and \cite[Proposition 2.11]{TavanfarAnnihilators}. 
 
 In 1973, Hochster \cite{HochsterContracted} proposed the Direct Summand Conjecture and  proved it in the case, where $R$ contains a field. In 2002, Heitmann's breakthrough \cite{HeitmannTheDirectSummand} suggested that in the mixed characteristic case the conjecture could be settled using almost ring theory. In that case, the conjecture  remained open until 2016 when Yves Andr\'e  \cite{AnLaConjecture} proved it     using  perfectoid spaces. Thereafter, Bhargav Bhatt \cite{BhattOnTheDirectSummand} proved the derived variant of it, again using  perfectoid spaces.  Perfectoid algebras are mixed characteristic and non-Noetherian algebras. In the equicharacteristic case,    proofs are known that do not use non-Noetherian rings. However, the known proofs  still depend on the characteristic of $R$.   In prime characteristic, one proof uses   Grothendieck's non-Vanishing Theorem together with the Frobenius homomorphism. In equal-characteristic zero, the splitting is obtained using the normalized trace map.

 
 
 

For the following implications,  proofs have been found that are characteristic free and   do not use non-Noetherian rings.

\begin{center}\begin{tikzcd}[   ]
	\text{CET for \ } R\arrow[Leftrightarrow, "\text{\cite[Theorem (3.15)]{HochsterCanonical}}"]{rrrrr}  &&&&& 
	\text{}\ \text{CE Property  for } R 
	\arrow[Leftrightarrow,"\text{\cite[Theorem (2.6)]{HochsterCanonical} and \cite{DutaOnTheCanonical}}"]{rrrrr} \arrow[Rightarrow, 
	shift left=0.9ex, 
	"\text{\cite[Remarks (2.2)(7), page 506]{HochsterCanonical}}"]{d}  
	&&&&&  \text{INIT for } R \\ 
	&&&&& \text{MT\ for\ }R\arrow[Rightarrow,
	 shift left=0.9ex, 
	"\text{\cite[Theorem 1]{HochsterContracted}}"]{d} \\
	&&&&& \text{DST\ for\ } R
	\end{tikzcd}\end{center}


On the other hand, for deducing the CE Property from the DST for a residually prime characteristic local ring,  the only known  proof needs the assumption that every local ring with prime characteristic residue field satisfies the DST  and the  proof  uses non-Noetherian rings (see \cite[Theorem (2.9)]{HochsterCanonical}).

In  Section \ref{SectionMainResults},  we reduce the existence  of  a characteristic free proof of the CE Property (and hence of  all other aforementioned  conjectures) which avoids the use of non-Noetherian rings to the well-behavior of the CE Property under  localization  (see Theorem \ref{MonomialTheoremCharacteristicFree} and its sufficient condition (\ref{itm:SufficentConditionLocalization})). 
 Moreover, in Section \ref{SectionNewVariant}, we present a  new variant of the Canonical Element Theorem (see Proposition-Definition \ref{PropositionVarientCET}) and we prove its validity (Corollary \ref{CorollaryVariantHolds}). Our proof of the validity of this new variant is based on the existence of a balanced big Cohen-Macaulay module that behaves well under  localization, and for this purpose we prove that the balanced big Cohen-Macaulay property for modules over excellent normal domains behaves well under  localization  (Theorem \ref{TheoremBalancedBigCMModuleLocalizes}\ \footnote{Alternatively, we could apply \cite[Proposition 2.11]{BaMaPaScTuWaWi21}. We refer to  Remark \ref{RemarkWhyLocalizingBBCMModules}  for some comments on the utility of Theorem \ref{TheoremBalancedBigCMModuleLocalizes} in parallel to  \cite[Proposition 2.11]{BaMaPaScTuWaWi21}.
 }). 
Furthermore, a possible characteristic free deduction  of  the new variant of the Canonical Element Theorem  from the original Canonical Element Theorem, if provable,  can provide us with a characteristic free proof of the Canonical Element Theorem  as shown in Theorem \ref{MonomialTheoremCharacteristicFree} (and its sufficient condition (\ref{itm:SufficientConditionExcellentUFD})).

As  another main result of the paper, applying    \cite[Proposition 4.3]{IyengarMaSchwedeWalker} and Theorem \ref{MonomialTheoremCharacteristicFree}, we present a sufficient condition for obtaining a characteristic free proof for the existence of  balanced big Cohen-Macaulay modules (Corollary \ref{CorollaryBigCohenMacaulay}).

We end the introduction by  outlining the main steps of the  proof of Theorem \ref{MonomialTheoremCharacteristicFree}.

\begin{enumerate}[(i)]
	\item  	\label{itm:GQorensteinReduction}
	We reduce the CE Property to the case  where $R$ is  a complete quasi-Gorenstein domain that is a locally complete intersection in codimension $\le 1$   by  considering the same $R$-algebra as in the proof of  \cite[Proposition 2.7]{TavanfarReduction}. 

	%

    \item \label{itm:GenericLinkageOk} Suppose that $R$ is a quasi-Gorenstein domain as in (i), and that $S$ is the localization at the homogeneous maximal ideal of a generic linkage of $R$   in the sense of Huneke and Ulrich \cite{HunekeUlrichDivisorClass}. Then $S$ is an almost complete intersection domain that is a  locally complete intersection in codimension $\le 1$. Since $S$ is an almost complete intersection  domain with enough depth, we can apply  Lemmas \ref{AttLocalCohomologyDepthCanonicalModule} and \ref{DeformationOFCEC}(\ref{itm:CETandDeformation}) and an inductive argument to see that it satisfies the CE Property (cf. \cite[step 2, page 238]{DuttaGriffithIntersectionMultiplicity} for an analogous result with a different proof which seems to be characteristic dependent as  it calls upon the proof of  \cite[Theorem 3.2]{DutaOnTheCanonicalII}).

	\item \label{itm:Bridge} Suppose that $R$ is  a  quasi-Gorenstein domain as in (i). By the definition of generic linkage, there exists an excellent   regular local ring $A$ and linked ideals $\fa$ and $\fb$ of $A$ such that  $S:=A/\fa$ is as in (ii)    while $A/\fb$  is a trivial deformation of $R$.  We use the (factorial) extended Rees algebra $A[\fa t,t^{-1}]$ ($A[\fa t]\cong \text{Sym}_A(\fa)$) of $\fa$  as a bridge between $S$ and $A/\fb$    to deduce the  CE Property for  $A/\fb$ (and thus for $R$) from the CE Property for $S$, under the validity of either of the two conditions of Theorem \ref{MonomialTheoremCharacteristicFree}.  The shape of this bridge can be observed via Proposition \ref{StructureProposition}(\ref{itm:QuotientOfEReesAlgebraIsACI}) together with   Lemma  \ref{LemmaEquationOfSymmetricAlgebra}(\ref{itm:ColonIdealAsaSpecilizationOfTheSymmetricAlgebra}).  
	 Here we  make an essential use of  Parts (\ref{itm:HomogeneousMaximalRegularSequenceOfEReesAlgebra})  and (\ref{itm:SpecializationOfEReesAlgebraIsS2}) of Proposition \ref{StructureProposition} by which we were able to prove that the extended Rees algebra $A[\fa t,t^{-1}]$ also satisfies the CE Property. 
	 The imposed conditions in Theorem \ref{MonomialTheoremCharacteristicFree} are only used for  the deduction of the CE Property for the localization $A[\fa t,t^{-1}]_{ht}$ of $A[\fa t,t^{-1}]$ from the CE Property of $A[\fa t,t^{-1}]$ (Step 7 of the proof of Theorem \ref{MonomialTheoremCharacteristicFree}). In our proof of the CE Property, this deduction is thus the only obstacle to obtaining a characteristic free proof without any imposed condition. 
	 

\end{enumerate}

In Part (\ref{itm:Bridge}) (above) 
we tried to follow the philosophy of Strooker and  St\"ukrad  \cite[page 158]{StrookerStuckrudMonomialConjecture} who asked the following analogous question: If $B$ and $B'$ are linked local rings and $B'$ satisfies the Small-Cohen-Macaulay Conjecture, does  $B$    also satisfy the Small Cohen-Macaulay Conjecture?  Lemma  \ref{LemmaEquationOfSymmetricAlgebra}(\ref{itm:ColonIdealAsaSpecilizationOfTheSymmetricAlgebra})  as well as the use of generic linkage  were inspired   by  (the quasi-Gorenstein counterpart of)  Ulrich  \cite{UlrichGorenstein}, as mentioned in our earlier paper \cite{TavanfarReduction}. 
The generic linkage yields   the linear type ideal  $\fa$   with factorial extended Rees Algebra   $A[\fa t,t^{-1}]$ (here \cite{HochsterCriteria} and \cite{HunekeAlmost} are also used).

\section{Preliminaries}

In this paper,  rings are Noetherian and commutative  with identity.  Graded rings and graded modules  appearing in this paper are  $\mathbb{Z}$-graded  and  often graded rings are \textit{graded local} in the sense that  they admit a unique homogeneous  ideal that is maximal with respect to all (homogeneous and non-homogeneous) proper ideals.   For any $i\in \mathbb{Z}$ the $i$-component of a graded module $M$ consisting of elements of degree $i$ is denoted by $M_{[i]}$. If $(F_\bullet,\partial^{F_\bullet}_\bullet)$ is a free resolution of a module $M$, then  $\text{syz}^{F_\bullet}_i(M):=F_i/\text{im}(\partial^{F_\bullet}_{i+1})$ for any $i\ge 1$. Let $\mathbf{x}:=x_1,\ldots,x_t$ be a sequence of elements of a ring $S$ and $M$ be an $S$-module. The notation $\mathbf{x}^v$ (respectively, $\mathbf{x}^{\underline{v}}$ where $\underline{v}=(v_1,\ldots,v_t)\in \mathbb{N}^t$) denotes the sequence $x_1^v,\ldots,x_t^v$ (respectively, $x_1^{v_1},\ldots,x_t^{v_t}$). For the definition   of the Koszul complex $K_\bullet(\mathbf{x};M)$ (and the Koszul homologies $H_i(\mathbf{x};M)$) we follow \cite[Definition 1.6.1]{BrunsHerzogCohenMacaulay} together with \cite[The Koszul Complex of a Sequence, page 46]{BrunsHerzogCohenMacaulay}.

\subsection{Preliminaries on   the Koszul complexes}
The following lemmas concerning the Koszul complex are well-known, but we prove them  for the convenience of the reader.

\begin{lem} \label{LemmaBasicFactInHomologicalAlgebra} Let $\mathbf{x}:=x_1,\ldots,x_d$ be a sequence of elements of a ring $S$ and $$(F_\bullet,\partial^{F_\bullet}_{\bullet}):=\cdots \rightarrow F_n\overset{\partial^{F_\bullet}_n}{\rightarrow} \cdots \rightarrow F_{n-1}\overset{\partial^{F_\bullet}_{n-1}}{\rightarrow}\cdots  \rightarrow F_1\overset{\partial^{F_\bullet}_1}{\rightarrow} F_0\rightarrow 0$$ be an acyclic complex. Suppose that $\varphi_\bullet,\varphi_\bullet':K_\bullet(\mathbf{x};S)\rightarrow F_\bullet$ are  chain maps of complexes such that 
	$\overline{\varphi_0},\overline{\varphi'_0}:S/\mathbf{x}S\rightarrow H_0(F_\bullet)$
	coincide. 
	\begin{enumerate}[(i)]	
		\item\label{itm:NullHomotopic} 	$\varphi_0-\varphi'_0$ is null-homotopic. 
		\item\label{itm:RelationInSyz} $\varphi_d(1)+\text{im}(\partial^{F_\bullet}_{d+1})-\varphi'_d(1)+\text{im}(\partial^{F_\bullet}_{d+1})\in (x_1,\ldots,x_d)\big(F_{d}/\text{im}(\partial^{F_\bullet}_{d+1})\big)$.
	\end{enumerate}
	\begin{proof}	
		(\ref{itm:NullHomotopic})	 Let $\pi:F_0\rightarrow H_0(F_\bullet)$ be the natural epimorphism. By our hypothesis, $\pi(\varphi_0-\varphi'_0)=0$. As $K_0(\mathbf{x};S)$ is free and $$F_1\overset{\partial^{F_\bullet}_1}{\rightarrow} F_0\overset{\pi}{\rightarrow} H_0(F_\bullet)$$ is   exact, there exists $g_0:K_0(\mathbf{x};S)\rightarrow F_1$ such that $\partial^{F_\bullet}_1\circ g_0=\varphi_0-\varphi'_0$. Set, $g_{-1}=0$ as well. Assume  inductively that for some $i\ge 0$,  $g_{i}:K_{i}(\mathbf{x};S)\rightarrow F_{i+1}$ and $g_{i-1}:K_{i-1}(\mathbf{x};S)\rightarrow F_{i}$ are constructed such that  
		\begin{equation}
		\label{Label}
		\varphi_{i}-\varphi_{i}'=\partial^{F_\bullet}_{i+1}\circ g_i+g_{i-1}\circ \partial^{K_\bullet(\mathbf{x};S)}_{i}.
		\end{equation}
		It suffices to find $g_{i+1}:K_{i+1}(\mathbf{x};S)\rightarrow F_{i+2}$ with the desired property.  But,
		\begin{align*}
		\partial^{F_\bullet}_{i+1}\circ\big((\varphi_{i+1}-\varphi'_{i+1})-(g_{i}\circ \partial^{K_\bullet(\mathbf{x};S)}_{i+1})\big)
		&
		=\partial^{F_\bullet}_{i+1}\circ(\varphi_{i+1}-\varphi'_{i+1})-\partial^{F_\bullet}_{i+1}\circ(g_{i}\circ \partial^{K_\bullet(\mathbf{x};S)}_{i+1})
		&\\&
		=\partial^{F_\bullet}_{i+1}\circ(\varphi_{i+1}-\varphi'_{i+1})-(\partial^{F_\bullet}_{i+1}\circ g_{i})\circ \partial^{K_\bullet(\mathbf{x};S)}_{i+1}
		&\\&
		\overset{\text{Eq.\ }(\ref{Label})}{=} 
		\partial^{F_\bullet}_{i+1}\circ(\varphi_{i+1} - \varphi'_{i+1})
		-(\varphi_{i}-\varphi_{i}'-g_{i-1}\circ \partial^{K_\bullet(\mathbf{x};S)}_{i})\circ \partial^{K_\bullet(\mathbf{x};S)}_{i+1}
		&\\&
		= 
		\partial^{F_\bullet}_{i+1}\circ(\varphi_{i+1} - \varphi'_{i+1})
		-(\varphi_{i}-\varphi_{i}')\circ \partial^{K_\bullet(\mathbf{x};S)}_{i+1}
		&\\&
		=\varphi_i\circ \partial^{K_\bullet(\mathbf{x};S)}_{i+1}-\varphi'_i\circ \partial^{K_\bullet(\mathbf{x};S)}_{i+1}-(\varphi_{i}-\varphi_{i}')\circ \partial^{K_\bullet(\mathbf{x};S)}_{i+1}
		&\\&
		=0
		\end{align*}
		Thus we get the map  $$(\varphi_{i+1}-\varphi'_{i+1})-(g_{i}\circ \partial^{K_\bullet(\mathbf{x};S)}_{i+1}):K_{i+1}(\mathbf{x};S)\rightarrow \ker(\partial^{F_\bullet}_{i+1}).$$ Thence, since $F_{i+2}\overset{\partial^{F_\bullet}_{i+2}}{\rightarrow} F_{i+1}\overset{\partial^{F_{\bullet}}_{i+1}}{\rightarrow} F_i$ is exact,  there exists some $g_{i+1}:K_{i+1}(\mathbf{x};S)\rightarrow F_{i+2}$ such that $\partial^{F_\bullet}_{i+2}\circ g_{i+1}=(\varphi_{i+1}-\varphi'_{i+1})-(g_{i}\circ \partial^{K_\bullet(\mathbf{x};S)}_{i+1})$ as required.
		
		(\ref{itm:RelationInSyz}) Since $\varphi_\bullet$ and $\varphi'_\bullet$ are null-homotopic by the previous part, following the notation of the previous part, we have $\varphi_d(1)-\varphi'_d(1)=\partial^{F_\bullet}_{d+1}\circ g_d(1) + g_{d-1}\circ \partial^{{K_\bullet(\mathbf{x};S)}}_{d}(1)$. So modulo $\text{im}(\partial^{F_\bullet}_{d+1})$ the image of $\varphi_d(1)-\varphi'_d(1)$ coincides with the image of $$g_{d-1}\circ \partial^{{K_\bullet(\mathbf{x};S)}}_{d}(1)=g_{d-1}(x_1,-x_2,\ldots,(-1)^{d+1}x_d)\in (x_1,\ldots,x_d)F_{d}.\qedhere$$
	\end{proof}
\end{lem}

\begin{lem}\label{LemmaInducedMapOnKoszulHomolgoies}
	Let $S$ be a ring and  let $\mathbf{x}:=x_1,\ldots,x_d$ and $\mathbf{y}:=y_1,\ldots,y_d$ be sequences of elements of  $S$ such that $(\mathbf{y})\subseteq (\mathbf{x})$.
	\begin{enumerate}[(i)]
		\item\label{itm:GeneralSubIdealSequence} 
		There exists a chain map $\varphi_\bullet:K_\bullet(\mathbf{y};S)\rightarrow K_\bullet(\mathbf{x};S)$ of Koszul complexes such that $\varphi_0=\id_{S}$. In particular, $\varphi_\bullet$ lifts the natural epimorphism $S/(\mathbf{y})\rightarrow S/(\mathbf{x})$.
		\item\label{itm:PowerOfX} Suppose that $\mathbf{y}=\mathbf{x}^{\underline{v}}$ for some $\underline{v}=(v_1,\ldots,v_d)\in \mathbb{N}^d$. Then  a chain map $\varphi_\bullet:K_\bullet(\mathbf{x}^{\underline{v}};S)\rightarrow K_\bullet(\mathbf{x};S)$ exists such that $\varphi_0=\id_S$, and moreover   $\varphi_{d}:S\rightarrow S$  is the multiplication map by $x_1^{v_1-1}\cdots x_d^{v_d-1}$ ($S=K_d(\mathbf{x}^{\underline{v}};S)=K_d(\mathbf{x};S)$).
	\end{enumerate}
  \begin{proof}
  	(\ref{itm:GeneralSubIdealSequence}) For each $1\le i\le d$, by our hypothesis  there are elements $s_{1,i},\ldots,s_{d,i}$ such that $y_i=\sum\limits_{j=1}^ds_{ji}x_{j}$. Let $\mathbf{A}:=\begin{pmatrix}s_{1,1}&s_{1,2}&\cdots&s_{1,d}\\ s_{2,1}&s_{2,2}&\cdots&s_{2,d}\\\vdots&\vdots&\vdots&\vdots\\s_{d,1}&s_{d,2}&\cdots&s_{d,d}\end{pmatrix}$, $\mathbf{B}:=\begin{pmatrix}x_1&x_2&\cdots&x_d\end{pmatrix}$ and   $\mathbf{C}:=\begin{pmatrix}y_1&y_2&\cdots&y_d\end{pmatrix}$ be matrices whose corresponding homomorphisms of free modules, $f_{\mathbf{A}},f_{\mathbf{B}}$ and $f_{\mathbf{C}}$ yield the commutative diagram   \ 
  		$\begin{CD} 
  		  S^d @>f_{\mathbf{C}}>> S\\
  		  @VVf_{\mathbf{A}}V @V\id_{S}VV\\
  		  S^d @>>f_{\mathbf{B}}> S
  	 \end{CD}$. Consequently, according to the definition of the Koszul complex  as well as \cite[Proposition 1.6.8 and its preceding paragraph]{BrunsHerzogCohenMacaulay}   $f_{\mathbf{A}}$ and $\id_{S}$ in the above commuting square extend to a chain map of Koszul complexes $\bigwedge f_{\mathbf{A}}:K_\bullet(\mathbf{y};S)\rightarrow K_\bullet(\mathbf{x};S)$ that is a homomorphism of $S$-DG algebras (see also \cite[page 41]{BrunsHerzogCohenMacaulay}, where $\bigwedge\varphi$ is defined). Note that since  $\bigwedge{f_{\mathbf{A}}}$ is an $S$-algebra homomorphism so $(\bigwedge{f_{\mathbf{A}}})_{0}=\id_{S}$. Hence the desired conclusion follows by setting $\varphi_{\bullet}:=\bigwedge{f_{\mathbf{A}}}$.

  	 (\ref{itm:PowerOfX}) In this case, we can set the matrix $\mathbf{A}$ to be the diagonal matrix $\text{diag}(x_1^{v_1-1},\ldots,x_d^{v_d-1})$. Hence, following the notation used in the proof of part  (\ref{itm:GeneralSubIdealSequence}) we have 
  	 \begin{align*}
  	   \varphi_d(1_{K_{d}(\mathbf{x}^{\underline{v}};S)})
  	   &=
  	   (\bigwedge f_{\mathbf{A}})_d(e_1\wedge\cdots\wedge e_d)
  	   &\\&=
  	   f_{\mathbf{A}}(e_1)\wedge\cdots\wedge f_{\mathbf{A}}(e_d) && (\text{see \cite[page 41]{BrunsHerzogCohenMacaulay}, where $\bigwedge \varphi$ is defined})
  	   &\\&=
  	   (x_1^{v_1-1}\cdots x_{d}^{v_d-1})(e_1\wedge\cdots \wedge e_d)
  	   &\\&=
  	   (x_1^{v_1-1}\cdots x_{d}^{v_d-1})1_{K_{d}(\mathbf{x};S)}. &&\qedhere
  	 \end{align*}
  \end{proof}
\end{lem}

\subsection{Preliminaries on   quasi-Gorenstein and  almost complete intersection rings, etc}

 For the definition and some properties of the  canonical module $\omega_{S}$ of a local ring $S$  see \cite[ Chapter 12]{BrodmannSharpLocal} (see also \cite{AoyamaSomeBasic}).

 \begin{rem}\label{RemarkCanonicalModuleColonIdealPresentation}
 \emph{Let $S$ be a  local ring  which is  a homomorphic image of a   Gorenstein local ring $G$. Suppose that  $S=G/\mathfrak{g}$ and that $l=\text{height}(\mathfrak{g})$.   By \cite[Remarks 12.1.3(iii)]{BrodmannSharpLocal} and \cite[Corollary 2.1.4]{BrunsHerzogCohenMacaulay}, $\omega_{S}\cong\ext^{l}_{G}(S,G).$ In particular,  if $\mathfrak{c}$ is a complete intersection contained in $\mathfrak{g}$ with the same height as  $\mathfrak{g}$, then
  \begin{equation}
   \label{EquationCanonicalModule}
   \omega_{S}\cong\ext^l_G(S,G)\overset{\text{\cite[Lemma 2(i), page 140]{Matsumura}}}{\cong} \homm_{G}(G/\mathfrak{g},G/\mathfrak{c})\cong (\fc:\mathfrak{g})/\fc.
 \end{equation}}
\end{rem}

\begin{defi}
	\emph{An  ideal 
		 $\fa$ of a   
		 regular local ring $A$ is called an \textit{almost complete intersection ideal}  if $\mu(\fa)=\text{height}(\fa)+1$, where $\mu$ denotes the minimal number of generators. An \textit{almost complete intersection ring}  is a residue ring of a   
		 regular local ring by an almost complete intersection ideal.}
\end{defi}

Note that if $\fa$ is an almost complete intersection ideal of a regular local ring $A$, then by    \cite[Exercise 1.2.21]{BrunsHerzogCohenMacaulay} there exists a complete intersection ideal $\fc$ of $A$ with the same height as  $\fa$ such that $\fa=\fc+(h)$ for some $h\in A$.

\begin{defi} \emph{A local ring $S$ is called \textit{quasi-Gorenstein} if it admits a  canonical module that is a   free module, necessarily of rank $1$.  An ideal $\fb$ of a regular local ring $A$ is called a \textit{quasi-Gorenstein ideal} if $A/\fb$ is a quasi-Gorenstein ring.}
\end{defi}

\begin{defi}
	\emph{Let $\fa$ be an 
		ideal of a 
		   ring $A$. Then   $\fa$ is called \textit{unmixed} if $\dim(A/\fp)=\dim(A/\fa)$ for all $\fp\in \text{ass}(\fa)$.}
\end{defi}

\begin{defi}\label{DefinitionRemarkLinkage} \emph{Let $A$ be a regular (graded) local ring and $\fa,\fb$ be (homogeneous)  ideals of $A$. Then we say that $\fa$ and $\fb$ are \textit{linked} over a (homogeneous) complete intersection ideal $\fc\subseteq \fa\cap \fb$ provided $\fb=\fc:\fa,\ \fa=\fc:\fb$. In this case, we say that $A/\fa$ and $A/\fb$ are \textit{linked rings}.
}
\end{defi}

\begin{rem}\label{RemarkLinkedSoTheSameHeight}
	\emph{From the above definition of  linkage, it is easily concluded  that  $\text{height}(\fa)=\text{height}(\fc)=\text{height}(\fb)$ provided the ideals $\fa$ and $\fb$ are     linked over the complete intersection $\fc$ and $\fa,\fb$ are proper (because any complete intersection ideal is unmixed). Moreover,  linked ideals are necessarily unmixed (see \cite[Lemma 2.1]{LynchAnnihilators}).}	
\end{rem}

 Some of our results in this  paper exploit the linkage of unmixed almost complete intersection ideals and quasi-Gorenstein ideals.  
 Although the next lemma is well-known to experts,  we include it in our paper for the convenience of the reader. 

\begin{lem}\label{RemarkLinkageOfQuasiGorensteinIdealsAndAlmostCompleteIntersectionIdeals}
	\emph{Let $A$ be a regular local ring. 
	\begin{enumerate}[(i)]
		\item\label{itm:LinkageOfQGrenseteinIsAlmostCI}   Let $\fb$ be a quasi-Gorenstein ideal of $A$ of height $g$ and  $\fc$ be a height $g$ complete intersection ideal contained properly in $\fb$. Then $\fa:=\fc:\fb$ is an unmixed almost complete intersection ideal  $\fa=\fc+(h)\subset A$. Moreover, $\fa$ and  $\fb$ are linked over $\fc$. 
		\item\label{itm:ACIIsLinkedToQG}  Let $\fa=\fc+(h)$ be a height $g$ unmixed almost complete intersection ideal of $A$  containing the complete intersection ideal $\fc$ of height $g$. Then $\fb:=\fc:\fa$ is a quasi-Gorenstein ideal. Moreover, $\fa$ and $\fb$ are linked over $\fc$.
\end{enumerate}}
 \begin{proof}
		\item [(\ref{itm:LinkageOfQGrenseteinIsAlmostCI})]   Since the canonical module is always $S_2$ (\cite[(1.10)]{AoyamaSomeBasic}) and $A/\fb\cong \omega_{A/\fb}$, so  $A/\fb$ satisfies the $(S_2)$-condition.  
		Immediately $\fb$ is an unmixed ideal since $A/\fb$ is $(S_1)$ and equidimensional. 
		 Then $\fa$ and $\fb$ are linked  by \cite[Proposition 2.2]{SchezelNotes} (the term ``pure height'' in the statement of  \cite[Proposition 2.2]{SchezelNotes} points to the same concept of unmixedness as in our paper) and so $\fa$ is unmixed by Remark \ref{RemarkLinkedSoTheSameHeight}.
		 	 Also in view of the display  (\ref{EquationCanonicalModule}),  $\fa/\fc\cong \omega_{A/\fb}\cong A/\fb$ 
		 	 and so 
		 	 $\fa/\fc$ is a cyclic module. Hence $\fa=\fc+(h)$ (generated minimally, so $\mu(\fa)=g+1$) for some $h\in A$.  Note that $\fc+(h)$ is not a complete intersection, since otherwise $\fb=\fc:\fa=\fc$ violates our hypothesis that $\fc\subsetneq \fb$. Consequently, $\fa$ is an almost complete intersection. 
		 
	   \item [(\ref{itm:ACIIsLinkedToQG})] Again by \cite[Proposition 2.2]{SchezelNotes}, $\fa$ and $\fb$ are linked.
	    Thus, in view of the display   (\ref{EquationCanonicalModule}), $\omega_{A/\fb}
	    \cong \fa/\fc$.
	     Moreover, as $(\fc:h)=\fb,\ \fa/\fc\cong A/\fb$ and $A/\fb$ is quasi-Gorenstein.
  \end{proof}
\end{lem}

\begin{defi} \emph{Let $H$ be an Artinian $S$-module. A prime ideal $\fp$ of $S$ is called an \textit{attached prime ideal} of $A$ if $\fp=0:_{S}(H/N)$ for some submodule $N$ of $H$.  The set of attached prime ideals of $H$ is denoted by  $\text{Att}(H)$.}
\end{defi}

\begin{lem}\label{LemmaAttachedPrime} Let $H$ be an Artinian $S$-module.
  \begin{enumerate}[(i)]
	\item\label{itm:AttIsFinite}   $\text{Att}(H)$ is a  finite set.
	\item\label{itm:AttAndAss}  If $M$ is a finitely generated $S$-module and $(S,\fm_{S})$ is local, then  $$\text{Att}\Big(\text{Hom}_{S}\big(M,\text{E}(S/\fm_{S})\big)\Big)=\text{Ass}(M).$$ 
	\item\label{itm:AttAndSurjectivity}  Let $x\in S$. Then    $x\notin \bigcup\limits_{\fp\in \text{Att}(H)}\fp$ if and only if the multiplicative map by $x$ on $H$ is surjective.
  \end{enumerate}	
  \begin{proof}
  	 The statements follow from \cite[Chapter 7, Corollary 10.2.20,  Proposition 7.2.11(i), respectively]{BrodmannSharpLocal}.
  \end{proof}
\end{lem}

\begin{lem} 
	\label{AttLocalCohomologyDepthCanonicalModule}
 Let $(S,\fm_{S})$ be a   local ring of dimension $\ge 2$ admitting a canonical module $\omega_{S}$. If $$\mathfrak{m}_{S}\in \text{Att}\big(H^{\dim(S)-1}_{\fm_{S}}(S)\big)$$ then $\depth(\omega_{S})=2$.
\end{lem}
\begin{proof}
	This follows from \cite[Lemma 2.1(2)(i)]{AoyamaGotOnTheEndomorphism}.
\end{proof}

\begin{rem}
	\emph{
		Under the conditions that $S$ is $(S_2)$ with Cohen-Macaulay formal fibers and $\widehat{S}$ is equidimensional,  the above lemma  and its converse can be proved via some results of \cite{DibaeiAStudy} and \cite{DibaeiJafariCohen-Macaulay}. For the details see Remark 2.7 of the arXiv version of \cite{TavanfarTousiAStudy}. For another proof of this lemma  and its converse    assuming only that  $\depth(S)>0$ and $\dim(S)\ge 3$, see Lemma 2.5 of the arXiv version of \cite{TavanfarTousiAStudy}. 
	}
\end{rem}

The following lemma   will be used in the proof of Theorem \ref{MonomialTheoremCharacteristicFree}.

\begin{lem}\label{LemmaIndeterminateActsSurjectivelyonLocalCohomology} Let $({S},\fm_{S})$ be a local ring  and  $X$ be an indeterminate over ${S}$. Then,  for all $i$  the multiplication map by $X$  on $H^{i}_{(\fm_{S},X)}\big({S}[X]_{(\fm_{S},X)}\big)$ is surjective.
	\begin{proof}
		$H^0_{(\fm_{S},X)}({S}[X]_{(\fm_{S},X)})=0$ in view of the regularity of the element $X$. Fix some $1\le i\le \dim({S})+1$. By \cite[4.3.2 Flat Base Change Theorem]{BrodmannSharpLocal}, $H^{i}_{(\fm_{S})}({S}[X])\cong H^{i}_{\fm_{S}}({S})[X]$, so $\Gamma_{(X)}\big(H^i_{(\fm_{S})}({S}[X])\big)=0$. Consequently, from  \cite[Corollary 1.4]{SchenzelOnTheUse} we get \begin{equation}
		\label{EquationIsomorphismofLocalCohomolgoies}
		H^i_{(\fm_{S},X)}({S}[X])\cong H^1_{(X)}\big(H^{i-1}_{(\fm_{S})}({S}[X])\big).
		\end{equation}  
		But $X$ acts surjectively on $H^{i-1}_{(\fm_{S})}({S}[X])_{X}$, while $H^1_{(X)}\big(H^{i-1}_{(\fm_{S})}({S}[X])\big)$ is a quotient of $H^{i-1}_{(\fm_{S})}({S}[X])_{X}$ by \cite[Corollary 2.2.21]{BrodmannSharpLocal}. From this fact and the display  (\ref{EquationIsomorphismofLocalCohomolgoies}) we conclude that $X$ acts surjectively on $H^i_{(\fm_{S},X)}({S}[X])\cong H^i_{(\fm_{S},X)}({S}[X]_{(\fm_{S},X)})$ (since  $(\fm_{S},X)$ is a maximal ideal, so the localization map at $(\fm_{S},X)$ is an isomorphism on each $(\fm_{S},X)$-torsion module).
	\end{proof}
\end{lem}
\subsection{Preliminaries on   the CE Property}
  In this subsection, we need to use the concept of  projective resolution of bounded below complexes. In the sequel,   $\simeq$ denotes the quasi-isomorphism of complexes.

\begin{defis}\label{DefinitionHyperhomology}
	\emph{Let $S$ be a  (Noetherian commutative) ring. A complex $C_\bullet$ of $S$-modules  is called a \textit{bounded below complex} if $C_i=0$ for $i\ll 0$. It is called a \textit{homologically bounded below complex} if $H_i(C)=0$ for all $i\ll 0$.  It is  called a degreewise finite complex provided $C_i$ is a finitely generated $S$-module for each $i$.   A \textit{projective resolution} of a homologically bounded below complex $C_\bullet$ is a bounded below complex $P_\bullet$ consisting of projective $S$-modules together with a quasi-isomorphism $\pi_\bullet:P_\bullet\overset{\simeq}{\rightarrow} C_\bullet$ (see \cite[Definitions (A.3.1)(P)]{ChristensenGorenstein}).}
\end{defis}

\begin{rem}\emph{\label{RemarkNiceProjectiveResolutionOfComplexes} (see \cite[Remarks 1.7]{AvramovFoxbyHomological})  Let $S$ be a (Noetherian commutative) ring and $C_\bullet$ be a homologically bounded below complex of $S$-modules with finitely generated homologies. Then $C_\bullet$ has a projective resolution $\pi_\bullet:P_\bullet\overset{\simeq}{\rightarrow} C_\bullet$ such that each $P_i$ is a finitely generated free module and such that $P_i=0$ for each $i<\inf\{j\in \mathbb{Z}:H_j(C_\bullet)\neq 0\}$  
		(\cite[Remarks 1.7]{AvramovFoxbyHomological}).
	}
\end{rem}

The following lemma will be used in the proof of Lemma \ref{LemmaHyperHomology}. This lemma is pointed out to the author by the referee in order to obviate the need for spectral sequences in the (previous) author's  proof.

\begin{lem}\label{LemmaPreliminaryLemma}
	If $L_\bullet,C_\bullet$  are both bounded below complexes of $S$-modules (or both bounded above)   and the complex $L_\bullet\otimes_SC_m$  is exact at $L_{i-m}\otimes_S C_m$ for all $m$, then $H_i(L_\bullet\otimes_SC_\bullet)=0$.
\end{lem}
\begin{proof}
 We only prove the bounded below case. By a naive shifting if necessary, to simplify the notation and without loss of generality, we   may and we do assume that    $C_j=0,L_j=0$ provided $j<0$.  So at the $i$-th spot the complex $L_\bullet\otimes_SC_\bullet$ is $$ \bigoplus_{j=0}^{i+1} L_{i-j+1}\otimes_S C_{j} \overset{\partial^{L_\bullet\otimes C_\bullet}_{i+1}}{\longrightarrow}  \bigoplus_{j=0}^i L_{i-j}\otimes_S C_{j} \overset{\partial^{L_\bullet\otimes C_\bullet}_{i}}{\longrightarrow}  \bigoplus_{j=0}^{i-1} L_{i-j-1}\otimes_S C_{j}$$ with well-known differentials. Let $h\in H_i(L_\bullet\otimes_SC_\bullet)$. Assume, inductively, that there is  a representative  
  $(\alpha_j)_{j=0}^i\in \ker \partial^{L_\bullet\otimes C_\bullet}_{i}$, where $\alpha_j\in L_{i-j}\otimes_SC_{j}$, such that $h=[(\alpha_j)_{j=0}^i]$ and $\alpha_j=0$ for each $j> l$.   We claim that then there is a representative of  $h$ as  
  \begin{equation}
  	\label{EquationHomologyClassVanish}
    h=[(\alpha'_j)_{j=0}^i]\text{\ \ s.t.\ \ }\alpha'_j\in L_{i-j}\otimes_SC_{j}, \ \ \alpha'_l=\alpha'_{l+1}=\cdots=\alpha'_i=0.
  \end{equation}
Then the statement will follow from our claim and the induction.

To prove the claim, we notice that from $0=\alpha_{l+1}\in L_{i-l-1}\otimes_S C_{l+1}$ and $(\alpha_j)_{j=0}^i\in \ker\ \partial^{L_\bullet\otimes C_\bullet}_{i}$ we get $$\partial^{L_\bullet\otimes C_{l}}_{i-l}(\alpha_l)=\pm\partial^{L_{i-l-1}\otimes C_\bullet}_{l+1}(\alpha_{l+1})=0.$$ From this and the exactness of $L_\bullet\otimes_S C_{l}$ at $L_{i-l}\otimes C_{l}$ (by our assumption), we get some $\gamma\in L_{i-l+1}\otimes C_{l}$ such that $\partial^{L_\bullet\otimes C_{l}}_{i-l+1}(\gamma)=\alpha_l$. Therefore 
\begin{align*}
	h=[(\alpha_j)_{j=0}^i]=[(\alpha_j)_{j=0}^i]+\partial^{L_\bullet\otimes C_\bullet}_{i+1}(-\gamma)&=[(\alpha_j)_{j=0}^i]-[\big(0,\ldots,0,\pm\partial^{L_{i-l+1}\otimes C_\bullet}_{l}(\gamma),\alpha_l,0,\ldots,0\big)]&\\&=[\big(\alpha_0,\ldots,\alpha_{l-2},\alpha_{l-1}\mp\partial^{L_{i-l+1}\otimes C_\bullet}_{l}(\gamma),0,\ldots,0,\big)]
\end{align*}
can be represented as claimed in display (\ref{EquationHomologyClassVanish}), as was to be proved.
\end{proof}

The next lemma is a new result which will be used in the proof of  Lemma \ref{DeformationOFCEC}(\ref{itm:CETandSpecialization}). Moreover, the following lemma can be used to provide an alternative proof for \cite[Corollary 2.8(i)]{DivaaniAzarMohammadiTavanfarTousi} avoiding the use of   big Cohen-Macaulay modules; even more,   using the next lemma one can deduce that  a non-exact complex $G_\bullet:=0\rightarrow G_s\rightarrow \cdots\rightarrow  G_0\rightarrow 0$ of $(S,\fm_S)$-modules as in \cite[Corollary 2.8(i)]{DivaaniAzarMohammadiTavanfarTousi}  satisfies the analogue of the Improved New Intersection Theorem (which is stronger than the New Intersection Theorem).

\begin{lem}\label{LemmaHyperHomology}
	Let $S$ be a (Noetherian commutative) ring and $$C_\bullet:=0\rightarrow C_s\rightarrow C_{s-1}\rightarrow \cdots \rightarrow C_1\rightarrow C_0\rightarrow 0$$ be a complex of $S$-modules with finitely generated homologies such that $H_0(C_\bullet)\neq 0$. Suppose that $\pd_S(C_i)=\ell$   for each $0\le i\le s$,  where $\ell\ge 1$. Then  there is a projective resolution $\pi_\bullet:P_\bullet\overset{\simeq}{\rightarrow} C_\bullet$ of $C_\bullet$ where
	$$P_\bullet=0\rightarrow P_{t}\rightarrow P_{t-1}\rightarrow \cdots \rightarrow P_1\rightarrow P_0\rightarrow 0$$ with  $t\le s+\ell$  (i.e. $P_\bullet$ has length $\le s+\ell$) and such that $P_\bullet$ consists of finitely generated projective modules.
	\begin{proof}

	 By Remark \ref{RemarkNiceProjectiveResolutionOfComplexes}, there is a projective resolution  $\pi'_\bullet:P'_\bullet\overset{\simeq}{\rightarrow}C_\bullet$ of $C_\bullet$ consisting of finitely generated projective $S$-modules  such that $P'_i=0$ for each $i\lneq 0$.
	 	
	 	 If $P'_\bullet$   already has  length $\le s+\ell$, then  we are done by setting $P_\bullet:=P'_\bullet$ and $\pi_\bullet:=\pi'_\bullet$. So we assume that $P'_\bullet$ is either an infinite complex or it has length $>  s+\ell$.	 	
	 	As $P'_\bullet\overset{}{\simeq} C_\bullet$ and $C_\bullet$ has length $s$ so
	 	\begin{equation}	 	
	 	\label{EquationHighHomolgoiesOfP'Vanish}
	 	\forall\ i\ge s+1,\ \ H_i(P'_\bullet)=0.
	 	\end{equation}
		Let $M$ be  an arbitrary $S$-module. We aim to show that  
		\begin{equation}
		\label{TorGandMForAnyMVanishesAfterlpluss}
		\forall\ i\ge s+\ell+1,\ \ \ H_i(M\otimes_S P'_\bullet)=0.
		\end{equation}		
		Let  $L^{M}_\bullet$ be an  $S$-free resolution of $M$. Then by \cite[(A.4.1) Preservation of Quasi-Isomorphisms and Equivalences]{ChristensenGorenstein},
		$$M\otimes_SP'_\bullet \overset{(- \otimes_S P'_\bullet)\text{\ preserves quasi-iso.}}{\simeq} L^{M}_\bullet\otimes_S P'_\bullet   \overset{(L^M_\bullet \otimes_S -)\text{\ preserves quasi-iso.}}{\simeq} L^{M}_{\bullet} \otimes_S C_\bullet.$$
		Hence display (\ref{TorGandMForAnyMVanishesAfterlpluss}) follows if we can show that  
		\begin{equation}
			 \label{EquationVeryNewVersionOfTHis}
			\forall\  i\ge s+\ell +1,\ \ \  H_i(L^{M}_{\bullet} \otimes_S C_\bullet)=0.
		\end{equation}

		If $q,p\in \mathbb{N}_0$ are such that $i=q+p\ge s+\ell +1$, then either $p\ge s+1$ or $q\ge \ell +1$,  so
        
		\begin{align}
		\label{EquationE1Page}
		H_{q}(L^M_{\bullet}\otimes_SC_p)\cong \tor^S_{q}(M,C_p)\overset{\text{}}{=}0,&& (\text{either\ }C_p=0\   (\text{if\ }p\ge s+1)\ \text{\ or\ }\pd_S(C_p)=\ell< q).  
		\end{align}		        
        Now  if we apply Lemma \ref{LemmaPreliminaryLemma} to the    vanishings in display (\ref{EquationE1Page}) for each $i=p+q\ge s+\ell+1$, then the desired vanishings in display (\ref{EquationVeryNewVersionOfTHis}) and thence in (\ref{TorGandMForAnyMVanishesAfterlpluss}) immediately follow.

	As $H_i(P'_\bullet)=0$ for each $i\ge s+1$ by display (\ref{EquationHighHomolgoiesOfP'Vanish}), so in  the hard truncation  
		$$P'_{\ge s+\ell}:=\cdots\rightarrow P'_{s+\ell+3}\rightarrow P'_{s+\ell+2}\rightarrow P'_{s+\ell+1}\rightarrow P'_{s+\ell}\rightarrow  0,$$ 
		of $P'_\bullet$,  all homologies,  but   $H_{s+\ell}(P'_{\ge s+\ell})=P'_{s+\ell}/\text{im}(\partial^{P'_\bullet}_{s+\ell+1})$, are zero. This together with display (\ref{TorGandMForAnyMVanishesAfterlpluss}) implies that $$\tor^S_1\big(M,P'_{s+\ell}/\text{im}(\partial^{P'_\bullet}_{s+\ell+1})\big)=H_{s+\ell+1}(M\otimes_SP'_\bullet)=0$$ for any $S$-module $M$. Consequently, $P'_{s+\ell}/\text{im}(\partial^{P'_\bullet}_{s+\ell+1})$ is a projective module as  finitely generated flat modules are projective.  Let $P_\bullet$  be  the soft truncation 
		$$ 0\rightarrow P'_{s+\ell}/\text{im}(\partial^{P'_\bullet}_{s+\ell+1})\overset{\overline{\partial^{P'_\bullet}_{s+\ell}}}{\rightarrow} P'_{s+\ell-1} \overset{\partial^{P'_\bullet}_{s+\ell-1}}{\rightarrow} P'_{s+\ell-2}\rightarrow \cdots \rightarrow P'_{0} \rightarrow 0 $$ of $P'_\bullet$, which is  a complex of finitely generated projective modules of length $\le s+\ell$. To define $\pi_\bullet:P_\bullet\rightarrow C_\bullet$, set $\pi_i=\pi'_i$ for each $0\le i< s+\ell$ and set $\pi_{s+\ell}:P'_{s+\ell}/\text{im}(\partial^{P'_\bullet}_{s+\ell+1})\rightarrow C_{s+\ell}(=0)$ to be the zero homomorphism.
		 Then  $\pi_\bullet:P_\bullet\rightarrow C_\bullet$ is a chain map and it is a quasi-isomorphism.
	\end{proof}
\end{lem}

\begin{rem} \emph{We bring to the reader's attention that in the statement of Lemma \ref{LemmaHyperHomology} the complex $C_\bullet$ is not assumed to be a degreewise finite complex and only the finiteness of its homologies is required.}
\end{rem}

The first part of  the following  lemma  will be used in   the last step of  the proof of  Theorem \ref{MonomialTheoremCharacteristicFree}. 
Namely, therein we arrive at a possibly non-local ring $Q$ with a maximal ideal $\fm$ such that $\big(Q/(\mathbf{x}),(\fm)\big)$ is a local ring for some regular sequence $\mathbf{x}\subseteq \fm$. Then Lemma \ref{DeformationOFCEC}(i) deduces the  CE Property for   $Q/(\mathbf{x})$  from the validity of  a similar  CE  Property for  $Q$.  
  Lemma \ref{DeformationOFCEC}(ii)   shows that the CE Property is preserved under deformation, but under a further condition that is (fortunately) available in the situation of the proof of Theorem \ref{MonomialTheoremCharacteristicFree}.

\begin{lem} 	\label{DeformationOFCEC} The following statements hold.
	\begin{enumerate}[(i)]
	\item\label{itm:CETandSpecialization}Let $Q$ be a Noetherian ring and $\fm$ be a maximal ideal  of height $d$. Let $\mathbf{x}\subseteq \fm$ be a regular sequence of $Q$   such that $\big(Q/\mathbf{x}Q,(\fm)\big)$ is a local ring. Let $F_\bullet$ be a  free resolution of $Q/\fm$. Presume that $Q$ satisfies the following CE Property:
	
	For some sequence $\mathbf{z}:=z_1,\ldots,z_d$ ($d=\text{height}(\fm)$) with $\sqrt{(\mathbf{z})}=\fm$ and for some chain map $\phi_\bullet:K_\bullet(\mathbf{z};Q)\rightarrow F_\bullet$ lifting the natural epimorphism $Q/(\mathbf{z})\twoheadrightarrow Q/\fm$, it is the case that $$\forall\  v\in \mathbb{N},\ z_1^{v-1}\ldots z_d^{v-1}\phi_d(1)+\text{im}(\partial^{F_\bullet}_{d+1})\notin (z_1^v,\ldots,z_d^v)\text{syz}_d^{F_\bullet}(Q/\fm).$$
	Then the local ring $Q/\mathbf{x}Q$ satisfies the CE Property.
	\item\label{itm:CETandDeformation} Let $(S,\fm_S)$ be a local ring that is a homomorphic image of a Gorenstein local ring. Suppose that   $x$ is a regular element of $S$ such that $S/xS$ satisfies the CE Property.
	If $$x\notin \bigcup\limits_{\fp\in \text{Att}\big(H^{\dim(S)-1}_{\fm_{S}}(S)\big)}\fp,$$ then  $S$ also satisfies the CE Property.
	\end{enumerate}
	\begin{proof}
		(i)  By virtue of \cite{DutaOnTheCanonical}, it suffices to show that $Q/\mathbf{x}Q$ satisfies the Improved New Intersection Conjecture. So let $$G_\bullet:=0\rightarrow G_s\rightarrow G_{s-1}\rightarrow \cdots\rightarrow G_0\rightarrow 0$$ be a  complex of finite free $(Q/\mathbf{x}Q)$-modules such that $\ell\big(H_i(G_\bullet)\big)<\infty$ for $i>0$ and $H_0(G_\bullet)$ admits a minimal generator $h$ of finite length. 
		We should show that $\dim(Q/\mathbf{x}Q)\le s$ (we assume that $G_s\neq 0$). 
		
		Let $\ell$ be the length of the regular sequence $\mathbf{x}$. So $\ell=\pd_{Q}(G_j)$ for any $0\le j\le s$, as $G_j=(Q/\mathbf{x}Q)^{n_j}$ for some $n_j\in \mathbb{N}_0$. Thus applying Lemma \ref{LemmaHyperHomology}, there exists a complex $$P_\bullet:=0\rightarrow P_{t}\rightarrow P_{t-1}\rightarrow \cdots\rightarrow P_1\rightarrow P_0\rightarrow 0$$    of finitely generated projective $Q$-modules with $t\le \ell+s$  such that it is quasi-isomorphic to $G_\bullet$.   
		Consequently $\fm^uH_i(P_\bullet)=0$ for all $i>0$, and $\fm^uh=0$ for some $u\in \mathbb{N}$. Therefore, considering the sequence $\mathbf{z}$ as in the statement of the lemma, we have $(\mathbf{z})^uH_i(P_\bullet)=(\mathbf{z})^uh=0$ for some $u\in \mathbb{N}$. Thus, arguing  as in the proof of \cite[Theorem (2.6)]{HochsterCanonical} we can construct  a chain map of complexes $$g_\bullet:K_\bullet(\mathbf{z}^v;Q)\rightarrow P_\bullet$$ for some  $v\ge u$  such that $\overline{g_0}:Q/\mathbf{z}^vQ\rightarrow H_0(P_\bullet)$ is given by $1+(\mathbf{z}^v)\mapsto h$.

		Since $h\in H_0(P_\bullet)\cong H_0(G_\bullet)$  is a minimal generator so there is an epimorphism $$\pi:H_0(P_\bullet)\twoheadrightarrow H_0(P_\bullet)/\fm H_0(P_\bullet)\overset{\text{projection}}{\twoheadrightarrow}  Q/\fm,\ \ \pi(h)=1+\fm.$$
		
		As $P_\bullet$ is a complex of projective $Q$-modules and $F_\bullet\rightarrow Q/\fm\rightarrow 0$  is an exact complex ($F_\bullet$ is the free resolution mentioned in the statement) so there is a chain map of complexes  $$ g'_\bullet:P_\bullet \rightarrow F_\bullet, \ \ \overline{g'_0}=\pi$$ where $\overline{g'_0}:H_0(P_\bullet)\rightarrow H_0(F_\bullet)=Q/\fm$  is the induced map by $g'_0$. Consequently, $$\psi_\bullet:=g'_\bullet\circ g_\bullet:K_\bullet(\mathbf{z}^v;Q)\rightarrow F_\bullet$$ is a chain map  lifting the natural surjection $Q/\mathbf{z}^vQ\rightarrow Q/\fm$. Also, we have another chain map $$_v\phi_\bullet:=\phi_\bullet\circ  {_v\lambda}_\bullet:K_\bullet (\mathbf{z}^v;Q)\rightarrow F_\bullet$$ that is lifting the natural surjection $Q/\mathbf{z}^vQ\rightarrow Q/\mathfrak{m}$,   where $\phi_\bullet:K_\bullet(\mathbf{z};Q)\rightarrow F_\bullet$ is the chain map mentioned in the statement, and  $_v\lambda_\bullet:K_\bullet(\mathbf{z}^v;Q)\rightarrow K_\bullet(\mathbf{z};Q)$ is as in  Lemma \ref{LemmaInducedMapOnKoszulHomolgoies}
		. Note that $$_v\phi_d(1)=\phi_d\big({_v\lambda_d}(1)\big)\overset{\text{Lemma \ref{LemmaInducedMapOnKoszulHomolgoies}(\ref{itm:PowerOfX})}}{=} \phi_d(z_1^{v-1}\cdots z_d^{v-1})=z_1^{v-1}\cdots z_d^{v-1}\phi_d(1).$$  Consequently, from Lemma \ref{LemmaBasicFactInHomologicalAlgebra}(\ref{itm:RelationInSyz}) we get
		
		$\psi_d(1)-z_1^{v-1}\cdots z_d^{v-1}{\phi_d}(1)+\text{im}(\partial^{F_\bullet}_{d+1})=\psi_d(1)-{_v\phi_d}(1)+\text{im}(\partial^{F_\bullet}_{d+1})\in (z^v_1,\ldots,z^v_d)\text{syz}_d^{F_\bullet}(Q/\fm)$.
		
		This, in conjunction with the  CE  Property assumed in the statement implies that $\psi_d=g'_d\circ g_d\neq 0$, whence $g_d\neq 0$ in particular. Consequently, $t \ge d$, otherwise the target of $g_d:K_{d}(\mathbf{z}^v;Q)\rightarrow P_d$ is the zero module (we recall that  the complex $P_\bullet$ has length    $t$, where $t\le s+\ell$). So $\ell+s\ge d$. Therefore,
		$$\dim({Q}/\mathbf{x}{Q})=\dim ({Q}_{\mathfrak{m}}/\mathbf{x}{Q}_{\mathfrak{m}})=\dim ({Q}_{\mathfrak{m}})-\ell = \text{height}(\fm)-\ell= d-\ell\le s,$$ as was to be proved.
		
		(ii) Suppose that ${S}$ is a homomorphic image of a  Gorenstein local ring $G$, say $S=G/\mathfrak{g}$ where $\mathfrak{g}$ is an ideal of $G$ of height $l$. From our hypothesis in conjunction with \cite[Theorem (3.15)]{HochsterCanonical}  we get $\eta_{{S}/x{S}}\neq 0$.  We are done if we can show that $\eta_{{S}}\neq 0$ (again by  \cite[Theorem (3.15)]{HochsterCanonical}). Since $l=\dim(G)-\dim(S)$ in view of \cite[Corollary 2.1.4]{BrunsHerzogCohenMacaulay}, so by virtue of 	 \cite[Theorem (5.3)]{HochsterCanonical} the non-vanishing of $\eta_{S}$ follows from the vanishing  of the module $$C=0:_{\Ext_G^{l+1}({S},G)}x$$ in the statement of \cite[Theorem (5.3)]{HochsterCanonical}, 
		where $C$  is introduced in \cite[Lemma (5.1) and its preceding paragraph]{HochsterCanonical}. 
		
		By  \cite[11.2.6 Local Duality Theorem]{BrodmannSharpLocal}), $\text{Hom}_{S}\big(\Ext_G^{l+1}({S},G),\text{E}({S}/\fm_{S})\big)\cong H^{\dim(S)-1}_{\fm_{S}}({S})$ which implies that 
		\begin{align*}
		\text{Hom}_{S}\big(C,\text{E}({S}/\fm_{S})\big)&\cong \text{Hom}_{S}\Big(\text{Hom}_{S}\big({S}/x{S},\Ext_G^{l+1}({S},G)\big),\text{E}({S}/\fm_{S})\Big)
		&\\&\cong 
		\text{Hom}_{{S}}\big(\Ext_G^{l+1}({S},G),\text{E}({S}/\fm_{S})\big)\otimes_{S}({S}/x{S}) && (\text{\cite[Lemma 9.71]{Rotman}})		
		&\\&\cong 
		H^{\dim(S)-1}_{\fm_{S}}({S})/xH^{\dim(S)-1}_{\fm_{S}}({S})
		&\\& 
		=0 && \text{(our hypo. and Lemma \ref{LemmaAttachedPrime}(\ref{itm:AttAndSurjectivity}))}.
		\end{align*}
		 Consequently $C=0$, as required.
	\end{proof}
\end{lem}


\begin{lem}
	\label{RingHomomorphismANDCEC}
	Let $\sigma:({S},\fm_{S})\rightarrow (T,\fm_T)$ be a local homomorphism of local rings such that the image  of a system of parameters  for $S$ is   a system of parameters for $T$. If $T$ satisfies the CE Property, then so does ${S}$.
\end{lem}
\begin{proof}
	See	\cite[Proposition (3.18)(a)]{HochsterCanonical} and \cite[Theorem (3.15)]{HochsterCanonical}.
\end{proof}


\section{Extended Rees algebra of prime almost complete intersection ideals}

This section  is  devoted to proving   Proposition \ref{StructureProposition}  on some nice properties of the extended Rees algebra $A[\fa t,t^{-1}]$, where $\fa$ is a prime almost complete intersection ideal of a regular local ring $(A,\fm_A)$. 
The following lemmas  are required for  	 Proposition \ref{StructureProposition}. 

\begin{lem}
	\label{GoodRegularSequenceCorollary}
	Let $(A,\mathfrak{m}_A)$ be a regular local ring and  $\fb:=\fc:h,\ \fa:=(\fc,h)$ be linked ideals over the complete intersection ideal $\fc:=(a_1,\ldots,a_g)$ such that $\fb$ is a quasi-Gorenstein prime ideal, and $\fa$ is a prime almost complete intersection ideal.  Assume that $A/\fb$ is  not Cohen-Macaulay. Then
	\begin{enumerate}[(i)]
	\item\label{itm:DepthIdentity} $\depth(A/\fa)=\depth(A/\fb)-1$.
	\item\label{itm:DSequencePrimeACI} $a_1,\ldots,a_g,h$ is a $d$-sequence.
	\item\label{itm:MaximalRegularSequenceAndThreeQuotientRings}	there exists a regular sequence $\mathbf{x}$ in $A$ whose image in $A/\fc$ is a regular sequence, whose image in  $A/\fa$ is  a maximal regular sequence,  and such that the image of $\mathbf{x},h$ in $A/\fb$ is a maximal regular sequence.
	\end{enumerate}
\begin{proof}
	(\ref{itm:DepthIdentity}) From the exact sequence $0\rightarrow A/\fb\overset{h}{\rightarrow }A/\fc\rightarrow A/\fa\rightarrow 0$, we conclude that $\depth(A/\fa)=\depth(A/\fb)-1$. Since $A/\fc$ is a complete intersection (thus Cohen-Macaulay) and $A/\fb$ is non-Cohen-Macaulay, so the statement follows easily from the long exact sequence obtained by $\Gamma_{\fm_A}(-)$ and the above exact sequence.

	(\ref{itm:DSequencePrimeACI})	We only need to show that $\fc:h^i=\fc:h$ for all $i\ge 2$, as we already know that $a_1,\ldots,a_g$ is a regular sequence. But  $\fc:h$ is a codimension $g$ prime quasi-Gorenstein ideal by our hypothesis while $\fc:h^i$   contains $\fc:h$,  so we are done once we can prove that $\fc:h^i$  also has codimension $g$. If $\text{height}\big(\fc:h^i\big)\ge g+1$  then $\fc:h^i\nsubseteq \bigcup\limits_{\mathfrak{p}\in \as(\fc)}\fp$, because $\mathfrak{c}$ is a complete intersection thus an unmixed ideal. It follows that then there exists some $x\in A$ such that  $(a_1,\ldots,a_g,x)$ is a regular sequence while $x\in \fc:h^i$. But then $h^i\in \fc$ implying that $\fa^i\subseteq \fc\subseteq \fb$, whence $\fa=\fb$ ($\fa$ is also a codimension $g$ prime ideal). But this is impossible because quasi-Gorenstein ideals are never almost complete intersection by  \cite[Proposition 1.1]{KunzAlmost}.

(\ref{itm:MaximalRegularSequenceAndThreeQuotientRings})  The element $h$ is regular on $A/\fb$ and since $\depth A/\fa =\depth A/(\fb,h)<\depth A/\fc <\depth A$,  the statement is a  simple result of the prime avoidance lemma (similar to what will be done in Step 4 of the proof of Theorem \ref{MonomialTheoremCharacteristicFree}).
\end{proof}
\end{lem}	  

The goal of the next lemma is to illustrate our situation in the last step of the proof of Theorem  \ref{MonomialTheoremCharacteristicFree}. 

\begin{lem}\label{LemmaEquationOfSymmetricAlgebra} 	Suppose that $\fa=(h,a_1,\ldots,a_g)$ is an ideal of a ring $A$.
 \begin{enumerate}[(i)]
 \item\label{itm:SymmetricAlgebraPresentation}  Let $A[Z_0,\ldots,Z_g]$ be the standard polynomial ring over $A$.
  The symmetric algebra  of $\fa$, denoted by $\text{Sym}_A(\fa)$,  as a graded $R$-algebra is isomorphic to  the quotient ring $$A[Z_0,\ldots,Z_g]/\mathfrak{d},\ \ \text{where}\  \mathfrak{d}=<\{r_0Z_0+\sum\limits_{i=1}^gr_iZ_i:r_0,\ldots,r_g\in A,\ hr_0+\sum\limits_{i=1}^gr_ia_i=0\}>.$$

 \item\label{itm:ColonIdealAsaSpecilizationOfTheSymmetricAlgebra}   We denote the sequence of   degree $1$ homogeneous elements 
 $h,a_1,\ldots,a_g\in \fa=\text{Sym}_A(\fa)_{[1]}$	by $ht,a_1t,\ldots,a_gt$.  Then 
 $$A/\big((a_1,\ldots,a_g):h\big)\cong
		\big({\text{Sym}_{A}(\fa)_{ht}}\big)_{[0]}/(a_1t/ht,\ldots,a_gt/ht).$$
 \end{enumerate}
 \begin{proof}
 	(\ref{itm:SymmetricAlgebraPresentation}) This is a folklore fact in the literature. The statement is, e.g., written in \cite[page 270]{HunekeSymmetric} without proof.
 	
 	(\ref{itm:ColonIdealAsaSpecilizationOfTheSymmetricAlgebra}) 	The dehomogenization with respect to $Z_0$
 	\begin{equation}
 	\label{EquationDehomogenization}
 	A[Z_0,\ldots,Z_g]/(\mathfrak{d},Z_0-1)\rightarrow \big((A[Z_0,\ldots,Z_g]/\mathfrak{d})_{Z_0}\big)_{[0]},\text{\ \ (induced by)\ } Z_i\mapsto (Z_i+\mathfrak{d})/Z_0,\ r\in A\mapsto (r+\mathfrak{d})/1
 	\end{equation}
 	  is an isomorphism. For the injectivity of this ring homomorphism, suppose that $\sum_{i=1}^tf_i+(\mathfrak{d},Z_0-1)$ belongs to the kernel where each $f_i$ is a non-zero monomial   of degree $n_i$. Then,   there exists $m\in \mathbb{N}_0$ with $\sum_{i=1}^t(Z_0^{m-n_i}f_i)\in \mathfrak{d}$. Since $Z_0$ is treated as $1$ in $A[Z_0,\ldots,Z_g]/(\mathfrak{d},Z_0-1)$ so the injectivity follows.
 	  
 	   The isomorphism (\ref{EquationDehomogenization}) and part (\ref{itm:SymmetricAlgebraPresentation}) yield  $$A[Z_0,\ldots,Z_g]/(\mathfrak{d},Z_0-1)\cong \big({\text{Sym}_{A}(\fa)_{ht}}\big)_{[0]},\text{\ \ (induced by)\ } Z_i\mapsto a_it/ht,\ Z_0\rightarrow 1,\ r\in A\mapsto r.$$ 
 	Consequently, 
  	\begin{align*}
 	A/\big((a_1,\ldots,a_g):h\big)&\overset{}{\cong}  A[Z_0,\ldots,Z_g]/(\mathfrak{d},Z_0-1,Z_1,\ldots,Z_g) && (\text{by the definition of}\ \mathfrak{d})
 	&\\&\cong  
 	\big({\text{Sym}_{A}(\fa)_{ht}}\big)_{[0]}/(a_1t/ht,\ldots,a_gt/ht)  && (\text{the previous display}).\qedhere
 \end{align*}
 \end{proof}
\end{lem}

\begin{defi}
	\emph{We say that a ring $S$ is a \textit{locally complete intersection at codimension $\le 1$} precisely when $S_\mathfrak{p}$ is a complete intersection ring for each prime ideal $\fp$ of $S$ with $\text{height}(\fp)\le 1$.}
\end{defi}

The next lemma will be used in the proof of Proposition  \ref{StructureProposition}(\ref{itm:HomogeneousMaximalRegularSequenceOfEReesAlgebra}).

\begin{lem}(due to Raymond Heitmann)\label{LemmaHeitmann} Let $A$ be a ring and suppose that $a_1,\ldots,a_g$ is a regular sequence of $A$. Let $\mathfrak{A}:=(a_1,\ldots,a_g,h)A$ and suppose that $\big((a_1,\ldots,a_g):h^i\big)=\big((a_1,\ldots,a_g):h\big)$	for all $i\ge 2$.  Then for every $j$, $1\le j\le 	g$ and every positive integer $n$, $(a_1,\ldots,a_j)A\cap \mathfrak{A}^n=(a_1,\ldots,a_j)\mathfrak{A}^{n-1}$.
	\begin{proof}
		We first prove the $j=g$ case for all $n$. Obviously, $\mathfrak{A}^n=(a_1,\ldots,a_g)\mathfrak{A}^{n-1}+h^nA$ and so it suffices to  show that whenever $h^nw\in (a_1,\ldots,a_g)$, $h^nw\in (a_1,\ldots,a_g)\mathfrak{A}^{n-1}$. However, by the hypothesis, $hw\in (a_1,\ldots,a_g)$ and so $h^nw\in h^{n-1}(a_1,\ldots,a_g)\subseteq (a_1,\ldots,a_g)\mathfrak{A}^{n-1}$.
		
		Now, we handle the $j<g$ case by induction on $n$. The $n=1$ case holds trivially. So let $n\ge 2$. Suppose that $w\in (a_1,\ldots,a_j)\cap \mathfrak{A}^n$. By the $j=g$ case, we have $w\in (a_1,\ldots,a_g)\mathfrak{A}^{n-1}$. We can choose $u\le g$ minimal so that $w\in (a_1,\ldots,a_u)\mathfrak{A}^{n-1}$. We finish  the proof by induction on $u$. If $u\le j$, we are done. Otherwise, we write $w=w'+a_u\alpha$ with $w'\in (a_1,\ldots,a_{u-1})\mathfrak{A}^{n-1}$ and $\alpha\in \mathfrak{A}^{n-1}$. Since $a_1,\ldots,a_u$ is a regular sequence, $\alpha\in (a_1,\ldots,a_{u-1})\cap\mathfrak{A}^{n-1}$. By the induction assumption on $n$, we have   $\alpha\in (a_1,\ldots,a_{u-1})\mathfrak{A}^{n-2}$  and so $a_u\alpha\in (a_1,\ldots,a_{u-1})\mathfrak{A}^{n-1}$. By induction on $u$, the proof is complete.
	\end{proof}
\end{lem}

\begin{rem}\emph{
		The preceding lemma  and Proposition \ref{StructureProposition}(\ref{itm:CIQuotientOfEReesAlgebraIsNice}) are pointed out to the author by Raymond Heitmann and  the referee  to simplify dramatically  the (previous) author's proof of Proposition \ref{StructureProposition}(\ref{itm:HomogeneousMaximalRegularSequenceOfEReesAlgebra}) and (\ref{itm:SpecializationOfEReesAlgebraIsS2}) where the author had exploited Herzog-Simis-Vasconcelos $\mathcal{M}$-approximation complexes through a spectral sequence argument.  We are thankful to Raymond Heitmann and the referee for their favor. 
	}
\end{rem}

	
The extended Rees algebra $A[\fa t,t^{-1}]$ in the next proposition is a factorial domain. This factoriality follows from the hypothesis that $A/\fa$ is a locally complete intersection in codimension $\le 1$,  by virtue of \cite[Theorem 1]{HochsterCriteria} (see parts F and A) and \cite[Theorem 2.2]{HunekeAlmost}.  

\begin{prop}
	\label{StructureProposition} 
	Let $(A,\fm_A)$ be a regular local ring of dimension $d$ and $\mathfrak{a}:=(a_1,\ldots,a_g,h)$ be a prime almost complete intersection ideal of $A$, containing the codimension $g$ complete intersection ideal $\mathfrak{c}:=(a_1,\ldots,a_g)$. Suppose further that the quasi-Gorenstein ideal $\mathfrak{b}:=(a_1,\ldots,a_g):h$, which is linked to $\mathfrak{a}$, is a prime ideal such that $A/\fb$ is not Cohen-Macaulay. Denote by  $\mathfrak{M}$  the unique homogeneous maximal ideal of $A[\mathfrak{a}t,t^{-1}]$.   To force $A[\fa t,t^{-1}]$ to be a factorial domain, we assume that $A/\fa$ is  a locally complete intersection at codimension $\le 1$. Then the factorial extended Rees algebra $A[\mathfrak{a} t,t^{-1}]$ satisfies the following properties.
	\begin{enumerate}[(i)]
		\item\label{itm:BeingPrimeIdealHeightd}   $\mathfrak{P}:=(\fm_A,a_1t,\ldots,a_gt,t^{-1})$ is a height $d$ prime ideal.
		\item\label{itm:CIQuotientOfEReesAlgebraIsNice} Let $(A/\fc)[(h)s,s^{-1}]$ be the extended Rees algebra of   $h(A/\fc)$. Then $$(A/\fc)[(h)s,s^{-1}]\cong A[\fa t,t^{-1}]/(a_1t,\ldots,a_gt).$$
			\item\label{itm:QuotientOfEReesAlgebraIsACI}
		$(A/\mathfrak{a})[X]\cong A[\mathfrak{a} t,t^{-1}]/(a_1t,\ldots,a_gt,ht)$ \footnotemark{}.\\
		\addtocounter{footnote}{-1}
		\stepcounter{footnote}
		\footnotetext{We stress that the ideal $(a_1t,\ldots,a_gt,ht)$ is not a complete intersection, but it is generated by $g+1$ elements while it contains the length $g$ complete intersection $(a_1t,\ldots,a_gt)$.}
		
		\item\label{itm:HomogeneousMaximalRegularSequenceOfEReesAlgebra}
		Let $\mathbf{x}:=x_1,\ldots,x_u\subseteq A$ be the regular sequence of $A$ obtained in Lemma  \ref{GoodRegularSequenceCorollary}(\ref{itm:MaximalRegularSequenceAndThreeQuotientRings}), that  forms a maximal regular sequence on the almost complete intersection quotient domain $A/\mathfrak{a}$ and a regular sequence on the rings $A/\mathfrak{b},A/\fc$. Then
		$$
		t^{-1},a_1t,\ldots,a_gt,x_1,\ldots,x_u
		$$
		(and any of its permutation) is a regular sequence on $A[\mathfrak{a} t,t^{-1}]$. It is maximal with respect to the property of consisting of only homogeneous elements.

		\item\label{itm:SpecializationOfEReesAlgebraIsS2} For each $0\le i\le g-1$, $$a_{i+1}t\notin \bigcup\limits_{\map \in\text{Att}\Big({H^{d-i}_{\mathfrak{M}}\big({A[\fa t,t^{-1}]}_{\mathfrak{M}}/(a_1t,\ldots,a_{i}t)\big)\Big)}}\map.$$
	\end{enumerate}

\begin{proof} 
	(\ref{itm:BeingPrimeIdealHeightd})  We have a homogeneous  ring epimorphism $\pi:(A/\fm_A)[X]\overset{X\mapsto ht}{\longrightarrow} A[\fa t,t^{-1}]/\mathfrak{P}$. Since $(A/\fm_A)[X]$ is a principal ideal domain and  $\pi$ is homogeneous so either $\pi$ is an isomorphism or $\ker(\pi)=(X^i)$ for some $i\in \mathbb{N}$. We show that the latter is impossible. If it were possible, then $(ht)^i\in \mathfrak{P}=(a_1t,\ldots,a_gt,\fm_A,t^{-1})$ which implies that $$h^i\in \fm_A \fa^i+\sum\limits_{j=1}^ga_j\fa^{i-1}+\fa^{i+1}\subseteq (\fc,\fm_Ah^i,h^{i+1})\subseteq (\fc,\fm_Ah^i).$$ It follows that $(1+m)h^i\in \fc$ for some $m\in \fm_A$, i.e. $h^i\in \fc$. But this contradicts  the fact that $$\fc:h^i\overset{\text{Lemma  \ref{GoodRegularSequenceCorollary}(\ref{itm:DSequencePrimeACI})}}{=}\fc:h=\fb$$ is our (quasi-Gorenstein) prime ideal.
	
	(\ref{itm:CIQuotientOfEReesAlgebraIsNice}) It is obvious that $\big(A[\fa t,t^{-1}]/(a_1t,\ldots,a_gt)\big)_{[0]}=A/\fc$. Set  $\fa^i=A$ for any $i\le 0$. The identities 
	$$\forall\ i\in \mathbb{Z},\ \  \big(A[\fa t,t^{-1}]/(a_1t,\ldots,a_gt)\big)_{[i]}=\fa^i/\fc\fa^{i-1}\overset{\text{Lemma \ref{LemmaHeitmann}}}{=}\fa^{i}/(\fc\cap \fa^{i}),$$
	provide us with the family of isomorphisms $$\varphi_i:\big(A[\fa t,t^{-1}]/(a_1t,\ldots,a_gt)\big)_{[i]}\rightarrow h^i(A/\fc),\ \ \ h^i+(\fc\fa^{i-1})\mapsto h^i+\fc.$$
	Evidently, the maps $\varphi_i$ yield a homogeneous ring isomorphism $A[\fa t,t^{-1}]/(a_1t,\ldots,a_gt)\rightarrow (A/\fc)[(h)s,s^{-1}]$ as claimed in the statement.
	
	
	
		(\ref{itm:QuotientOfEReesAlgebraIsACI}) This isomorphism  is easy, because the natural ring homomorphism $A \rightarrow A[\fa t,t^{-1}]/(a_1t,\ldots,a_gt,ht)$ induces a homogeneous ring epimorphism $$(A/\mathfrak{a})[X]\rightarrow A[\mathfrak{a} t,t^{-1}]/(a_1t,\ldots,a_gt,ht),\ \ X\mapsto t^{-1}$$ and to see that this is injective, it suffices to note that each homogeneous element in the kernel is zero.

  (\ref{itm:HomogeneousMaximalRegularSequenceOfEReesAlgebra}) 	 
  Applying  Lemma \ref{LemmaHeitmann}, it is straightforward to  observe that $t^{-1},a_1t,\ldots,a_gt$ is a regular sequence on $A[\fa t,t^{-1}]$. More immediately, in light of Lemma \ref{LemmaHeitmann} and \cite[Proposition 3.5]{HuckabaMarleyHilbert}  the sequence $a_1t,\ldots,a_gt$ forms a regular sequence on  $\text{gr}_{\fa}(A)=A[\fa t,t^{-1}]/(t^{-1})$. Hence, we should show that $\mathbf{x}$ forms a regular sequence on $$A[\fa t,t^{-1}]/(t^{-1},a_1t,\ldots, a_gt)=(A/\fc)[(h)s,s^{-1}]/(s^{-1})=\text{gr}_{h(A/\fc)}(A/\fc).$$

  But for each $i\ge 1$ there is an isomorphism $A/(\fb,h)\overset{\ 1\mapsto h^i}{\longrightarrow}(h^i,\fc)/(h^{i+1},\fc)$ (by Lemma  \ref{GoodRegularSequenceCorollary}(\ref{itm:DSequencePrimeACI})). Therefore 
  \begin{equation}
  	\label{EquationAGrRingIsAaDSumAbAndh}
  	\text{gr}_{h(A/\fc)}(A/\fc)=A/\fa\oplus \big(\oplus_{i=1}^\infty A/(\fb,h)\big)
  \end{equation}
   as  $A/\fc$-modules,  so the sequence  $\mathbf{x}$ is  regular on $\text{gr}_{A/\fc}\big(h(A/\fc)\big)$ as needed.

	It is proved that the desired sequence is regular on $A[\fa t,t^{-1}]$.	 To see that this is a maximal homogeneous regular sequence, note that the quotient of $A[\mathfrak{a} t,t^{-1}]$ by this regular sequence is generated by only one element of degree $1$ that is the residue class of $ht$. So if we can extend this regular sequence by adding a homogeneous element, the new element would be either some power of $ht$ or an element of degree zero. The latter is impossible, because the degree zero subring of $A[\fa t,t^{-1}]/(t^{-1},a_1t,\ldots,a_gt,x_1,\ldots,x_u)$ is the almost complete intersection $A/\big(\fa+(\mathbf{x})\big)$ which has depth zero by our choice of $\mathbf{x}$. The former case is also impossible, otherwise $t^{-1},a_1t,\ldots,a_gt,ht$ would be a regular sequence of $A[\fa t,t^{-1}]$ which contradicts with $A[\fa t,t^{-1}]/(t^{-1},a_1t,\ldots,a_gt,ht)=A/\fa$ (part (\ref{itm:QuotientOfEReesAlgebraIsACI})), because then $A[\fa t,t^{-1}]_{\mathfrak{M}}$,  as a deformation of an almost complete intersection, would be an almost complete intersection by \cite[Theorem 2.3.4(a)]{BrunsHerzogCohenMacaulay} and \cite[Theorem 21.1(iii)]{Matsumura}. However, almost complete intersections are not quasi-Gorenstein (\cite[Proposition 1.1]{KunzAlmost}) while $A[\fa t,t^{-1}]_{\mathfrak{M}}$ is a factorial domain that is a homomorphic image of a regular local ring  and so it is quasi-Gorenstein (see e.g. \cite[Lemma (2.4)]{FossumFoxbyEtAlMinimalInjective}).
	
	(\ref{itm:SpecializationOfEReesAlgebraIsS2}) First we show that $Q:=A[\fa t,t^{-1}]/(a_1t,\ldots,a_gt)$ satisfies the $(S_2)$-condition. Let $\fp\in \text{Spec}(A[\fa t,t^{-1}])$  such that it contains $(a_1t,\ldots,a_gt)$. 
	 If $t^{-1}\notin \fp$, then $A[\fa t,t^{-1}]$ is regular at $\fp$ due to the regularity of $A[\fa t,t^{-1}]_{t^{-1}} \cong A[t,t^{-1}]$, whence $Q$ is a complete intersection at $\mathfrak{p}$ by part (iii).

	 If $t^{-1}\in \fp$,  then we are done once we can show that
	$$
	G:=\text{gr}_{\fa}(A)/\big(a_1+\fa^2,\dots,a_g+\fa^2\big)\cong Q/(t^{-1})
	$$
	satisfies the $(S_1)$-condition (again  $t^{-1}$ is a regular element of $Q$ in view of part (\ref{itm:HomogeneousMaximalRegularSequenceOfEReesAlgebra})), because then $\depth(G_{\fp})\ge \min\{\dim(G_{\fp}),1\}$ implying that $\depth(Q_{\fp})\ge\min\{\dim(Q_{\fp}),2\}$, as required.  
	In the rest of the proof, by abuse of notation, we consider  $\fp$  as a prime ideal of $G$. We also denote the image in $G$, or in any quotient of $A$, of any  element  $r\in A$  by the same notation $r$.

	 If $ht\notin\fp$, then by part (\ref{itm:CIQuotientOfEReesAlgebraIsNice})	  (and Lemma \ref{GoodRegularSequenceCorollary}(\ref{itm:DSequencePrimeACI}) for the second isomorphism)
	 \begin{align*}
	 	Q_{ht}
	 	&\overset{}{\cong} \Big((A/\fc)[(h)s,s^{-1}]\Big)_{hs}
	 	&\\&
	 	\cong \Big((A/\fc)[X,Y]/\big(XY-h,(0:_{A/\fc}h)X\big)\Big)_{X}, &&  (X\mapsto hs,\ Y\mapsto s^{-1} )
	 	&\\&
	 	\cong \Big((A/\fc)[X,Y]/\big(XY-h,0:_{A/\fc}h\big)\Big)_{X}
	 	&\\&\cong 
	 	\Big((A/\fb)[X,Y]/(XY-h)\Big)_{X}.
	 \end{align*}
	 Thus $$G_{ht}=\big(Q/(t^{-1})\big)_{ht}\cong \Big((A/\fb)[X,Y]/(XY-h,Y)\Big)_{X}\cong \Big(\big(A/(\fb,h)\big)[X]\Big)_{X}$$ which is indeed $(S_1)$ because $A/\fb$ is an $(S_2)$-domain.
	 
	 It remains only to discuss   the case where $t^{-1},ht\in \mathfrak{p}$.  Let $\fq:=\fp\cap A/\fa=\fp\cap G_{[0]}$. We can assume furthermore that $\fa\subsetneq \fq$, because  otherwise   $\text{height}_G(\fp)=0$ necessarily and then there is nothing to prove. Let $r\in \fq\backslash \fa$. If $r$ does not belong to any associated  prime (or equivalently minimal prime) of $(\fb,h)$, then  $r$ is regular on $G=\text{gr}_{h(A/\fc)}(A/\fc)$ in view of display (\ref{EquationAGrRingIsAaDSumAbAndh}) and thus $\fp$ contains a regular element as needed. If otherwise,   
	 there exists some $s\in A\backslash (\fb,h)$ such that $sr\in (\fb,h)$. Note that $sr\notin \fa$ because $\fa$ is a prime ideal. Then, there is some $\beta\in \fb\backslash \fa$ such that $sr+\fa=\beta+\fa$, thus $\beta\in (\fq\cap \fb)\backslash \fa$. Consequently, $\beta+ht\in \fp$ (here the image of $ht$ in $G=Q/(t^{-1})$ is denoted by the same notation $ht$). Then it is readily seen that $\beta +ht\in \fp$ is a regular element of $G$ (because $\fa$ is a prime ideal, $0:_Ght=0:_G(ht)^2$ as can easily be verified and $\beta(ht)=0$ in $G$).

	  It follows that $Q=A[\fa t,t^{-1}]/(a_1t,\ldots,a_gt)$ satisfies the $(S_2)$-condition, as was to be proved. Hence in  light of  \cite[Corollaire (5.12.4)]{GrothendieckEGAIVII}, each $A[\fa t,t^{-1}]_{\mathfrak{M}}/(a_1t,\ldots,a_it)$ satisfies $(S_2)$-condition ($1\le i \le g$). So our statement follows from the equivalent conditions of \cite[Corollary 2.8]{TavanfarTousiAStudy} (again we are using the fact that $A[\fa t,t^{-1}]_{\mathfrak{M}}$ is a factorial domain and quasi-Gorenstein).	 	 
\end{proof}
	
	
\end{prop}

\section{A  new variant of the Canonical Element Theorem}\label{SectionNewVariant}

\subsection{Various aspects of our new variant of the Canonical Element Theorem}

We begin this section by  recalling the definition of the canonical element  with respect to a primary ideal to the maximal ideal.

\begin{defi} \label{DefinitionCanonicalElement}
	\emph{(see \cite[Definition 3.2]{HochsterCanonical}) Let $(S,\fm_S)$ be a $d$-dimensional local ring, $\fq$ be an $\fm_S$-primary ideal and   $F_\bullet$ be a free resolution of $S/\fq$. Consider the natural epimorphism $\pi:F_{d}\rightarrow \text{syz}^{F_\bullet}_d(S/\fq)$  defining the element $$\epsilon_{\fq}:=[\pi]\in \ext^{d}_S\big(S/\fq,\text{syz}^{F_\bullet}_{d}(S/\fq)\big).$$
		Let $\eta_\fq$  be the image of $\epsilon_{\fq}$ under the natural map  to the local cohomology $$\Theta^{\text{syz}^{F_\bullet}_{d}(S/\fq)}_\fq:\ext^{d}_{S}\big(S/\fq,\text{syz}_{d}^{F_\bullet}(S/\fq)\big)\rightarrow \lim\limits_{\overset{\longrightarrow}{m\in \mn}}\ext^{d}_{S}\big(S/\fq^m,\text{syz}_{d}^{F_\bullet}(S/\fq)\big)
		.$$  Then $\eta_\fq$ is said to be the \textit{canonical element} of $S$ with respect to $\fq$. 	The element $\eta_{\fm_S}$ is called  the  \textit{canonical element} of $S$ and it is denoted, also, by $\eta_S$. We  recall that for any $s\in S$, the vanishing/non-vanishing of $s.\epsilon_\fq$ and $s.\eta_\fq$ is independent of the chosen resolution $F_\bullet$ (\cite[Proposition  (3.3)(c)]{HochsterCanonical}).}
\end{defi}

	 \begin{rem}\label{RemarkMappingConeFreeResolution} \emph{Let $\fp$ be  a prime ideal of a  ring $S$ and      $y\in S\backslash \fp$ be a non-unit. If $F_\bullet$ is the minimal free resolution of $S/\fp$, then the mapping cone $${_wG_\bullet}:=\text{cone}(y^w.\id_{F_\bullet}:F_\bullet \overset{}{\rightarrow} F_\bullet)$$ provides us with the minimal free resolution of $S/(\fp,y^w)$. This fact can be observed via  the exact sequence of complexes
		\begin{equation} 
		\label{EquationMappingConeExactSequence}
		0\rightarrow F_\bullet \rightarrow {{_wG_\bullet}}\rightarrow F_\bullet[-1]\rightarrow 0
		\end{equation}  
		inducing the homology long exact sequence
			\begin{equation*}
			\label{EquationExactSequencesOnHomologies}
			\cdots\rightarrow	H_i(F_\bullet)\overset{y^w}{\rightarrow} H_i(F_\bullet)\rightarrow H_{i}({_wG_\bullet})\rightarrow H_{i-1}(F_\bullet)\overset{y^w}{\rightarrow}H_{i-1}(F_\bullet) \rightarrow \cdots
			\end{equation*}
			as $y^{w}$ forms a non-zero divisor over the integral domain $S/\fp$ (see  \cite[Lemma 10.38]{Rotman} for a reference to these exact sequences).}
  \end{rem}

In the statement of the following lemma, all syzygies are obtained from the minimal free resolutions. So when we are denoting a syzygy in the lemma, we will omit the particular free resolution by which the syzygy is computed.

\begin{lem}\label{LemmaSysyzgiesRelations}  Let $\fp$ be a prime ideal of a  ring $S$, $y\in S\backslash \fp$ be a non-unit and  $i,s,t\in \mathbb{N}$ with $t\ge s$ ($s$ and $t$ stand for powers of $y$, while $i$ stands for the $i$-th syzygy). Consider the minimal free resolutions ${_sG_\bullet}$  and ${_tG_\bullet}$ of $S/(\fp,y^s)$ and $S/(\fp,y^t)$ respectively, as mentioned in Remark  \ref{RemarkMappingConeFreeResolution}.
	\begin{enumerate}[(i)]
		\item \label{itm:SyzygiesMapsBetweenThem} There are chain maps $${_{t,s}\lambda_\bullet}:{_tG_\bullet}\rightarrow {_sG_\bullet}, \ \ \  {_{s,t}\lambda_\bullet}:{_sG_\bullet}\rightarrow {_tG_\bullet}$$ inducing $S$-homomorphisms respectively $$f_{t,s}:\text{syz}_{i}\big(S/(\fp,y^t)\big)\overset{}{\rightarrow} \text{syz}_{i}(S/(\fp,y^s)),\ \ \ f_{s,t}:\text{syz}_{i}\big(S/(\fp,y^s)\big)\overset{}{\rightarrow} \text{syz}_{i}\big(S/(\fp,y^t)\big)$$ such that $$f_{s,t}\circ f_{t,s}=y^{t-s}\id_{\text{syz}_{i}\big(S/(\fp,y^t)\big)},\ \ \ {_{s,t}\lambda_\bullet}\circ {_{t,s}\lambda_{\bullet}}=y^{t-s}\id_{{_{t}G_\bullet}},\ \ \ {_{t,s}\lambda_0}=\id_S.$$
		\item \label{itm:SyzygyWtoSyzygy1ExactSequence}  If $y$ is a regular element of $S$, then there is an exact sequence $$0\rightarrow \text{syz}_{i}\big(S/(\fp,y^t)\big)\overset{f_{t,s}}{\rightarrow} \text{syz}_{i}\big(S/(\fp,y^s)\big)\rightarrow \text{syz}_{i-1}(S/\fp)/y^{t-s}\text{syz}_{i-1}(S/\fp)\rightarrow 0.$$
		\item \label{itm:Syzygy1toSyzygyW} If $y$ is a regular element of $S$, then there is an exact sequence $$0\rightarrow \text{syz}_{i}\big(S/(\fp,y^s)\big)\overset{f_{s,t}}{\rightarrow} \text{syz}_{i}\big(S/(\fp,y^t)\big)\rightarrow \text{syz}_{i}(S/\fp)/y^{t-s}\text{syz}_{i}(S/\fp)\rightarrow 0.$$
	\end{enumerate}
	\begin{proof}
		 In view of the rules of the differentials of  ${_sG_\bullet}$ and ${_tG_\bullet}$, there is a chain map  ${_{t,s}\lambda_\bullet}:{_tG_\bullet}\rightarrow {_sG_\bullet}$ given by $$\begin{pmatrix} \id_{F_j} &  0 \\ 0 & y^{t-s}\id_{F_{j-1}}  \end{pmatrix}:\underset{{_tG_j}}{\underbrace{F_j\oplus F_{j-1}}}\overset{}{\rightarrow} \underset{{_sG_j}}{\underbrace{F_j\oplus F_{j-1}}}$$ for each $j$. By definition, ${_{t,s}\lambda_0}=\id_S$. This chain map induces   $f_{t,s}:\text{syz}_i\big(S/(\fp,y^t)\big)\rightarrow \text{syz}_i\big(S/(\fp,y^s)\big)$ which fits in the  diagram with exact rows of $S$-modules and homomorphisms
		$$
		\begin{CD}
		0@>>> \text{syz}_i(S/\fp) @>>> \text{syz}_{i}\big(S/(\fp,y^t)\big)@>>> \text{syz}_{i-1}(S/\fp) @>>> 0 \\
		@. @| @VVf_{t,s} V @VVy^{t-s}\id_{\text{syz}_{i-1}(S/\fp)}V @.\\
		0@>>> \text{syz}_i(S/\fp) @>>> \text{syz}_{i}\big(S/(\fp,y^s)\big)@>>> \text{syz}_{i-1}(S/\fp) @>>> 0 \\
		\end{CD}
		$$
		where the exact sequence of rows exist in view of display (\ref{EquationMappingConeExactSequence}) in the previous remark. Now, part (\ref{itm:SyzygyWtoSyzygy1ExactSequence}) follows after applying the  Snake Lemma to the above diagram.
		
		Dual to the above chain map, one can see that there is a chain map ${_{s,t}\lambda_\bullet}:{_sG_\bullet}\rightarrow {_tG_\bullet}$ given by $$\begin{pmatrix} y^{t-s}\id_{F_j} &  0 \\ 0 & \id_{F_{j-1}}  \end{pmatrix}:\underset{{_sG_j}}{\underbrace{F_j\oplus F_{j-1}}}\overset{}{\rightarrow} \underset{{_tG_j}}{\underbrace{F_j\oplus F_{j-1}}}$$ which similarly as above induces $f_{s,t}:\text{syz}_{i}\big(S/(\fp,y^s)\big)\rightarrow\text{syz}_{i}\big(S/(\fp,y^{t})\big)$ fitting in the diagram with exact rows
		\vspace{3mm}
		$$
		\begin{CD}
		0@>>> \text{syz}_i(S/\fp) @>>> \text{syz}_{i}\big(S/(\fp,y^s)\big)@>>> \text{syz}_{i-1}(S/\fp) @>>> 0 \\
		@. @VVy^{t-s}\id_{\text{syz}_{i}(S/\fp)}V @VVf_{s,t} V @| @.\\
		0@>>> \text{syz}_i(S/\fp) @>>> \text{syz}_{i}\big(S/(\fp,y^t)\big)@>>> \text{syz}_{i-1}(S/\fp) @>>> 0 \\
		\end{CD}.
		$$
		\vspace{1mm}

		Again part (\ref{itm:Syzygy1toSyzygyW}) follows from the Snake Lemma and the above diagram. The last statement of part (\ref{itm:SyzygiesMapsBetweenThem}) is evident from the definition of the aforementioned chain maps.
	\end{proof}
\end{lem}

The next proposition is, roughly speaking, a part of  Proposition-Definition \ref{PropositionVarientCET}.

\begin{prop}\label{PropositionCanonicalElementTimesParameter}
	Let $(S,\fm_S)$ be a local ring of dimension $d$ which is a homomorphic image of a Gorenstein ring (thus it admits a canonical module $\omega_S$). Let $\fq$ be an ideal of $S$ of height $d-1$ and $x_1$ be an element of $S$ such that $(\fq,x_1)$ is $\fm_S$-primary. The following statements are equivalent (for the notation used here see Definition \ref{DefinitionCanonicalElement}):
	\begin{enumerate}[(i)]
		\item \label{itm:CanonicalElementScalarMultiplication} $x_1^{w-1}\eta_{(\fq,x_1^w)}\neq 0$.
		\item \label{itm:SomeMapFromExtToLocalCohomologyNonZero} Let $\nu_w:S/(\fq,x_1)\rightarrow S/(\fq,x_1^w)$ be the multiplication map by $x_1^{w-1}$.  The following  map is non-zero: $$\Ext^d_{S}\big(S/(\fq,x_1^w),\omega_S\big)\overset{\Ext^d_S(\nu_w,\id_{\omega_S})}{\longrightarrow}\Ext^d_{S}\big(S/(\fq,x_1),\omega_S\big)\overset{\Theta^{\omega_S}_{(\fq,x_1)}}{\longrightarrow}H^{d}_{(\fq,x_1)}(\omega_S).$$ 
		\item \label{itm:AnotherMapFromExtToLocalCohomologyNonZero} The following  map is non-zero:
		$$ \Ext^d_{S}\big(S/(\fq,x_1^w),\omega_S\big)\overset{\Theta^{\omega_S}_{(\fq,x_1^w)}}{\longrightarrow}H^{d}_{(\fq,x_1^w)}(\omega_S) \overset{x_1^{w-1}}{\longrightarrow}H^{d}_{(\fq,x_1^w)}(\omega_S).$$
	\end{enumerate}
  \begin{proof}
   	 (\ref{itm:CanonicalElementScalarMultiplication})$\Leftrightarrow$ (\ref{itm:SomeMapFromExtToLocalCohomologyNonZero}): Our proof is an adaptation of the proof of \cite[Theorem (4.3)]{HochsterCanonical} to the situation of our lemma. Let  $$\Psi_M:H^d_{(\fq,x_1)}(M)\rightarrow H^d_{(\fq,x_1^w)}(M)$$ be the isomorphism induced by the maps $\Ext^d_S\big(S/(\fq,x_1)^n,M\big)\rightarrow \Ext^d_S\big(S/(\fq,x_1^w)^n,M\big)$ each of which, itself, is induced by the natural epimorphism $\pi_n:S/(\fq,x_1^w)^n\rightarrow S/(\fq,x_1)^n$.  Let us also to denote $\Psi_{\omega_S}$ by $\Psi$. First we prove two claims.
   	 
     \vspace{2mm}
  	 \textbf{Claim 1:} $\Psi_M\circ \Theta^M_{(\fq,x_1)}\circ \Ext^d_S\big(\nu_w,\id_{M}\big)=x_1^{w-1}.\Theta^M_{(\fq,x_1^w)}$ for each $S$-module $M$.

  	 Proof of the claim:  Our claim easily follows from the identities
  	 \begin{align*}
  	   \Psi_M\circ \Theta^M_{(\fq,x_1)}\circ \Ext^d_S\big(\nu_w,\id_{M}\big)&=\Theta^M_{(\fq,x_1^w)}\circ \Ext^d_S(\pi_1,\id_{M})\circ \Ext^d_S\big(\nu_w,\id_{M}\big)&\\&=\Theta^M_{(\fq,x_1^w)}\circ \Ext^d_S\big(x_1^{w-1}\id_{S/(\fq,x_1^w)},\id_{M}\big).
  	 \end{align*}
  	 
  	 \vspace{2mm}
  	 \textbf{Claim 2:} If there is some $S$-module $M$ for which  the map $$\Theta^M_{(\fq,x_1)}\circ \Ext^d_S(\nu_w,\id_M):\Ext^d_S\big(S/(\fq,x_1^w),M\big)\rightarrow H^d_{(\fq,x_1)}(M)$$ is non-zero, then $x_1^{w-1}\eta_{(\fq,x_1^w)}\neq 0$.
  	 
  	 Proof of the claim: Let $(G_\bullet,\partial^{G_\bullet}_{\bullet})$ be a free resolution of $S/(\fq,x_1^w)$, and pick some $\varphi:G_d\rightarrow M$ such that $[\varphi]\in \Ext^d_{S}\big(S/(\fq,x_1^w),M)\big)$ has non-zero image under $\Theta^M_{(\fq,x_1)}\circ \Ext^d_S(\nu_w,\id_M)$. As $\varphi$ is a cycle,  $\varphi$ factors through $\overline{\varphi}:G_d/\im(\partial^{G_\bullet}_{d+1})=\text{syz}^{G_\bullet}_d\big(S/(\fq,x_1^w)\big)\rightarrow M$. By our hypothesis and the bijectivity of $\Psi_M$,
  	  \begin{equation} 
  	  \label{EquationClaim1ImpliesNonZero}
  	    x_1^{w-1}.\Theta^M_{(\fq,x_1^w)}([\varphi])\overset{\text{Claim 1}}{=}\big(\Psi_M\circ \Theta^M_{(\fq,x_1)}\circ \Ext^d_S(\nu_w,\id_M)\big)([\varphi])\neq 0.
  	  \end{equation}  
  	In view of the  definition of $\epsilon_{(\fq,x_1^w)}$, it is readily seen that 
  	\begin{equation}
  	  \label{EquationInterPretationOFVarphi}
  	  \Ext^d_S\big(S/(\fq,x_1^w),\overline{\varphi}\big)(\epsilon_{(\fq,x_1^w)})=[\varphi].
  	\end{equation}  
  	   Therefore,
  	   \begin{align*}
  	     H^d_{(\fq,x_1^w)}(\overline{\varphi})(x_1^{w-1}\eta_{(\fq,x_1^w)})&=
  	     x_1^{w-1}H^d_{(\fq,x_1^w)}(\overline{\varphi})\circ \Theta^{\text{syz}^{G_\bullet}_d\big(S/(\fq,x_1^w)\big)}_{(\fq,x_1^w)}(\epsilon_{(\fq,x_1^w)})&\\&=
  	     x_1^{w-1}\Theta^M_{(\fq,x_1^w)}\circ \Ext^d_S\big(S/(\fq,x_1^w),\overline{\varphi}\big)(\epsilon_{(\fq,x_1^w)})&\\&=
  	     x_1^{w-1}\Theta^M_{(\fq,x_1^w)}([\varphi]), && \text{by Eq. (\ref{EquationInterPretationOFVarphi}}) &\\&
  	     \neq 0, && \text{by Eq. (\ref{EquationClaim1ImpliesNonZero}})
  	   \end{align*}  
  	   so Claim 2 follows.

One can readily verify that $x_1^{w-1}\eta_{(\fq,x_1^w)}\neq 0$ in $S$ if and only if $x_1^{w-1}\eta_{(\fq,x_1^w)\widehat{S}}\neq 0$ in $\widehat{S}$, one approach would be using the faithful flatness of the completion map which induces an injective map on the relevant Ext-modules as well as considering  each canonical element as an element of the direct limit of Ext-modules (see also \cite[Corollary (3.6) and Proposition (3.18)]{HochsterCanonical}). Similarly, the composited map in part (\ref{itm:SomeMapFromExtToLocalCohomologyNonZero}) is non-zero if and only if the analogue map $\Ext^d_{\widehat{S}}\big(\widehat{S}/(\fq,x_1^w),\omega_{\widehat{S}}\big)\rightarrow \Ext^d_{\widehat{S}}\big(\widehat{S}/(\fq,x_1),\omega_{\widehat{S}}\big)\rightarrow H^d_{(\fq,x_1)}(\omega_{\widehat{S}})$ is non-zero (we recall that $\omega_{\widehat{S}}=\omega_{S}\otimes_S\widehat{S}$). Thus in order to prove the equivalence of  (\ref{itm:CanonicalElementScalarMultiplication}) and  (\ref{itm:SomeMapFromExtToLocalCohomologyNonZero}), without loss of generality, we can assume that $S$ is complete.
  	   
  	   By \cite[page 534]{HochsterCanonical} (see also the proof of \cite[Theorem (4.3)]{HochsterCanonical}), there exists an $S$-homomorphism $\delta:H^d_{(\fq,x_1^w)}(\omega_S)\rightarrow E(S/\fm_S)$ which induces the functorial isomorphism 
  	   $$\rho_M:\Hom_S(M,\omega_S)\rightarrow H^d_{(\fq,x_1^w)}(M)^{\vee},\ \ f\mapsto \big(h\mapsto \delta\circ H^d_{(\fq,x_1^w)}(f)(h)\big).$$ Having the isomorphism  $\rho:=\rho_{\text{syz}^{G_\bullet}_d\big(S/(\fq,x_1^w)\big)}$,  we can construct the diagram
  	   \begin{center}
  	   	  $\begin{CD} 
  	   	    \Hom_S\Big(\text{syz}^{G_\bullet}_d\big(S/(\fq,x_1^w)\big),\omega_S\Big) @>\rho\ (\cong)>> 	H^d_{(\fq,x_1^w)}\Big(\text{syz}^{G_\bullet}_d\big(S/(\fq,x_1^w)\big)\Big)^\vee @>\gamma^\vee>> S^\vee @>\Xi\ (\cong)>> E(S/\fm_S)\\
  	   	    @VV\beta:\ f\mapsto \Ext^d_S(\id,f)(\epsilon_{(\fq,x_1^w)}) V @. @. @AA\delta A\\
  	   	    \Ext^d_S\big(S/(\fq,x_1^w),\omega_S\big) @>>\Ext^d_S(\nu_w,\id)> \Ext^d_S\big(S/(\fq,x_1),\omega_S\big) @>>\Theta^{\omega_S}_{(\fq,x_1)}> H^d_{(\fq,x_1)}(\omega_S) @>\cong>\Psi> H^d_{(\fq,x_1^w)}(\omega_S) 
  	   	  \end{CD}$
  	   \end{center}
     where $\gamma:S\rightarrow H^d_{(\fq,x_1^w)}\Big(\text{syz}^{G_\bullet}_d\big(S/(\fq,x_1^w)\big)\Big)$ is given by $1\rightarrow x_1^{w-1}\eta_{(\fq,x_1^w)}$, and  $\Xi:S^\vee\rightarrow E(S/\fm_S)$ is the evaluation map at $1$, thus  $x_1^{w-1}\eta_{(\fq,x_1^w)}\neq 0$ precisely when $\Xi\circ \gamma^\vee\circ \rho\neq 0$. Thus, if we can show that the diagram is commutative, then  from the non-vanishing of  $x_1^{w-1}\eta_{(\fq,x_1^w)}\neq 0$ we get the non-vanishing of $\Theta^{\omega_S}_{(\fq,x_1)}\circ \Ext^d_S(\nu_w,\id)$. The reverse implication also holds in view of Claim 2. Consequently, the proof is complete once  the commutativity of the diagram is shown. This  commutativity can be readily checked using Claim 1, so we leave it to the reader to check that the diagram is commutative.
     
     (\ref{itm:SomeMapFromExtToLocalCohomologyNonZero})$\Leftrightarrow $(\ref{itm:AnotherMapFromExtToLocalCohomologyNonZero}): Let, $\pi_1:S/(\fq,x_1^w)\rightarrow S/(\fq,x_1)$ be the natural epimorphism. We take into account the diagram
     \begin{center}
     	\xymatrix{
     		 \Ext^d_S\big(S/(\fq,x_1^w),\omega_S\big)
     		\ar[rr]^{\ext^d_S(\nu_w,\id)} 
     		\ar[rrd]_{x_1^{w-1}.\id}
     		&&\Ext^d_S\big(S/(\fq,x_1),\omega_S\big)
     		\ar[rr]^{\Theta^{\omega_S}_{(\fq,x_1)}}
     		\ar[d]^{\Ext^d_S(\pi_1,\id)}
     		&& H^d_{(\fq,x_1)}(\omega_S) 
     		\ar[d]^{\Psi}_{\cong}
     		\\
     		&&\Ext^d_S\big(S/(\fq,x_1^w),\omega_S\big) 
              \ar[rr]_{\Theta^{\omega_S}_{(\fq,x_1^w)}} 
     		&& H^d_{(\fq,x_1^w)}(\omega_S) 
     	}
     \end{center}     
    which is commutative in view of  Claim 1 in the proof of the previous implications. Then the statement is clear in view of the commutativity of this diagram as well as the bijectivity of $\Psi$.
  \end{proof}
\end{prop}

\begin{consi}\label{Consideration}
	\emph{Let $(S,\fm_S)$ be a local ring of dimension $d$. Throughout the rest of this section we will repeatedly take into account  the following data:}
\begin{itemize}
		 \item \emph{a system of parameters $\mathbf{x}:=x_1,\ldots,x_d$ of $S$.}
		 \item \emph{some height $d-1$ ideal $\fq$ containing $x_2,\ldots,x_d$ (but not $x_1$).}
		 \item \emph{some $w\in \mathbb{N}$.}
		 \item \emph{a (not necessarily minimal) free resolution $G_\bullet$ of $S/(\fq,x_1^w)$.}
	\end{itemize}
\end{consi}

\begin{prop-defi} \label{PropositionVarientCET} (A variant of the Canonical Element Theorem) Let $(S,\fm_S)$ be   a local ring of dimension $d$.   The following statements are equivalent.
	\begin{enumerate}[(i)]
    	  \item \label{itm:VariantCEPropertyUnifiedSOP}
		  	For any choice of $\mathbf{x},\fq,w,G_\bullet$ as in  Consideration \ref{Consideration},
		   any chain map
		   $$\phi_\bullet:K_\bullet(x_1^w,x_2,\ldots,x_{d};S)\rightarrow G_\bullet\ \ \text{s.t.\ \ } \overline{\phi_0}\big(r+(x_1^w,x_2,\ldots,x_d)\big)=r+(\fq,x_1^w)$$  and any $v\in\mathbb{N}$,
		  \begin{align*}
		    (x_1^{wv-w}x_2^{v-1}\cdots x_{d}^{v-1}).\big(x_1^{w-1}\phi_{d}(1)\big)+\im&(\partial^{G_\bullet}_{d+1})&\\&\notin (x_1^{wv},x_2^v,\ldots,x_{d}^v)\text{syz}^{G_\bullet}_{d}\big(S/(\fq,x_1^w)\big).
		  \end{align*}		
		  \item \label{itm:VariantCEProperty} Statement (\ref{itm:VariantCEPropertyUnifiedSOP}) holds when $v=1$ is fixed, i.e. 
		  $$x_1^{w-1}\phi_{d}(1)+\im(\partial^{G_\bullet}_{d+1})\notin (x_1^{w},x_2,\ldots,x_{d})\text{syz}^{G_\bullet}_{d}\big(S/(\fq,x_1^w)\big). $$
	      \item \label{itm:Variant0thHomologyMapTheMultiplicationMap} 
	       For any $\mathbf{x},\fq,w,G_\bullet$ as in Consideration \ref{Consideration}, and	 for any    chain map 
	      $$\phi_\bullet:K_\bullet(x_1^w,x_2,\ldots,x_{d};S)\rightarrow G_\bullet\ \ \text{s.t.\ \ }  \overline{\phi_0}\big(r+(x_1^w,x_2,\ldots,x_{d})\big)=x_1^{w-1}.r+(\fq,x_1^w)$$ we have,	     
	      $\phi_{d}\neq 0$.	
	      \item 	\label{itm:VariantCET}		For any  $\mathbf{x},\fq,w$ as in Consideration \ref{Consideration},
	      $$x_1^{w-1}\eta_{(\fq,x_1^w)}\neq 0.$$				
	      \item \label{itm:ExtInterpretation} (Assuming that $S$ admits a canonical module) For any $\mathbf{x},\fq,w$ as in Consideration \ref{Consideration}, the composited map $$\Ext^d_S\big(S/(\fq,x_1^w),\omega_S\big)\overset{\Ext^d_S(\nu_w,\id_{\omega_S})}{\longrightarrow} \Ext^d_S\big(S/(\fq,x_1),\omega_S\big)\overset{\Theta^{\omega_S}_{(\fq,x_1)}}{\longrightarrow} H^d_{(\fq,x_1)}(\omega_S)$$
	      is non-zero, where $\nu_w:S/(\fq,x_1)\rightarrow S/(\fq,x_1^w)$ is the multiplication map by $x_1^{w-1}$ and $\Theta^{\omega_S}_{(\fq,x_1)}$ is the natural map to the local cohomology.
	      \item \label{itm:ExtSecondInterpretation} (Assuming that $S$ admits a canonical module) For any $\mathbf{x},\fq,w$ as in Consideration \ref{Consideration}, the composited map
	      $$\Ext^d_S\big(S/(\fq,x_1^w),\omega_S\big) \overset{\Theta^{\omega_S}_{(\fq,x_1^w)}}{\longrightarrow} H^d_{(\fq,x_1^w)}(\omega_S)\overset{x_1^{w-1}\id}{\longrightarrow}  H^d_{(\fq,x_1^w)}(\omega_S)$$ is non-zero.
	      \item\label{itm:Definition} The local ring $S$ satisfies the Strong Canonical Element Property.
     \end{enumerate}	
	 \begin{proof}      
	   
	   (\ref{itm:VariantCEPropertyUnifiedSOP})$\Leftrightarrow$(\ref{itm:VariantCEProperty}):	 Only one side needs a proof. The first guess would be that this implication follows by the same argument as in \cite[Remark  (2.2)(7)]{HochsterCanonical}. But some   more work is required in this situation. Because if we change our under consideration system of parameters from $x_1^w,x_2,\ldots,x_d$ to $x_1^{wv},x_2^v,\ldots,x_d^v$,  then we have to also consider $\text{syz}_d^{}\big(S/(\fq,x_1^{vw})\big)$ in place of  $\text{syz}_{d}\big(S/(\fq,x_1^w)\big)$. 
	     
	      Arguing by contradiction, we consider a chain  map $\phi_\bullet$ as in   (\ref{itm:VariantCEPropertyUnifiedSOP}) such that 	
	      \begin{align}
	      \label{EquationOneImpossibleRelation}
	      (x_1^{wv-w}x_2^{v-1}\cdots x_{d}^{v-1}).\big(x_1^{w-1}\phi_{d}(1)\big)+\im&(\partial^{G_\bullet}_{d+1})\in (x_1^{vw},x_2^v,\ldots,x_{d}^v)\text{syz}^{G_\bullet}_{d}\big(S/(\fq,x_1^w)\big)
	      \end{align}		
	      for some $v\ge 2$ (if $v=1$ then we  already get  a contradiction).   Without loss of generality, we can suppose that $\fq$ is a prime ideal. 
	      Namely, since $\fq$ has height $d-1$ and $\mathbf{x}$ is a system of parameters, so there exists a minimal prime $\fp$ of $\fq$ such that $x_1\notin \fp$. Then a composition of $\phi_\bullet$ with a chain map from the free resolution of $S/(\fq,x_1^w)$ to a free resolution of $S/(\fp,x_1^w)$ that lifts the natural epimorphism on the $0$-th homologies  yields a similar equation as in (\ref{EquationOneImpossibleRelation})   but in the $d$-th syzygy of $S/(\fp,x_1^w)$. Henceforth, until the end of proof of this implication we assume that $\fq$ is a prime ideal.
	      
	      Taking into account the chain map  $_{v}\eta_\bullet:K_\bullet(x_1^{wv},x_2^v,\ldots,x_d^v;S)\rightarrow K_\bullet(x_1^w,x_2,\ldots,x_d;S)$ as in Lemma \ref{LemmaInducedMapOnKoszulHomolgoies}(\ref{itm:PowerOfX}) and setting  $$\psi_\bullet:=\phi_\bullet\circ{_{v}\eta_\bullet}:K_\bullet(x_1^{wv},x_2^v,\ldots,x_d^v;S)\rightarrow G_\bullet$$  we get 
	      \begin{equation}
	        \label{EquationFirstBadRelationForWv}
	        x_1^{w-1}\psi_d(1)+\im(\partial^{G_\bullet}_{d+1})\in (x_1^{wv},x_2^v,\ldots,x_{d}^v)\text{syz}^{G_\bullet}_{d}\big(S/(\fq,x_1^w)\big).
	      \end{equation}   
	      
	      Next, we consider the minimal free resolution ${_wG_\bullet}$ (respectively, ${_{wv}G_\bullet}$) of $S/(\fq,x_1^w)$ (respectively, $S/(\fq,x_1^{wv})$) as in Remark \ref{RemarkMappingConeFreeResolution} as well as the chain maps $$K_\bullet(x_1^{wv},x_2^v,\ldots,x_d^v;S)\overset{\gamma_\bullet}{\longrightarrow} {_{wv}G_\bullet}\overset{{_{wv,w}\lambda_\bullet}}{\longrightarrow} {_{w}G_\bullet}$$
	      where $\gamma_\bullet$ is induced by the natural epimorphism on $0$-th homologies and ${_{wv,w}\lambda_{\bullet}}$ is as in Lemma \ref{LemmaSysyzgiesRelations}(\ref{itm:SyzygiesMapsBetweenThem}). Therefore, as ${_{wv,w}\lambda_\bullet}=\id_S$, so from Lemma \ref{LemmaBasicFactInHomologicalAlgebra}(\ref{itm:RelationInSyz}) and display (\ref{EquationFirstBadRelationForWv}) we get  
	      \begin{equation}
	        \label{EquationSecondBadRelationForWv}
	        x_1^{w-1}.({_{wv,w}\lambda_d}\circ \gamma_d)(1)+\im(\partial^{{_wG_\bullet}}_{d+1})\in(x_1^{wv},x_2^v,\ldots,x_d^v)\text{syz}_{d}^{{_wG_\bullet}}\big(S/(\fq,x_1^w)\big)
	      \end{equation}
	       (because by mapping $G_\bullet$ to ${_wG_\bullet}$ appropriately, we can presume that the target of $\psi_\bullet$ is ${_wG_\bullet}$).  
	      
	      Next, in view of Lemma \ref{LemmaSysyzgiesRelations},  we  take into account ${_{w,wv}\lambda_\bullet}:{_{w}G_\bullet}\rightarrow {_{wv}G_\bullet}$ which has the property ${_{w,wv}\lambda_\bullet}\circ {_{wv,w}\lambda_{\bullet}}=x_1^{w(v-1)}\id_{{_{wv}G_\bullet}}$, so from display (\ref{EquationSecondBadRelationForWv})  and considering ${_{w,wv}\lambda_{\bullet}}\circ ({_{wv,w}\lambda_{\bullet}}\circ \gamma_{\bullet})$ in conjunction with $\ {_{w,wv}\lambda_{d}}\circ {_{wv,w}\lambda_{d}}=x_1^{w(v-1)}\id_{{_{wv}G_d}}$   we get
	      \begin{equation}
	         x_1^{wv-1}.\gamma_d(1)+\im(\partial^{{_{wv}G_\bullet}}_{d+1})\in(x_1^{wv},x_2^v,\ldots,x_d^v)\text{syz}_{d}^{{_{wv}G_\bullet}}\big(S/(\fq,x_1^{wv})\big).
	      \end{equation}
       This contradicts  (\ref{itm:VariantCEProperty}).
	 	
       (\ref{itm:VariantCEPropertyUnifiedSOP})$\Leftrightarrow$(\ref{itm:VariantCET}):  This is immediate in  light of \cite[Theorem (3.12)]{HochsterCanonical}.	  	 
       
       (\ref{itm:VariantCEProperty})$\Rightarrow$(\ref{itm:Variant0thHomologyMapTheMultiplicationMap}): Pick some chain map  $\psi_\bullet:K_\bullet(x_1^w,x_2,\ldots,x_d)\rightarrow G_\bullet$   lifting the natural epimorphism on the $0$-th homologies. If  there is another chain map $\phi_\bullet:K_\bullet(x_1^w,x_2,\ldots,x_d)\rightarrow G_\bullet$ as in the statement of (\ref{itm:Variant0thHomologyMapTheMultiplicationMap})  such that $\phi_d=0$, then since $x_1^{w-1}\psi_\bullet$ and $\phi_\bullet$ both induce the same map on the $0$-th homologies, so from Lemma \ref{LemmaBasicFactInHomologicalAlgebra}(\ref{itm:RelationInSyz}) and our hypothesis we get $$x_1^{w-1}\psi_d(1)+\im(\partial^{G_\bullet}_{d+1})\in (x_1^w,x_2,\ldots,x_d)\text{syz}^{G_\bullet}_d\big(S/(\fq,x_1^w)\big)$$ which is a contradiction.
       
       (\ref{itm:Variant0thHomologyMapTheMultiplicationMap})$\Rightarrow$(\ref{itm:VariantCEProperty}):  Suppose, to the contrary  that,  for some 
        chain map $\phi_\bullet:K_\bullet(x_1^w,x_2,\ldots,x_d)\rightarrow G_\bullet$ as in the statement of (\ref{itm:VariantCEProperty}), we have $$x_1^{w-1}.\phi_{d}(1)+\im(\partial^{G_\bullet}_{d+1})\in (x_1^{w},x_2,\ldots,x_{d})\text{syz}^{G_\bullet}_d\big(S/(\fq,x_1^w)\big).$$ 		
       
       Setting $\psi_\bullet:=x_1^{w-1}\phi_\bullet$,		 
       it is a chain map such that $\psi_0\big(r+(x_1^w,x_2,\ldots,x_d)\big)=\big(rx_1^{w-1}+(\fq,x_1^w)\big)$ on the $0$-th homologies  and such that  $\psi_{d}(1)+\im(\partial^{G_\bullet}_{d+1})\in (x_1^w,x_2,\ldots,x_d)\text{syz}_{d}^{G_\bullet}\big(S/(\fq,x_1^w)\big)$, say $$\psi_d(1)=x_1^w\alpha_1+\sum\limits_{i=2}^d\alpha_ix_i+\partial_{d+1}^{G_\bullet}(\beta),\ \ \alpha_1,\ldots,\alpha_d\in G_d,\ \beta\in G_{d+1}.$$
       Then, we turn $\psi_{\bullet}$ into another chain map $\psi'_\bullet:K_\bullet(x_1^w,x_2,\ldots,x_d)\rightarrow G_\bullet$ as follows. We set $\psi'_i=\psi_i$ for each $0\le i\le d-2$  and  $\psi'_{d}=0$, while $\psi'_{d-1}$ is defined as $$\psi'_{d-1}(e_1\wedge\cdots\wedge \widehat{e_{i}}\wedge\cdots \wedge e_d)=\psi_{d-1}(e_1\wedge\cdots\wedge \widehat{e_i}\wedge \cdots \wedge e_d)-(-1)^{i+1}\partial^{G_\bullet}_{d}(\alpha_i).$$ Then  $\psi'$ satisfies all of the conditions of (\ref{itm:Variant0thHomologyMapTheMultiplicationMap}), but the last non-vanishing condition of it. This is a contradiction.		        	  
       
       (\ref{itm:VariantCET})$\Leftrightarrow$ (\ref{itm:ExtInterpretation})$\Leftrightarrow$ (\ref{itm:ExtSecondInterpretation}):		     Without loss of generality, we can assume that $R$ is complete thus it is a homomorphic image of a Gorenstein ring. Then, these implications hold in view of Proposition \ref{PropositionCanonicalElementTimesParameter}.
	 \end{proof}
\end{prop-defi}

\begin{rem}\emph{If in the statement of  any part of the previous proposition we set $w=1$,  then  it follows from  the original Canonical Element Theorem.}
\end{rem}

\begin{rem}
	\emph{
		 At the time of writing this paper the author does not know any variant of the Improved New Intersection Theorem which is equivalent to the Strong Canonical Element Property.
	}
\end{rem}

The first part of the following corollary shows  how $\eta_{(\fq,x_1)}$ and $x_1^{w-1}\eta_{(\fq,x_1^{w})}$ are connected. The second part of the corollary can perhaps be useful to connect the non-vanishing of $x_1^{w-1}\eta_{(\fq,x_1^{w})}$ (for some $w$) to the formal local cohomology?
 The syzygies in the statement and the proof of the next corollary are computed with respect to the resolutions as in Remark \ref{RemarkMappingConeFreeResolution} and Lemma \ref{LemmaSysyzgiesRelations}.

\begin{cor}\label{CorollaryVariantCETCorollary}
	Let $\mathbf{x},\fq,w$ be as in Consideration \ref{Consideration}  and assume that $\fq$ is a prime ideal.
	\begin{enumerate}[(i)]
		\item \label{itm:MappingCEtoVCE} There are  $S$-homomorphisms $$\Ht^{d}_{(\fq,x_1^w)}\Big(\text{syz}_{d}\big(S/(\fq,x_1^w)\big)\Big)\overset{\mu_{w,1}}{\rightarrow} \Ht^{d}_{(\fq,x_1)}\Big(\text{syz}_{d}\big(S/(\fq,x_1)\big)\Big)\overset{\mu_{1,w}}{\rightarrow} \Ht^{d}_{(\fq,x_1^w)}\Big(\text{syz}_{d}\big(S/(\fq,x_1^w)\big)\Big)$$
		such that 
		$$\mu_{w,1}(\eta_{(\fq,x_1^w)})=\eta_{(\fq,x_1)},\ \ \ \mu_{1,w}\circ \mu_{w,1}=x_1^{w-1}\id.$$ 
		In particular, $$\mu_{1,w}(\eta_{(\fq,x_1)})=x_1^{w-1}\eta_{(\fq,x_1^w)}.$$
		\item \label{itm:CriterionVCENonzero} Suppose  that $x_1$ is a regular element of $S$. Let $$\delta_{\mathbf{x},\fq,w}:H^{d-1}_{(\fq,x_1)}\big(\text{syz}_{d}(S/\fq)/x_1^{w-1}\text{syz}_{d}(S/\fq)\big)\rightarrow H^{d}_{(\fq,x_1)}\Big(\text{syz}_{d}\big(S/(\fq,x_1)\big)\Big)$$ be the connecting homomorphism induced by the exact sequence in  Lemma \ref{LemmaSysyzgiesRelations}(\ref{itm:Syzygy1toSyzygyW}). A necessary and sufficient condition for the validity of   Proposition-Definition \ref{PropositionVarientCET}(\ref{itm:VariantCET})  for this  choice of $\mathbf{x},\fq,w$,   is that $$\eta_{(\fq,x_1)}\notin \text{im}(\delta_{\mathbf{x},\fq,w}).$$
	\end{enumerate}
	\begin{proof}
		Let $\zeta:S/(\fq,x_1^w)\rightarrow S/(\fq,x_1)$ be the natural epimorphism. 
		
		(\ref{itm:MappingCEtoVCE}) Following the notation of Lemma \ref{LemmaSysyzgiesRelations} we consider the commutative diagram,
		\vspace{2mm}
		\begin{center}
	   \xymatrix{
		  \ext^d_S\big(\frac{S}{(\fq,x_1^w)},\text{syz}_d(\frac{S}{(\fq,x_1^w)})\big) 
		  \ar[r]^{\ext^d_S(\id,f_{w,1})} \ar[dd]_{\Theta_{(\fq,x_1^w)}^{\text{syz}_d\big(S/(\fq,x_1^w)\big)}}
		   &\ext^d_S\big(\frac{S}{(\fq,x_1^w)},\text{syz}_d(\frac{S}{(\fq,x_1)})\big) 
     	   \ar[dd]_{\Theta_{(\fq,x_1^w)}^{\text{syz}_d\big(S/(\fq,x_1)\big)}}
		   & \ext^d_S\big(\frac{S}{(\fq,x_1)},\text{syz}_d(\frac{S}{(\fq,x_1)})\big) 
		   \ar[l]_{\ext^d_S(\zeta,\id)}
		   \ar[dd]^{\Theta_{(\fq,x_1)}^{\text{syz}_d\big(S/(\fq,x_1)\big)}}
		   \\\\
		  \lim\limits_{\underset{m\in \mathbb{N}}{\longrightarrow}} \ext^d_S\big(\frac{S}{(\fq,x_1^w)^m},\text{syz}_d(\frac{S}{(\fq,x_1^w)})\big) 
		  \ar[r]_{H^*}
		  &\lim\limits_{\underset{m\in \mathbb{N}}{\longrightarrow}} \ext^d_S\big(\frac{S}{(\fq,x_1^w)^m},\text{syz}_d(\frac{S}{(\fq,x_1)})\big) 
		  & \lim\limits_{\underset{m\in   \mathbb{N}}{\longrightarrow}} \ext^d_S\big(\frac{S}{(\fq,x_1)^m},\text{syz}_d(\frac{S}{(\fq,x_1)})\big) \ar[l]^{\Psi_1}
	    }
		\end{center}

		\vspace{1mm}
		where  $H^*=H^d_{(\fq,x_1^w)}(f_{w,1})$ and it is well-known that the natural induced map $\Psi_1$ on the limit of Ext modules is an isomorphism. We define $$\mu_{w,1}:=\Psi_1^{-1}\circ H^d_{(\fq,x_1^w)}(f_{w,1}).$$
		In view of the definition of  the element $\epsilon_{I}$ assigned to ideal $I$ and the induced maps on Ext-modules, it is readily verified that $\ext^d_S(\zeta,\id)(\epsilon_{(\fq,x_1)})=\ext^d_S(\id,f_{w,1})(\epsilon_{(\fq,x_1^w)})$, so the above commutative diagram shows that $\mu_{w,1}(\eta_{(\fq,x_1^w)})=\eta_{(\fq,x_1)}$.  Considering the induced map $$\Psi_w:\lim\limits_{\underset{m\in \mathbb{N}}{\longrightarrow}}\ext^d_S\Big(S/(\fq,x_1)^m,\text{syz}_d\big(S/(\fq,x_1^w)\big)\Big)\rightarrow \lim\limits_{\underset{m\in \mathbb{N}}{\longrightarrow}}\ext^d_S\Big(S/(\fq,x_1^w)^m,\text{syz}_d\big(S/(\fq,x_1^w)\big)\Big)$$ by $\zeta$ on direct limits of Ext modules, we set $$\mu_{1,w}:=\Psi_w\circ H^d_{(\fq,x_1)}(f_{1,w}):\lim\limits_{\underset{m\in \mathbb{N}}{\longrightarrow}}\ext^d_S\Big(S/(\fq,x_1)^m,\text{syz}_d\big(S/(\fq,x_1)\big)\Big)\rightarrow \lim\limits_{\underset{m\in \mathbb{N}}{\longrightarrow}}\ext^d_S\Big(S/(\fq,x_1^w)^m,\text{syz}_d\big(S/(\fq,x_1^w)\big)\Big).$$
		Finally from the commutative diagram,
		$$
		\begin{CD}
		\lim\limits_{\underset{m\in \mathbb{N}}{\longrightarrow}}\ext^d_S\Big(S/(\fq,x_1)^m,\text{syz}_d\big(S/(\fq,x_1)\big)\Big)
		@>\Psi_1>>
		\lim\limits_{\underset{m\in \mathbb{N}}{\longrightarrow}}\ext^d_S\Big(S/(\fq,x_1^w)^m,\text{syz}_d\big(S/(\fq,x_1)\big)\Big)
		\\
		@VH^d_{(\fq,x_1)}(f_{1,w})VV @VVH^d_{(\fq,x_1^w)}(f_{1,w})V\\
		\lim\limits_{\underset{m\in \mathbb{N}}{\longrightarrow}}\ext^d_S\Big(S/(\fq,x_1)^m,\text{syz}_d\big(S/(\fq,x_1^w)\big)\Big)
		@>>\Psi_w>
		\lim\limits_{\underset{m\in \mathbb{N}}{\longrightarrow}}\ext^d_S\Big(S/(\fq,x_1^w)^m,\text{syz}_d\big(S/(\fq,x_1^w)\big)\Big)
		\end{CD}
		$$
		we get
		\begin{center}
			\begin{align*}
			\mu_{1,w}\circ \mu_{w,1}
			=\Psi_w\circ H^d_{(\fq,x_1)}(f_{1,w}) \circ  \Psi_1^{-1}\circ H^d_{(\fq,x_1^w)}(f_{w,1})
			&=H^d_{(\fq,x_1^w)}(f_{1,w})\circ \Psi_1\circ  \Psi_1^{-1}\circ H^d_{(\fq,x_1^w)}(f_{w,1})
			&\\&=H^d_{(\fq,x_1^w)}(f_{1,w}\circ f_{w,1}) 
			&\\&=H^d_{(\fq,x_1^w)}\Big(x_1^{w-1}\id_{\text{syz}_d\big(S/(\fq,x_1^w)\big)}\Big) && \text{(by Lemma \ref{LemmaSysyzgiesRelations})}
			&\\&=x_1^{w-1}\id_{H^d_{(\fq,x_1^w)}\Big(\text{syz}_d\big(S/(\fq,x_1^w)\big)\Big)}.		 
			\end{align*}		 
		\end{center}
		
		(\ref{itm:CriterionVCENonzero})  By the  previous part,  $x_1^{w-1}\eta_{(\fq,x_1^w)}\neq 0$ if and only if $\mu_{1,w}(\eta_{(\fq,x_1)})\neq 0$ and by definition of $\mu_{1,w}$ this holds precisely when $H^d_{(\fq,x_1)}(f_{1,w})(\eta_{(\fq,x_1)})\neq 0$ (as $\Psi_w$ is an isomorphism). Thus,   the long exact sequence  obtained from  the exact sequence of Lemma \ref{LemmaSysyzgiesRelations}(\ref{itm:Syzygy1toSyzygyW}) and  $\Gamma_{(\fq,x_1)}(-)$  yields the statement.		
	\end{proof}
\end{cor}

\subsection{The validity of our variant of the Canonical Element Theorem}
In the remainder of this section, we first show that if $S$ has a balanced big Cohen-Macaulay module satisfying two additional properties, then $S$ has the Strong Canonical Element Property (Lemma \ref{LemmaDesiredBalancedBigCMImpliesVariantOFCEC}).
 Thereafter,  we will show that  a balanced big Cohen-Macaulay module with our desired additional properties exists    over any normal excellent local domain and so demonstrate that the Strong Canonical Element Property holds in general (Corollary \ref{CorollaryVariantHolds}). 
 We begin with a preparatory remark, definition, convention  and lemma.

The following remark will be referred to later.

\begin{rem}\label{RemarkVariantUnderLocalHomomorphism}
	\emph{Suppose that $\varphi:(S,\fm_S)\rightarrow (S',\fm_{S'})$ is a local ring homomorphism such that the image, $\varphi(\mathbf{x})$, of a system of parameters $\mathbf{x}:=x_1,\ldots,x_d$ of $S$ forms a system of parameters for $S'$. Let $\fq$ (respectively, $\fq'$) be a height $d-1$ ideal of $S$ (respectively,   of $S'$) such that $x_2,\ldots,x_d\in \fq$ (respectively, $\varphi(x_2),\ldots,\varphi(x_d)\in \fq'$). Assume that $\fq S'\subseteq \fq'$. If the conclusion  of any of  propositions \ref{PropositionVarientCET}(\ref{itm:VariantCEPropertyUnifiedSOP}), (\ref{itm:VariantCEProperty})  or (\ref{itm:Variant0thHomologyMapTheMultiplicationMap})  holds in $S'$ for this $\varphi(\mathbf{x}),\fq',w$ and any chain map as in the statement, then it   also holds    for $S$ for $\mathbf{x},\fq,w$ and any chain map as in the statement. Because, if otherwise, then by tensoring a counterexample into $S'$ and then mapping  $G_\bullet\otimes_SS'$ to a free resolution of $S'/\big(\fq',\varphi(x_1^w)\big)$ we get a contradiction in $S'$ for the triple $\varphi{(\mathbf{x})},\fq',w$.}
\end{rem}

 \begin{defi}\label{DefinitionGradedKoszul}(\cite[Definition 9.1.1]{BrunsHerzogCohenMacaulay})
 	\emph{
 		Let $S$ be a ring (possibly non-Noetherian), $I$  an ideal generated by $\mathbf{x}:=x_1,\ldots,x_n$ and $M$ be an $R$-module. If all Koszul homologies $H_i(\mathbf{x},M)$ vanish, then we set $\text{grade}(I,M)=\infty$; otherwise, $\text{grade}(I,M)=n-v$ where $v=\sup\{i:H_i(\mathbf{x},M)\neq 0\}$. 
 	} 
 \end{defi}

 \begin{conv}
 	\emph{
 		In what follows when no specific generating set for $I$ is determined, by abuse of notation, we let $H_i(I,M)$  denote the $i$-th Koszul homology of $M$ with respect to any generating set of $I$. For our purpose here, which is the realization of grade in terms of Koszul homologies, this convention causes no confusion (see the paragraph after \cite[Definition 9.1.1]{BrunsHerzogCohenMacaulay}).}
 \end{conv}

See Remark \ref{RemarkImposedConditionsForLemmaNonVanishingOfLocalCohomologyAtGrade} for some comments on the imposed  conditions in the statement of Lemma \ref{LemmaInjectivityOfMapTOLocalCohomologyAtGrade}(\ref{itm:ConditionSufficientForLocalCohomologyNonZeroAtGrade}).
 
 \begin{lem}\label{LemmaInjectivityOfMapTOLocalCohomologyAtGrade} Let $(S,\fm_S)$ be a local ring, $\fa$ be an  ideal of $S$ generated by $n$ elements and $M$ be an $S$-module. 
 	\begin{enumerate}[(i)]
 	\item	\label{itm:LocalCohomologiesBelowTheGradeAllVanish}
        	 $H^i_{\fa}(M)=0$ for each $0\le i< \text{grade}(\fa,M)$.
    \item \label{itm:ConditionSufficientForLocalCohomologyNonZeroAtGrade}    	         	  Suppose that $M$ is a balanced big Cohen-Macaulay  $S$-module and that  $\grade_{}(\fa,M)=g=\text{height}(\fa)$. Then  the  map $$H_{n-g}(\fa,M)\rightarrow H^g_{\fa}(M)$$ (as in \cite[Theorem 5.2.9]{BrodmannSharpLocal})  is injective. In particular, $$H^g_{\fa}(M)\neq 0.$$
 	\end{enumerate}
 	\begin{proof}
 	 (\ref{itm:LocalCohomologiesBelowTheGradeAllVanish})	
 	  First from \cite[Proposition 9.1.3(b)]{BrunsHerzogCohenMacaulay}  we get $\text{grade}(\fa^s,M)=\text{grade}(\fa,M)$ for each $s\in \mathbb{N}$, therefore $\Ext^{i}_{S}(S/\fa^s,M)=0$ for each $0\le i<\text{grade}(\fa,M)$ and each $s\in \mathbb{N}$ by \cite[Exercise 9.1.10(d)]{BrunsHerzogCohenMacaulay}. Consequently, 
 	 $$\forall\ 0\le i<\text{grade}(\fa,M),\ \ \ H^i_{\fa}(M)=0\ \ \ (\text{\cite[Theorem 1.3.8]{BrodmannSharpLocal}}).$$
 		
 	 (\ref{itm:ConditionSufficientForLocalCohomologyNonZeroAtGrade})	We use induction on $g$. If $g=0$ then this map is the inclusion $0:_M\fa \rightarrow \Gamma_{\fa}(M)$ and the statement holds. Suppose that $g>0$ and the statement holds for $g-1$.
 		
 		Since $\text{height}(\fa)>0$, so $\fa$ contains some $x\in \fa$ such that $x\in S^\circ=S\backslash \cup_{\fp\in \text{min}(S)}\fp$, in particular $x$ forms a parameter element in $S$ and in any localization   $S_{\fp'}$ of $S$.
 		Thus 
 		\begin{equation}
 		  \label{EquationHeightIsReducedByOne}
 		  \text{height}_{S/xS}(\fa/xS)=\text{height}(\fa)-1=g-1.
 		\end{equation}
 		Moreover,  		
 since $M$ is assumed to  be a balanced big Cohen-Macaulay module, so 		$x$  is   regular over $M$, 
 		 $M/xM$ is a balanced big Cohen-Macaulay $S/xS$-module (\cite[Lemma (2.3)]{SharpCohenMacaulayProperties}) and finally 
 		 \begin{equation}
 		   \label{EquationGradeAndHieghtAndSpecialization}
 		   \text{grade}(\fa/xS,M/xM)\overset{\text{\cite[Proposition 9.1.2(b)]{BrunsHerzogCohenMacaulay}}}{=}g-1\overset{\text{Eq.  (\ref{EquationHeightIsReducedByOne})}}{=}\text{height}_{S/xS}(\fa/xS).
 		 \end{equation}  
 		
 		Using long exact sequences arising from $0\rightarrow M\rightarrow M\rightarrow M/xM\rightarrow 0$, we get the diagram 
 		\begin{center}
 			$\begin{CD}
 			\underset{\substack{=0\\\text{Part (\ref{itm:LocalCohomologiesBelowTheGradeAllVanish})}}}{\underbrace{H^{g-1}_{\fa}(M)}} @>>> H^{g-1}_{\fa}(M/xM)@>\text{injective}>> H^g_{\fa}(M)\\
 			@AAA @AAA @AAA\\
 			\underset{\substack{=0\\\text{(definition of grade)}}}{\underbrace{H_{n-(g-1)}(\fa,M)} } @>>> H_{n-(g-1)}(\fa,M/xM) @>>\cong> H_{n-g}(\fa,M) @>>> \underset{\substack{=0\ \\(\text{\cite[Proposition 1.6.5(b)]{BrunsHerzogCohenMacaulay}}})}{\underbrace{xH_{n-g}(\fa,M)}}
 			\end{CD}$
 		\end{center}
 		Hence, it suffices to notice that $ H_{n-(g-1)}(\fa,M/xM) \rightarrow H^{g-1}_{\fa}(M/xM)$ is injective, which holds true by our inductive hypothesis together with display (\ref{EquationGradeAndHieghtAndSpecialization}) and the balanced big Cohen-Macaulayness of $M/xM$ (over $S/xS$) as well as the fact that $H_{n-(g-1)}(\fa,M/xM)$ and $H^{g-1}_{\fa}(M/xM)$ both can be realized as the same corresponding Koszul homology and local cohomology  over the ring $S/xS$ (with respect to $\fa( S/xS)$).
 	\end{proof}
 \end{lem}

\begin{rem}\label{RemarkImposedConditionsForLemmaNonVanishingOfLocalCohomologyAtGrade}
   \emph{	In the proof of the second part of the previous lemma the assumption on the balanced big Cohen-Macaulayness of $M$ is used (thus perhaps is necessary) for deducing the existence of an $M$-regular element $x\in \fa$ from $\text{grade}(\fa,M)>0$ (see also \cite[Theorem (2.6)]{SharpCohenMacaulayProperties} and \cite[Exercise 9.1.10(d)]{BrunsHerzogCohenMacaulay}). Then the assumption on $\grade(\fa,M)=\text{height}(\fa)$ is imposed to enable us to pick $x\in \fa$ so that $M/xM$ is a balanced big Cohen-Macaulay $S/xS$-module as well, which  is needed for our inductive step (cf. Remark \ref{RemarkGradeCanBeLargerThanHeight}).}
\end{rem}

\begin{lem}\label{LemmaDesiredBalancedBigCMImpliesVariantOFCEC} Fix some $\fp,\mathbf{x},w,G_\bullet$ as in Consideration \ref{Consideration} such that $\fp$ is a prime ideal \footnote{Here we used the notation $\fp$ in place of $\fq$ when we  referred to Consideration \ref{Consideration}  because  it is a prime ideal.}. Suppose that  there is a balanced big Cohen-Macaulay $S$-module $M$ such that
	\begin{enumerate}[(i)]
		\item\label{itm:AssCondition} $\fp\in \Ass_S\big(M/(x_2,\ldots,x_d)M\big)$, or equivalently $\text{grade}(\fp,M)=d-1$. In the case where $S$ is (furthermore) a catenary domain, this is also equivalent to say that $M_\fp$ is a  balanced big Cohen-Macaulay $S_\fp$-module.
		\item\label{itm:SeparatedCondition}  $M/(x_2,\ldots,x_d)M$ is $\fm_S$-adically separated. 
	\end{enumerate}	
		Then  Proposition-Definition \ref{PropositionVarientCET}(\ref{itm:VariantCEProperty})  holds for $\fp,\mathbf{x},w,G_\bullet$ and any chain map $\phi_\bullet:K_\bullet(x_1^w,x_2,\ldots,x_d;S)\rightarrow G_\bullet$ lifting the natural epimorphism on $0$-th homologies.
	\begin{proof}
        First we notice that,   if $\fp\in \Ass_S\big(M/(x_2,\ldots,x_d)M\big)$ then $\grade\big(\fp,M/(x_2,\ldots,x_d)M\big)=0$ (\cite[Proposition 9.1.2(a)]{BrunsHerzogCohenMacaulay}), i.e. $\text{grade}(\fp,M)=d-1$ (\cite[Proposition 9.1.2(b)]{BrunsHerzogCohenMacaulay}). Conversely, suppose that $\text{grade}(\fp,M)=d-1$  or equivalently $\grade\big(\fp,M/(x_2,\ldots,x_d)M\big)=0$. Then $M/(x_2,\ldots,x_d)M$ is a balanced big Cohen-Macaulay module over $S/(x_2,\ldots,x_d)$ (\cite[Lemma (2.3)]{SharpCohenMacaulayProperties}), therefore we must have $\fp\in \Ass\big(M/(x_2,\ldots,x_d)M\big)$ because $\fp$ annihilates  a non-zero element of $M/(x_2,\ldots,x_d)M$ in view of \cite[Proposition 9.1.2(a)]{BrunsHerzogCohenMacaulay}, while $\fm\notin \Ass\big(M/(x_2,\ldots,x_d)M\big)$ by \cite[Proposition 8.5.5]{BrunsHerzogCohenMacaulay}.  If $S$ is a catenary domain, for the mentioned equivalence in (\ref{itm:AssCondition}) with the balanced big Cohen-Macaulayness of $M_\fp$ over $S_\fp$ we refer to \cite[Theorem 4.3]{SharpACousin}. Here we remark further  that if $M_\fp$ is a balanced big Cohen-Macaulay $S_\fp$-module, then $\grade(\fp S_\fp,M_\fp)=\text{height}(\fp)$ by \cite[Exercise 9.1.12]{BrunsHerzogCohenMacaulay} \footnote{See \cite[page 350]{BrunsHerzogCohenMacaulay},  for the notion of depth used in the cited exercise.} from which we can deduce that $\fp\in \Ass\big(M/(x_2,\ldots,x_d)M\big)$ as above.

		Suppose to the contrary that there is some chain map $\phi_\bullet:K_\bullet(x_1^w,x_2,\ldots,x_d;S)\rightarrow G_\bullet$ lifting the natural epimorphism on $0$-th homologies, for which Proposition-Definition \ref{PropositionVarientCET}(\ref{itm:VariantCEProperty}) does not hold.  Let   $e:=\phi_0(1)$ be a basis of $G_0$  representing the homology class $[1+(x_1^w,\fp)]$ in $H_0(G_\bullet)$, and let $\mathbf{x}':=x_2,\ldots,x_d$. 
    	By condition (\ref{itm:AssCondition}),
		there is  some $b+\mathbf{x}'M\in M/ \mathbf{x}'M$ with  	  $\fp= (\mathbf{x}'M):_Sb\subseteq \fm_S$. Additionally, without loss of generality, we may and we do assume that 
		\begin{equation}
		\label{EquationbIsNotThere}
		b\notin (\mathbf{x})M.
		\end{equation}
		Namely, let  $b\in \mathbf{x}M$. Then  by condition (\ref{itm:SeparatedCondition}),  $b\in (x_1^n,x_2,\ldots,x_d)M\backslash  (x_1^{n+1},x_2,\ldots,x_d)M$ for  some $n\in\mathbb{N}$. Assuming $b=x_1^nb_1+\sum\limits_{i=2}^dx_ib_i$,   we   get $\fp=(\mathbf{x}'M):_Sx_1^nb_1$. From this and the balanced big Cohen-Macaulayness of $M$   we get $\fp= (\mathbf{x}'M):_Sb_1$. On the other hand $b_1\notin (\mathbf{x})M$ 
		otherwise we get $b\in 
		(x_1^{n+1},x_2,\ldots,x_d)M$    contradicting  our choice of $n$.
		
		Since $\fp\subseteq(\mathbf{x}'M:_Sb)$ and $K_\bullet(x_1^w,x_2,\ldots,x_d;M)$ is acyclic so there is a chain map $$\theta_\bullet:G_\bullet\rightarrow K_\bullet(x_1^w,x_2,\ldots,x_d;M)$$ such that $\theta_0(e)=b$, and consequently $$\Delta_\bullet:=\theta_\bullet\circ\phi_\bullet:K_\bullet(x_1^w,x_2,\ldots,x_d;S)\rightarrow K_\bullet(x_1^w,x_2,\ldots,x_d;M)$$ is a chain map such that $\Delta_0(1)=b$ while 
		\begin{equation}
		\label{EquationSecondBadRelationInB}
		x_1^{w-1}\Delta_d(1)\in (x_1^{w},x_2,\ldots,x_d)M.
		\end{equation}
		
		In view of Lemma  \ref{LemmaBasicFactInHomologicalAlgebra}(\ref{itm:NullHomotopic}),   $\Delta_\bullet$ is homotopic to  $$\Psi_\bullet:K_\bullet(x_1^w,\ldots,x_d;S)\overset{\cong}{\rightarrow} K_\bullet(x_1^w,\ldots,x_d;S)\otimes_S S\overset{\text{id}\otimes (1\mapsto b)}{\longrightarrow} K_\bullet(x_1^w,\ldots,x_d;S)\otimes_SM\overset{\cong}{\rightarrow} K_\bullet(x_1^w,x_2,\ldots,x_d;M)$$ with $\Psi_d(1)=b$ and $\Psi_0(1)=b$  
		and   we get   $$\Delta_d(1)-\Psi_d(1)\in (x_1^w,x_2,\ldots,x_d)M$$ 
		(Lemma  \ref{LemmaBasicFactInHomologicalAlgebra}(\ref{itm:RelationInSyz})) and consequently $$x_1^{w-1}\Delta_d(1)-x_1^{w-1}\Psi_d(1)\in  (x_1^w,x_2,\ldots,x_d)M.$$
		
		From this and display (\ref{EquationSecondBadRelationInB}) we get $$ x_1^{w-1}b=x_1^{w-1}\Psi_d(1)\in (x_1^{w},x_2\ldots,x_d)M,$$ implying that $b\in (x_1,\ldots,x_d)M$ as $M$ is  Cohen-Macaulay. This contradicts  display (\ref{EquationbIsNotThere}).  	  	  
	\end{proof}
\end{lem}

\begin{rem}\label{RemarkGradeCanBeLargerThanHeight}
	\emph{Especially when $S$ is a local ring with non-trivial zero divisors, it is possible   to have a maximal Cohen-Macaulay $S$-module $M$ and a prime ideal $\fp$ of $S$ such that $\grade(\fp,M)\gneq \text{height}(\fp)$. For instance, let $\fm$ be the maximal ideal  of a regular local ring $A$ of positive dimension and consider the Rees algebra $A[\fm t]$ and let  $S$ to be the localization of $A[\fm t]/(x^2t)$  at $(\fm,\fm t)$ where $x\in \fm$ is a non-zero element. Then $A=S/(\fm t)$ is a maximal Cohen-Macaulay $S$-algebra, $\text{height}(\fm S)=\text{height}(\fm A[\fm t])-1=0$, while $\grade(\fm, S/(\fm t))=\grade(\fm,A)=\dim(A)$. Thus, when we are applying Lemma \ref{LemmaDesiredBalancedBigCMImpliesVariantOFCEC}  even for rings admitting a maximal Cohen-Macaulay $S$-module $M$, we should be  careful to be assured that the condition $\grade(\fp,M)=d-1(=\text{height}(\fp))$ is satisfied. However  if $S$ is a normal excellent domain, then   for any balanced big Cohen-Macaulay module $M$ and any prime ideal $\fp$ of $S$   we have $\text{grade}(\fp, M)=\text{height}(\fp)$ by Theorem \ref{TheoremBalancedBigCMModuleLocalizes}. Thus in case $S$ is an excellent normal domain  and $M$ is a balanced big Cohen-Macaulay $S$-module, we can apply Lemma \ref{LemmaDesiredBalancedBigCMImpliesVariantOFCEC}  without checking condition (\ref{itm:AssCondition}).}
\end{rem}

\begin{thm}\label{TheoremBalancedBigCMModuleLocalizes} (cf. \cite[Proposition 2.11]{BaMaPaScTuWaWi21})
      Let $(S,\fm_S)$ be an excellent local normal domain.  Let  $M$ be a balanced big Cohen-Macaulay $S$-module. Then for any $\fp\in \text{Spec}(S)$,  $M_{\fp}$ is a balanced big Cohen-Macaulay $S_{\fp}$-module and  $\text{grade}(\fp,M)=\text{height}(\fp)$. 
      \begin{proof}
      	Following \cite[Exercise 8.5.8]{BrunsHerzogCohenMacaulay}, we define the \textit{small support} of $M$ as $$\text{supp}(M):=\{\fp\in \text{Spec}(S):\fp M_{\fp}\neq M_\fp\}$$
      	while $\Supp(M)$ denotes the usual notion of support that is  $\Supp(M)=\{\fp\in \text{Spec}(S):M_{\fp}\neq 0\}$.
      	
      	First, we  notice that $\text{Supp}(M)=\text{Spec}(S)$, as $S$ is a domain and $M$ is a balanced big Cohen-Macaulay module. 
      	Pick some $\fp\in \text{Spec}(S)\backslash \{\fm_S\}$ (the statement holds for $\fm_S$).  In order to prove that  $M_\fp$ is a  balanced big Cohen-Macaulay  $R_{\fp}$-module and also to prove that $\grade(\fp,M)=\text{height}(\fp)$,	by  localization at a prime ideal $\fq$ containing $\fp$ with $\text{height}(\fq)=\text{height}(\fp)+1$, we can see that  it suffices to prove the statement under the assumption that $\text{height}(\fp)=\dim(S)-1$. Namely, for the equality $\grade(\fp,M)=\text{height}(\fp)$ we notice that $\text{height}(\fp)\overset{}{\le}\text{grade}(\fp,M)$ as $M$ is a balanced big Cohen-Macaulay (\cite[Proposition 9.1.2(b)]{BrunsHerzogCohenMacaulay}) and thus
      	\begin{align*}
      	  \text{height}(\fp)\overset{}{\le}\text{grade}(\fp,M)&\overset{}{\le} \text{grade}(\fp S_{\fq},M_{\fq}), && (\text{definition of grade})&\\&\overset{}{=}  \text{height}(\fp S_{\fq}), && (\text{statement holds over }S_{\fq}  (\text{height}(\fq)=\text{height}(\fp)+1))&\\&=\text{height}(\fp).
      	\end{align*}  

      	Hence, without loss of generality, we may and we do assume that $\text{height}(\fp)=\dim(S)-1$. Since $S$ is an excellent normal domain, thus so  is its completion $\widehat{S}$, therefore we can appeal to \cite[8.2.1 The Lichtenbaum–Hartshorne Vanishing Theorem]{BrodmannSharpLocal} to observe that $H^{\dim(S)}_{\fp}(S)=0$ and thence $H^{\dim(S)}_{\fp}(M)=0$, therefore we know that 
      	\begin{equation}
      	   \label{EquationVanishingOfLocalCohomologiesAtPForGreaterOrEqualThanD}
      	    \forall\ i\ge \dim(S),\ \ \     	H^i_{\fp}(M)=0
      	\end{equation}
      	by  Grothendieck's Vanishing Theorem. 
      	
      	Since $H_0(\fp,M)=M/\fp M\neq 0$, so $\text{grade}(\fp,M)<\infty$ by  definition of grade (Definition \ref{DefinitionGradedKoszul}). We show that $\text{grade}(\fp,M)=\text{height}(\fp)$.  If  $\text{grade}(\fp,M)\ge \dim(S)$, then Lemma \ref{LemmaInjectivityOfMapTOLocalCohomologyAtGrade}(\ref{itm:LocalCohomologiesBelowTheGradeAllVanish}) in conjunction with display (\ref{EquationVanishingOfLocalCohomologiesAtPForGreaterOrEqualThanD}) implies that $H^i_{\fp}(M)=0$ for all $i$. Then appealing to \cite[Corollary 1.4]{SchenzelOnTheUse} for some (or any) $x\in \fm_S\backslash \fp$, from $H^*_{\fp}(M)=0$ we conclude that $H^i_{\fm_S}(M)=0$ for all $i$ (we recall that $\text{height}(\fp)=\dim(S)-1$) which is a contradiction (\cite[Exercise 9.1.12(a) and (c)]{BrunsHerzogCohenMacaulay}).
      	 Therefore, $\grade(\fp,M)\le \dim(S)-1$. On the other hand $\grade(\fp,M)\ge \dim(S)-1$ in view of \cite[Proposition 9.1.2(b)]{BrunsHerzogCohenMacaulay}, because $\fp$ contains a part of a system of parameters of length $\dime(S)-1$ which has to be a regular sequence on $M$. So $\grade(\fp,M)=\dim(S)-1=\text{height}(\fp)$ and thence Lemma \ref{LemmaInjectivityOfMapTOLocalCohomologyAtGrade}(\ref{itm:ConditionSufficientForLocalCohomologyNonZeroAtGrade}) and (\ref{itm:LocalCohomologiesBelowTheGradeAllVanish}) implies that $H^{\dim(S)-1}_{\fp}(M)$ is the only non-zero local cohomology of $M$ supported at $\fp$.

      	Let $h\in S\backslash \fp$.	      In view of \cite[Corollary 1.4]{SchenzelOnTheUse},   $H^0_{hS}\big(H^{\dim(S)-1}_{\fp}(M)\big)$ is a homomorphic image of $H^{\dim(S)-1}_{\fm_S}(M)$ while the latter module is zero because of the big Cohen-Macaulayness of $M$. Therefore  $H^0_{hS}\big(H^{\dim(S)-1}_{\fp}(M)\big)=0$, i.e. the non-zero module $H^{\dim(S)-1}_{\fp}(M)$ (non-zero because of the conclusion of the previous paragraph) has no non-zero $h$-torsion element and thence $H^{\dim(S)-1}_{\fp S_{h}}(M_h)\neq 0$. From this fact together the maximality of    $\fp S_h\in \text{Spec}(S_h)$ we get $\Supp\big(H^{\dim(S)-1}_{\fp S_{h}}(M_h)\big)=\{\fp S_h\}$ and consequently $H^{\dim(S)-1}_{\fp S_\fp}(M_\fp)\neq 0$. Then the statement follows from \cite[Exercise 8.5.9]{BrunsHerzogCohenMacaulay} and \cite[Theorem 4.3]{SharpACousin}.
	\end{proof}
\end{thm}

The condition (\ref{itm:SeparatedCondition}) in Lemma \ref{LemmaDesiredBalancedBigCMImpliesVariantOFCEC} is available in view of the following remark.

\begin{rem}\label{RemarkBalancedBCMofCompletion}
	\emph{Let $(S,\fm_S)$ be a local ring and $M$ be a  big Cohen-Macaulay module (with respect to any system of parameters).}
	\begin{enumerate}[(i)]
		\item \label{itm:CompletionBigCM}  \emph{The $\fm_S$-adic completion $\widehat{M}$ of $M$ is  a balanced big Cohen-Macaulay module (\cite[Corollary  8.5.3]{BrunsHerzogCohenMacaulay})}.
		\item \label{itm:CompletionHasSeparatedQuotients} 
		\emph{ If $M$ is a balanced big Cohen-Macaulay module, then for each part of a system of parameters $\mathbf{x}'$ of $S$ we have the isomorphism $\widehat{M}/\mathbf{x}'\widehat{M}\cong \widehat{M/\mathbf{x}'M}$ by \cite[Theorem 5.2.3]{StrookerHomological}, so   $\widehat{M}/\mathbf{x}'\widehat{M}$ is $\fm_S$-adically complete  (see also \cite[Theorem 5.1.7]{StrookerHomological} and \cite[Definition 5.1.13]{StrookerHomological}).}
	\end{enumerate} 
\end{rem}

Now we show all local rings satisfy the Strong Canonical Element Property. 
To this aim, we apply the existence of    balanced  big Cohen-Macaulay algebras  thanks to \cite{HochsterHunekeInfinite}, \cite{HochsterHunekeApplications} and \cite{AndreWeakFunctoriality}. 
In particular, the next corollary  in this section and its  characteristic dependent  proof (induced by the characteristic dependent proof of the existence of balanced big Cohen-Macaulay algebras) have no relation with the content of the next section.  So the characteristic free nature of the results of  other sections, including the main results of our paper given in the next section, are  unaffected.

\begin{cor}\label{CorollaryVariantHolds}
  Let $(S,\fm_S)$ be a local ring. Then $S$ satisfies the Strong Canonical Element Property.
  \begin{proof}
 Let $d:=\dim(S)$ and pick some  $\fq,\mathbf{x},w,G_\bullet$ as in Consideration \ref{Consideration} as well as some chain map $\phi_\bullet:K_\bullet(x_1^w,x_2,\ldots,x_d;S)\rightarrow G_\bullet$ lifting the natural surjection on $0$-th homologies. We  show that Proposition-Definition \ref{PropositionVarientCET}(\ref{itm:VariantCEProperty}) holds for  $\fq,\mathbf{x},w,G_\bullet$ and $\phi_\bullet$. 
 
 By Remark \ref{RemarkVariantUnderLocalHomomorphism}, it suffices to prove the statement for the case where $S$ is complete. Also,   without loss of generality we may  and we do  assume that $\fq$ is prime.  Since $\text{height}(\mathfrak{\fq})=d-1$, so we automatically have $\text{height}(\mathfrak{q})+\dim(S/\mathfrak{q})=d$ and thus there exists a minimal prime ideal  $\fp_0$  of $S$ contained in $\fq$ such that $\dim(S/\fp_0)=\dim(S)$. Therefore, in view of Remark \ref{RemarkVariantUnderLocalHomomorphism}, by passing to $S/\fp_0$, without loss of generality we can assume that $S$ is a  domain. Then, similarly, by passing to the integral closure $\overline{S}$  of $S$  in $\Frac(S)$ and then replacing $\fq$ by a prime ideal of $\overline{S}$ lying over $\fq$, without loss of generality, we can assume that $S$ is a complete local normal domain.

  		Let $B$ be an $\fm_S$-adically complete balanced big-Cohen-Macaulay $S$-algebra  which exists in view of   Remark \ref{RemarkBalancedBCMofCompletion}(\ref{itm:CompletionBigCM}) as well as \cite{HochsterHunekeInfinite},  \cite{HochsterHunekeApplications} and \cite{AndreWeakFunctoriality} (see also \cite[Tag 05GG(1)]{Stacks}).    		
  		 Then by Theorem \ref{TheoremBalancedBigCMModuleLocalizes}, or \cite[Proposition 2.11]{BaMaPaScTuWaWi21}, we observe that $B_{\fq}$ is a balanced big Cohen-Macaulay $S_{\fq}$-algebra and $\text{grade}(\fq,B)=d-1$.  Moreover, $B/(x_2,\ldots,x_d)B$ is $\fm_S$-adically separated by Remark \ref{RemarkBalancedBCMofCompletion}(\ref{itm:CompletionHasSeparatedQuotients}). Therefore the statement follows from Lemma \ref{LemmaDesiredBalancedBigCMImpliesVariantOFCEC}.
  \end{proof}
\end{cor}

\section{The main results}\label{SectionMainResults}

 Applying the results of the previous sections, in this section  we  show that a possible  deduction of   our variant of the Canonical Element Theorem   from the original Canonical Element Theorem provides us with a characteristic free proof  of the Canonical Element Theorem. The same conclusion will be shown under an alternative assumption on the stability of  the CE Property   with respect to  the localization over excellent factorial domains.  
  Another result of this section is Corollary \ref{CorollaryBigCohenMacaulay}  which suggests an approach for settling the Balanced Big Cohen-Macaulay Module Conjecture by a characteristic free proof.
 
 
\begin{thm} 
	\label{MonomialTheoremCharacteristicFree}Let $(R,\fm_R)$ be a local ring.	Then $R$ satisfies   Hochster's Canonical Element Conjecture (followed by a characteristic free proof) provided either of the following assertions holds (in a characteristic free way):
	\begin{enumerate}[(i)]
		\item \label{itm:SufficientConditionExcellentUFD} For a fixed excellent factorial local domain that is a homomorphic image of a regular local ring, equivalent parts of Proposition-Definition \ref{PropositionVarientCET} hold    provided  the original Canonical Element Conjecture holds.
		\item\label{itm:SufficentConditionLocalization} If a fixed excellent factorial local domain that is a homomorphic image of a regular local ring satisfies the Canonical Element Conjecture, then so does the localization of it at any prime ideal.
	\end{enumerate}
\end{thm}
\begin{proof}
 We prove that $R$ satisfies the CE property provided either of the assertions   (\ref{itm:SufficientConditionExcellentUFD}) or (\ref{itm:SufficentConditionLocalization}) holds.	We argue by induction on  $n:=\dim(R)$.

 Since the statement   is immediate in  dimension $\le 2$, so we assume that $n\ge 3$ and the statement has been proved for dimension less than $n$.

\textbf{Step 1: Reduction to  complete quasi-Gorenstein domains that are   locally complete intersection in codimension $\le 1$.} 
By \cite[(3.9) Remark]{HochsterCanonical},   we can reduce to the case where $R$ is a complete local normal domain. Since $R$ is a complete local domain,   $R$ admits a canonical module $\omega_R$ which is an ideal of $R$ (\cite[(3.1)]{AoyamaGotOnTheEndomorphism}). Then we proceed as follows:

\begin{enumerate}[(i)]
	\item\label{itm:SquareRoot}  Using \cite[Lemma 2.2]{TavanfarReduction}, we choose $a\in \fm_R$  which has no square root in $\Frac(R)$.  Then   the $R$-algebra $R(a^{1/2}):=R[X]/(X^2-a)$  is an integral domain  (see \cite[Remark 2.1]{TavanfarReduction}). The element $a^{1/2}:=X+(X^2-a)$ is a square root of $a$ in $R(a^{1/2})$.
	
	\item\label{itm:RDsumW}  
	  Here we follow \cite[Remark 2.3]{TavanfarReduction}. Let $\mathfrak{k}$ be the kernel of the
	  $R$-algebra homomorphism $R[\omega_RX]\hookrightarrow  R[X]\twoheadrightarrow R[X]/(X^2-a)$, where $R[\omega_RX]$  is the Rees algebra of the ideal $\omega_R$. Then $R[\omega_R X]/\mathfrak{k}$ is a domain by part (i), as it is a subring of $R(a^{1/2})$. Since $X+(X^2-a)=a^{1/2}$ in $R(a^{1/2})$, we denote 
	  $\omega_RX+\mathfrak{k}$ (respectively  $r+wX+\mathfrak{k}$, where $r\in R,\ w\in \omega_R$)    by $\omega_Ra^{1/2}$ (respectively  $r+wa^{1/2}$). 
	 
	 \item\label{itm:ModuleFinite} 
	  We endow $R\oplus \omega_R\subseteq   R\oplus R$ with an $R$-algebra structure  by which it is a subalgebra of $R(a^{1/2})\overset{\text{as\ }R\text{\ modules}}{\cong}  R\oplus  R.$ It is easily seen that the $R$-module homomorphism $$R\oplus \omega_R\rightarrow R[\omega_RX]/\mathfrak{k},\ \ (r,w)\in R\oplus \omega_R\mapsto r+wa^{1/2}$$ is an isomorphism of $R$-modules inducing an $R$-algebra structure on $R\oplus \omega_R$ via $$(r,w)(r',w'):=(rr'+ww'a,rw'+r'w).$$ In particular, $R[\omega_RX]/\mathfrak{k}$ is a module finite extension of $R$, so it is also a complete local domain by \cite[Theorem 7]{CohenOnTheStructure}. It is easily verified that $\fm_R+\omega_Ra^{1/2}$ is the unique maximal ideal of $R[\omega_RX]/\mathfrak{k}$.
	 
	 \item \label{itm:CEPropertyandModuleFinite}
	   Since $R\rightarrow R[\omega_RX]/\mathfrak{k}$ is a module finite  extension  in view of part (iii) (thus an integral extension by \cite[Theorem 9.1(i)]{Matsumura}),  Lemma \ref{RingHomomorphismANDCEC} implies that $R$ satisfies the CE Property provided  that  $R[\omega_RX]/\mathfrak{k}$ does so. Here we use that $\dim(R[\omega_RX]/\mathfrak{k})=\dim(R)=n$ by \cite[Exercise 9.2, page 69]{Matsumura} and that $\fm_R(R[\omega_RX]/\mathfrak{k})$ is primary to the maximal ideal of $R[\omega_RX]/\mathfrak{k}$ by  \cite[Lemma 2, page 66]{Matsumura}. Thus any system of parameters for $R$ is also a system  of parameters for $R[\omega_RX]/\mathfrak{k}$.
 	 \item \label{itm:LCICodimension1}
 	   By \cite[Lemma 2.5(ii)]{TavanfarReduction}, $R[\omega_RX]/\mathfrak{k}$ is a locally complete intersection in codimension $1$. 
	 \item \label{itm:QGorenstein}
	   The complete local domain $R[\omega_RX]/\mathfrak{k}$ is quasi-Gorenstein. The proof of this fact is written in the proof of \cite[Proposition 2.7]{TavanfarReduction}.
\end{enumerate}

In view of parts (\ref{itm:CEPropertyandModuleFinite}), (\ref{itm:LCICodimension1}) and (\ref{itm:QGorenstein}), we can replace $R$ by the $n$-dimensional domain $R[\omega_RX]/\mathfrak{k}$ and we hereafter assume that $R$ is a complete quasi-Gorenstein local domain that is a locally complete intersection in codimension $\le 1$\ \footnote{In the case where $R$ is a Cohen-Macaulay  local  domain with canonical ideal $\omega_R$, such an $R$-algebra structure on $R\oplus \omega_R$ is considered in \cite[Definition 1.1 and Proposition 1.2]{EnescuApplications} where this $R$-algebra is called as a pseudocanonical double cover of $R$, and it is proved that the pseudocanonical double cover is a Gorenstein local ring. However,   the author of the present paper became aware of this $R$-algebra structure on $R\oplus \omega_R$   by his own experience and later he   found  a similar construction in \cite{EnescuApplications}.}. Here, we also remark that $R$ is not normal, because our involved quasi-Gorenstein domain extension loses the normal property even though it is a locally complete intersection in codimension $\le 1$. 
	
\textbf{Step 2: The generic linkage and the setting.}	By the previous arguments, we can assume that $R$ is a (non-Cohen-Macaulay) complete quasi-Gorenstein domain which is a locally complete intersection in codimension $\le 1$. Then the Huneke and Ulrich generic linkage is applied to enable us to get a prime almost complete intersection ideal $\fa$ of an excellent regular local ring $A$ such that $A/\fa$ is a locally complete intersection in codimension $\le 1$ and it is linked to a trivial deformation of $R$. Namely, following   \cite{HunekeUlrichDivisorClass} and \cite{UlrichGorenstein}, we proceed as follows.
	
	We  assume that $R:=B/\fb'$ where $(B,\fm_B,K_B)$ is a complete regular local ring and that $\fb'$ is  a grade $g$ ideal of $B$ which is generated by $u$ elements $b'_1,\ldots,b'_u$. Then we consider the indeterminates $$\mathbf{X}:=X_{1,1},\ldots,X_{1,u},X_{2,1},\ldots,X_{2,u},\ldots,X_{g,1},\ldots,X_{g,u}$$ over $B$
	and the sequence $$a_1:=\sum\limits_{j=1}^uX_{1,j}b'_j\ ,\ a_2:=\sum\limits_{j=1}^uX_{2,j}b'_j\ ,\ \ldots\ ,\ a_g:=\sum\limits_{j=1}^uX_{g,j}b'_j.$$ Then $a_1,\ldots,a_g$ is a maximal regular sequence contained in $\fb'B[\mathbf{X}]$ (see \cite{HochsterProperties}). Set   $$\fa':=(a_1,\ldots,a_g):\fb'B[\mathbf{X}].$$  
	It is said that $\fa'$ is a generic linkage of $\fb'$ (\cite[Definition 2.3]{HunekeUlrichDivisorClass}).
	Since $B/\fb'\rightarrow B[\mathbf{X}]_{(\fm_B,\mathbf{X})}/\fb' B[\mathbf{X}]_{(\fm_B,\mathbf{X})}$ is a flat local homomorphism with regular closed fiber $K_B[\mathbf{X}]_{(\mathbf{X})}$, so $B[\mathbf{X}]_{(\fm_B,\mathbf{X})}/\fb' B[\mathbf{X}]_{(\fm_B,\mathbf{X})}$ is also (a non-Cohen-Macaulay) quasi-Gorenstein domain by \cite[Theorem 4.1]{AoyamaGotOnTheEndomorphism}\footnote{In view of \cite[Corollary 2.8]{TavanfarTousiAStudy}, this quasi-Gorenstein property can also  be deduced by an iterated use of Lemma \ref{LemmaIndeterminateActsSurjectivelyonLocalCohomology}. Also, one can deduce this fact from Lemma \ref{RemarkLinkageOfQuasiGorensteinIdealsAndAlmostCompleteIntersectionIdeals}(\ref{itm:ACIIsLinkedToQG}), because the faithfully flat extension $B\rightarrow B[\mathbf{X}]_{(\fm_B,\mathbf{X})}$ preserves linkage of ideals, as well as the unmixedness  and the almost complete intersection property of ideals.}. Consequently, $\fa' B[\mathbf{X}]_{(\fm_B,\mathbf{X})}$  is an unmixed almost complete intersection ideal by Lemma \ref{RemarkLinkageOfQuasiGorensteinIdealsAndAlmostCompleteIntersectionIdeals}(\ref{itm:LinkageOfQGrenseteinIsAlmostCI}).
	By virtue of \cite[Proposition 2.6]{HunekeUlrichDivisorClass}   $\fa'$ is a prime ideal, and $B[\mathbf{X}]/\fa'$  is a locally complete intersection in codimension $\le 1$ in  light of \cite[Proposition 2.9(b)]{HunekeUlrichDivisorClass}. Hence    $\fa' B[\mathbf{X}]_{(\fm_B,\mathbf{X})}$ is also a prime ideal whose residue ring is  a locally complete intersection in codimension $\le 1$.  Set 
	\begin{equation}
	\label{EquationNotationForGenericLinkedACIObtained}
	(A,\fm_A):=B[\mathbf{X}]_{(\fm_B,\mathbf{X})},\ \ \fc:=(a_1,\ldots,a_g)A,\ \  \fb:=\fb'A,\ \ \fa:=\fa'A\overset{\text{Lemma \ref{RemarkLinkageOfQuasiGorensteinIdealsAndAlmostCompleteIntersectionIdeals}(\ref{itm:LinkageOfQGrenseteinIsAlmostCI})}}{=}\fc+(h),\  (\text{for\ some\ }h).
	\end{equation} 
	So $A$ is an excellent regular local ring, and $\fa,\fb$ are linked ideals of $A$ over the complete intersection $\fc$ (see Lemma \ref{RemarkLinkageOfQuasiGorensteinIdealsAndAlmostCompleteIntersectionIdeals}(\ref{itm:LinkageOfQGrenseteinIsAlmostCI})).
	 	
	\textbf{Step 3: Collecting  some depth identities.} Since $A/\fb$ is a trivial deformation of $R=B/\fb'$  (by the image of $\mathbf{X}$ in $A/\fb$), so  setting $m:=\dim(A/\fb)(=\dim(A/\fa))$ 
	\begin{equation}
	\label{DepthQuasiGorenstein}
	\depth(A/\fb)=\dim(A/\fb)-\dim(R)+\depth(R)\ge m-n+2
	\end{equation} 
	($R(\cong \omega_R)$ is quasi-Gorenstein and it has depth $\ge 2$ because the canonical module is always $(S_2)$ and $\dim(R)\ge 3$ by our assumption). Consequently, 
	\begin{equation}
	\label{DepthCanonicalModuleAlmostCompleteIntersection}
	\depth_A(\omega_{A/\fa})\overset{\text{Eq.\ (\ref{EquationCanonicalModule})}}{=}
	\depth_A\big((\fc:\fa)/\fc\big)=
	\depth_A(\fb/\fc)=\depth(A/\fb)+1\ge m-n+3
	\end{equation}
	where the third equality follows from $0\rightarrow \fb/\fc\rightarrow A/\fc\rightarrow A/\fb\rightarrow 0$ and the fact that $A/\fc$ is complete intersection (thus Cohen-Macaulay) while $A/\fb$ is not Cohen-Macaulay. Furthermore  
	\begin{equation}
	\label{DepthAlmostCompleteIntersection}
	\depth(A/\fa)=\depth(A/\fb)-1\ge m-n+1.
	\end{equation}
	in view of  display (\ref{DepthQuasiGorenstein})  and Lemma  \ref{GoodRegularSequenceCorollary}(\ref{itm:DepthIdentity}).
	
	
	
	\textbf{Step 4: Proving the   CE Property for $A/\fa$.} 
	It suffices to find a regular sequence $\mathbf{y}:=y_1,\ldots,y_{m-n+1}$ of $A/\fa$ such that 
	\begin{equation}
	\label{DesiredAttachedPrimeHypothesis}
	y_{i+1}\notin \bigcup\limits_{\fp\in \text{Att}\big(H^{m-i-1}_{\fm_A}\big(A/(\fa,y_1,\ldots,y_{i})\big)\big)}\fp
	\end{equation}
	for each $0\le i\le m-n$. Because  then $\dim\big(A/(\fa,\mathbf{y})\big)=n-1$  implying that $A/(\fa,\mathbf{y})$ satisfies the CE property by the induction hypothesis and then $A/\fa$ satisfies the CE Property in view of Lemma \ref{DeformationOFCEC}(\ref{itm:CETandDeformation}).
	
	We prove the existence of the sequence $\mathbf{y}$ by induction. By display (\ref{DepthCanonicalModuleAlmostCompleteIntersection}) we have $\depth(\omega_{A/\fa})\gneq 2$, so Lemma \ref{AttLocalCohomologyDepthCanonicalModule} implies that $\fm_{A}\notin \text{Att}\big(H^{m-1}_{\fm_{A}}(A/\fa)\big)$ and we can pick some $$y_1\in \fm_{A}\backslash \big((\bigcup\limits_{\fp\in \text{Att}\big(H^{m-1}_{\fm_A}\big(A/(\fa)\big)\big)}\fp) \bigcup \fb \bigcup  (\bigcup\limits_{\fp\in \text{ass}(a_1,\ldots,a_g)}\fp)\big).$$
	Denoting the residue class of $y_1$ in $A/\fa$ with the same notation $y_1$, it is automatically a regular element as $A/\fa$ is a domain. Moreover,  from the exact sequence 
	$0\rightarrow A/\fb\overset{h}{\rightarrow} A/\fc\rightarrow A/\fa\rightarrow 0$ 
	(see  display (\ref{EquationNotationForGenericLinkedACIObtained}) for the notation $h$) we get the exact sequence $$0\overset{y_1\text{\ is regular}}{=}\text{Tor}^A_1(A/\fa,A/(y_1))\rightarrow A/(\fb,y_1)\overset{h}{\rightarrow} A/(\fc,y_1)\rightarrow A/(\fa,y_1)\rightarrow 0 $$ 
	which implies that $(\fc,y_1):(\fa,y_1)\overset{\fa=\fc+(h)}{=}(\fc,y_1):h\overset{}{=}(\fb,y_1)$. Therefore $$\omega_{A/(\fa,y_1)}\overset{\text{Eq. (\ref{EquationCanonicalModule})}}{=}\big((\fc,y_1):(\fa,y_1)\big)/(\fc,y_1)\overset{}{=}(\fb,y_1)/(\fc,y_1)$$ (note that $(\fc,y_1)$ is a (maximal) complete intersection ideal contained in $(\fa,y_1)$ by our choice of $y_1$). In particular, 
	\begin{align*}
	\depth(\omega_{A/(\fa,y_1)}) 
	&= \depth\big((\fb,y_1)/(\fc,y_1)\big)  &&
	&\\&
	= \depth\big(A/(\fb,y_1)\big)+1               && ((\fb,y_1)/(\fc,y_1) \text{\ is the first syzygy of }  \underset{\text{non-CM}}{\underbrace{A/(\fb,y_1)}}\text{\ over } \underset{\text{c.i.}}{\underbrace{A/(\fc,y_1)}})
	&\\&=
	\depth(A/\fb) && (y_1\notin \text{Z}(A/\fb))
	&\\&\overset{}{\ge} m-n+2 && \text{Eq.\ } (\ref{DepthQuasiGorenstein}).
	\end{align*}
	Suppose, inductively, that for some $1\le j<m-n+1$ we have found a   sequence $y_1,\ldots,y_{j}$ in $A$ which forms a regular sequence on each of $A/\fc$, $A/\fb$ and  $A/\fa$  such that  it satisfies display (\ref{DesiredAttachedPrimeHypothesis}) for each $0\le i< j$,   such that $\depth(\omega_{A/(\fa,y_1,\ldots,y_j)})\ge m-n+3-j$, and finally such that $$(\fc,y_1,\ldots,y_j):(\fa,y_1,\ldots,y_j)=(\fc,y_1,\ldots,y_j):h=(\fb,y_1,\ldots,y_j).$$
	
	Since $j<m-n+1$, so $m-n+3-j>2$. Therefore, again Lemma \ref{AttLocalCohomologyDepthCanonicalModule} implies that $\fm_A\notin \text{Att}\big(H^{m-j-1}_{\fm_{A}}\big(A/(\fa,y_1,\ldots,y_j)\big)\big)$ and we may, and we do, pick some $y_{j+1}\in \fm_A$ such that $y_{j+1}$ does not belong to $$(\bigcup\limits_{\fp\in \text{Att}\big(H^{m-j-1}_{\fm_A}\big(A/(\fa,y_1,\ldots,y_j)\big)\big)}\fp) \bigcup (\bigcup\limits_{\fp\in \text{ass}(\fb,y_1,\ldots,y_j)}\fp) \bigcup  (\bigcup\limits_{\fp\in \text{ass}(\fc,y_1,\ldots,y_j)}\fp)\bigcup  (\bigcup\limits_{\fp\in \text{ass}(\fa,y_1,\ldots,y_j)}\fp).$$
	This is possible in view of displays (\ref{DepthQuasiGorenstein}) and (\ref{DepthAlmostCompleteIntersection}). Therefore the exact sequence 
	$$0\rightarrow A/\big((\fc,y_1,\ldots,y_j):h\big)\overset{h}{\rightarrow}A/(\fc,y_1,\ldots,y_j)\rightarrow A/(\fa,y_1,\ldots,y_j)\rightarrow 0$$
	yields the exact sequence 
	$$0\rightarrow A/\big((\fc,y_1,\ldots,y_j):h,y_{j+1}\big)\overset{h}{\rightarrow}A/(\fc,y_1,\ldots,y_{j+1})\rightarrow A/(\fa,y_1,\ldots,y_{j+1})\rightarrow 0$$
	because $\text{Tor}^A_1\big(A/(\fa,y_1,\ldots,y_{j}),A/(y_{j+1})\big)=0$ by our choice of $y_{j+1}$. From this and our induction hypothesis, we get 
	\begin{align*}
	\depth\big(\omega_{A/(\fa,y_1,\ldots,y_{j+1})}\big)
	&\overset{\text{}}{=}\depth\big((\fb,y_1,\ldots,y_{j+1})/(\fc,y_1,\ldots,y_{j+1})\big) && (\text{display\ } (\ref{EquationCanonicalModule}))
	&\\&=\depth\big(A/(\fb,y_1,\ldots,y_{j+1})\big)+1
	&\\&\overset{}{=}\depth\big(A/(\fb,y_1,\ldots,y_{j})\big) && (y_{j+1}\notin \text{Z}\big(A/(\fb,y_1,\ldots,y_{j})\big))
	&\\& = \depth\big((\fb,y_1,\ldots,y_{j})/(\fc,y_1,\ldots,y_{j})\big)-1
	&\\& = \depth(\omega_{A/(\fa,y_1,\ldots,y_j)})-1
	&\\& \ge m-n+3-(j+1).
	\end{align*}
	So our inductive argument is done, the desired regular sequence $y_1,\ldots,y_{m-n+1}$ exists and the CE Property holds for $A/\fa$.

	\textbf{Step 5: Proving the   CE Property for  $A[\fa t,t^{-1}]_{\mathfrak{M}}$ in general, and thus the Strong Canonical Element Property for $A[\fa t,t^{-1}]_{\mathfrak{M}}$ under the validity of the  assertion (\ref{itm:SufficientConditionExcellentUFD}).} 
		By   Step 4 together with  Lemma \ref{LemmaIndeterminateActsSurjectivelyonLocalCohomology} and Lemma \ref{DeformationOFCEC}(\ref{itm:CETandDeformation}), $(A/\fa[X])_{(\fm_A,X)}$ also satisfies the CE Property. From this fact together with Proposition \ref{StructureProposition}(\ref{itm:QuotientOfEReesAlgebraIsACI})  we conclude that
	\begin{equation}
	\label{Equality}
	A[\fa t,t^{-1}]_{\mathfrak{M}}/(a_1t,\ldots,a_gt,ht)\cong\big((A/\fa)[X]\big)_{(\fm_A,X)}
	\end{equation}
	satisfies the CE Property.  Since $\text{height}_A(\fa)\overset{\text{Remark  \ref{RemarkLinkedSoTheSameHeight}}}{=}\text{height}_A(\fb)=\text{height}_B(\fb')=g$, so $$\dim\big((A/\fa[\mathbf{X}])_{(\fm_A,X)}\big)=\dim(A)-g+1$$ and also $$\dim\big(A[\fa t,t^{-1}]_{\mathfrak{M}}/(a_1t,\ldots,a_gt)\big)=\dim(A)+1-g$$ as $a_1t,\ldots,a_gt$ is a (homogeneous) regular sequence   on $A[\fa t,t^{-1}]$ by Proposition \ref{StructureProposition}(\ref{itm:HomogeneousMaximalRegularSequenceOfEReesAlgebra}). So, display (\ref{Equality}) shows that $A[\fa t,t^{-1}]_{\mathfrak{M}}/(a_1t,\ldots,a_gt)$ has the same dimension as its quotient $A[\fa t,t^{-1}]_{\mathfrak{M}}/(a_1t,\ldots,a_gt,ht)$, implying that the CE Property holds for $A[\fa t,t^{-1}]_{\mathfrak{M}}/(a_1t,\ldots,a_gt)$   by Lemma \ref{RingHomomorphismANDCEC}.  But by Proposition \ref{StructureProposition}(\ref{itm:SpecializationOfEReesAlgebraIsS2}), the regular sequence $a_1t,\ldots,a_gt$ satisfies  the condition  $$a_{i+1}t\notin \bigcup\limits_{\fp\in \text{Att}\big(H_{\mathfrak{M}}^{d-i}({A[\fa t,t^{-1}]}_{\mathfrak{M}}/(a_1t,\ldots,a_it))\big)}\fp$$ for each $0\le i\le g-1$. In particular, a repeated use of Lemma \ref{DeformationOFCEC}(\ref{itm:CETandDeformation}) implies that $A[\fa t,t^{-1}]_{\mathfrak{M}}$ also fulfills the CE Property. Thus, since it is already known that  $A[\fa t,t^{-1}]_{\mathfrak{M}}$ is an excellent factorial domain  and it is a homomorphic image of a regular local ring, so from the validity of assertion (\ref{itm:SufficientConditionExcellentUFD})   we  can conclude that $A[\fa t,t^{-1}]_{\mathfrak{M}}$ satisfies  the Strong Canonical Element Property.
	
	\textbf{Step 6: Collecting some sequences and identities:} Now we turn our attention to $A[\fa t,t^{-1}]_{ht}$.   Let $\mathbf{z}''$ be any system of parameters for $A/\fa$, and lift it to a system of parameters $\mathbf{z}'$ for $A/\fc$ (such a lift of $\mathbf{z}''$ can be found applying the Davis' prime avoidance lemma as $\dim(A/\fc)=\dim(A/\fa)$).  We fix a lift of $\mathbf{z}'$ in $A$, that is denoted  by abuse of notation, again with  $\mathbf{z}'$. Set $\mathfrak{P}:=(\fm_A,a_1t,\ldots,a_gt,t^{-1})$, that is a height $d:=\dim(A)$ prime ideal (Proposition \ref{StructureProposition}(\ref{itm:BeingPrimeIdealHeightd})). We claim that, $\mathbf{z}',a_1t,\ldots,a_gt$ is a sequence of length $d$ such that  
	\begin{equation}
	\label{GoodHomogeneousSystemOfParameters}
	\sqrt{(\mathbf{z}',a_1t,\ldots,a_gt)A[\fa t,t^{-1}]_{ht}}=\mathfrak{P}A[\fa t,t^{-1}]_{ht}.
	\end{equation} 
	
	Namely, we have $(\mathbf{z}',a_1t,\ldots,a_gt)A[\fa t,t^{-1}]=(\mathbf{z}',\fc,a_1t,\ldots,a_gt)A[\fa t,t^{-1}]$ consequently it contains $(\fm_A^u,a_1t,\ldots,a_gt)$ for some $u$ ($\mathbf{z}'$ forms a system of parameters for $A/\fc$). So 
	\begin{equation}
	  \label{EquationRadical}
	  \sqrt{(\mathbf{z}',a_1t,\ldots,a_gt)A[\fa t,t^{-1}]}=\sqrt{(\fm_A,a_1t,\ldots,a_gt)A[\fa t,t^{-1}]}.
	\end{equation}
	 Moreover, $(\fm_A,a_1t,\ldots,a_gt)A[\fa t,t^{-1}]_{ht}=\mathfrak{P}A[\fa t,t^{-1}]_{ht}$ because $t^{-1}=h/ht\in (\fm_A,a_1t,\ldots,a_gt)A[\fa t,t^{-1}]_{ht}$.  Hence display (\ref{GoodHomogeneousSystemOfParameters}) holds.  Set $$\mathbf{z}:=\mathbf{z}',a_1t/ht,\ldots,a_gt/ht,$$   a sequence of elements of degree zero of $A[\fa t,t^{-1}]_{ht}$.
	
	
	In addition to $\mathbf{z}$, there is a possibly inhomogeneous sequence $\mathbf{f}:=f_1,\ldots,f_d$ of elements of $\mathfrak{P}$ such that  $$\sqrt{(f_1,\ldots,f_d,ht)A[\fa t,t^{-1}]_{\mathfrak{M}}}=(\mathfrak{M}),$$ so $f_1,\ldots,f_d,ht$ forms a system of parameters for $A[\fa t,t^{-1}]_{\mathfrak{M}}$ (this is possible as $\mathfrak{M}=(\mathfrak{P},ht)$ and $ht$ is a parameter element of the domain $A[\fa t,t^{-1}]_{\mathfrak{M}}$).

	\textbf{Step 7: Proving an analogue of the CE Property  for $A[\fa t,t^{-1}]_{ht}$ and the sequence $f_1,\ldots,f_d$.} Let $(F_\bullet,\partial^{F_\bullet}_\bullet)$ be a minimal graded free resolution of $A[\fa t,t^{-1}]/\mathfrak{P}$.  
	 Fix some chain map $$\phi_\bullet:K_\bullet(\mathbf{f};A[\fa t,t^{-1}])\rightarrow F_\bullet$$ lifting the epimorphism $A[\fa t,t^{-1}]/(\mathbf{f})\rightarrow A[\fa t,t^{-1}]/\mathfrak{P}$. 	Passing from $A[\fa t,t^{-1}]$ to its localization $A[\fa t,t^{-1}]_{ht}$, this induces the chain map $$(\phi_\bullet)_{ht}:K_\bullet(\mathbf{f};A[\fa t,t^{-1}]_{ht})\rightarrow {F_\bullet}_{ht}$$ for which we claim that    (under the validity of either  assertion (\ref{itm:SufficientConditionExcellentUFD}) or assertion (\ref{itm:SufficentConditionLocalization}) of the statement)
	\begin{equation}
	\label{CEforfAtht}
	\forall\ v\in \mathbb{N},\ \ \ 
	f_1^{v-1}\ldots f_d^{v-1}\phi_d(1)/1+\big(\text{im}({\partial^{F_\bullet}_{d+1}})\big)_{{ht}}\notin (f_1^v,\ldots,f_d^v) \text{syz}^{{F_\bullet}_{ht}}_d\big(A[\fa t,t^{-1}]_{ht}/(\mathfrak{P})\big).
	\end{equation}
	To get a contradiction, suppose  the opposite of display (\ref{CEforfAtht})  for some $v\in \mathbb{N}$.  Then,   $$(ht)^{m'}f_1^{v-1}\ldots f_d^{v-1} \phi_d(1)+\text{im}({\partial^{F_\bullet}_{d+1}})\in (f_1^v,\ldots,f_d^v) \text{syz}^{F_\bullet}_d\big(A[\fa t,t^{-1}]/\mathfrak{P}\big)$$ for some $m'\in \mathbb{N}_0$. 
	
	If assertion (\ref{itm:SufficentConditionLocalization}) holds, then $A[\fa t,t^{-1}]_{\mathfrak{P}}$ satisfies the Canonical Element Conjecture and we get a contradiction with display  (\ref{CEforfAtht})  after localizing at $\mathfrak{P}$ (because $f_1,\ldots,f_d$ forms a system of parameters for $A[\fa t,t^{-1}]_{\mathfrak{P}}$ and $ht$ becomes invertible in $A[\fa t,t^{-1}]_{\mathfrak{P}}$). So let us assume that  assertion (\ref{itm:SufficientConditionExcellentUFD}) holds.
	
	Setting $w:=m'+v$
	\begin{equation}
	\label{CEforfAthtInR}
	(ht)^{w-1}f_1^{w-1}\ldots f_d^{w-1} \phi_d(1)+\text{im}({\partial^{F_\bullet}_{d+1}})\in (f_1^w,\ldots,f_d^w) \text{syz}^{F_\bullet}_d\big(A[\fa t,t^{-1}]/\mathfrak{P}\big).
	\end{equation}

	Let $G_\bullet$ be the mapping cone of $$(ht)^w.\id_{F_\bullet}:F_\bullet \overset{}{\rightarrow} F_\bullet,$$
 
  so $G_\bullet$ is a free resolution of $A[\fa t,t^{-1}]/(\mathfrak{P},(ht)^w)$ in view of Remark \ref{RemarkMappingConeFreeResolution}.  Also the mapping cone of  $$(ht)^w.\id_{K_\bullet(\mathbf{f};A[\fa t,t^{-1}])}:K_\bullet(\mathbf{f};A[\fa t,t^{-1}])\rightarrow K_\bullet(\mathbf{f};A[\fa t,t^{-1}])$$ coincides with the Koszul complex $K_\bullet(\mathbf{f},(ht)^w;A[\fa t,t^{-1}])$ (see \cite[Definition and Observations 5.2.1]{SchenzelSimonCompletion}, and also \cite[1.5.1]{SchenzelSimonCompletion} for the used notation). Then  $\phi_\bullet$  extends to a chain map on the mapping cones, resulting in the chain map  $$\psi_{\bullet}:K_\bullet(\mathbf{f},(ht)^w;A[\fa t,t^{-1}])\rightarrow G_\bullet,\ \ \ \psi_i:=(\phi_i,\phi_{i-1})$$ which lifts the natural epimorphism $A[\fa t,t^{-1}]/(\mathbf{f},(ht)^w)\twoheadrightarrow A[\fa t,t^{-1}]/(\mathfrak{P},(ht)^w)$. Then $\psi_{d+1}=\phi_d$, and the relation   in display (\ref{CEforfAthtInR}) yields 
 
	\begin{equation}
	\label{EquationOneMoreThing}
	  (ht)^{w-1}f_1^{w-1}\ldots f_d^{w-1}\psi_{d+1}(1)+\text{im}({\partial^{G_\bullet}_{d+2}})\in ((ht)^{w},f_1^w,\ldots,f_d^w) \text{syz}^{G_\bullet}_{d+1}\big(A[\fa t,t^{-1}]/(\mathfrak{P},(ht)^w)\big).
	\end{equation}  
	Namely, by  display (\ref{CEforfAthtInR}) we have $$(ht)^{w-1}f_1^{w-1}\cdots f_d^{w-1}\phi_d(1)=\sum\limits_{i=1}^df_i^{w}\alpha_i+\partial^{F_\bullet}_{d+1}(\beta),\ \ \alpha_1,\ldots,\alpha_d\in F_d,\ \beta\in F_{d+1}.$$  So, it is easily seen that $$(ht)^{w-1}f_1^{w-1}\cdots f_{d}^{w-1}\psi_{d+1}(1)=(ht)^w(-\beta,0)+\sum\limits_{i=1}^df_i^w(0,\alpha_i)+\partial^{G_\bullet}_{d+2}(0,-\beta).$$
	
	Applying   Lemma \ref{LemmaInducedMapOnKoszulHomolgoies} and composing $\psi_\bullet$ with the natural chain map $$\eta_w:K_\bullet((ht)^w,f_1^w,\ldots,f_d^w;A[\fa t,t^{-1}])\rightarrow K_\bullet((ht)^w,\mathbf{f};A[\fa t,t^{-1}]),$$ in view of display (\ref{EquationOneMoreThing}), we get the chain map
	$\Delta_\bullet:K_\bullet((ht)^w,f_1^w,\ldots,f_d^w;A[\fa t,t^{-1}])\rightarrow G_\bullet$ lifting   $A[\fa t,t^{-1}]/(\mathbf{f}^w,(ht)^w)\twoheadrightarrow A[\fa t,t^{-1}]/(\mathfrak{P},(ht)^w)$ such that 
	\begin{align*}
	(ht)^{w-1}\Delta_{d+1}(1)+\text{im}({\partial^{G_\bullet}_{d+2}})\in ((ht)^{w},f_1^w,\ldots,f_d^w) \text{syz}^{G_\bullet}_{d+1}\big(A[\fa t,t^{-1}]/(\mathfrak{P},(ht)^w)\big).
	\end{align*}

   However, by Step 5 and assertion (\ref{itm:SufficientConditionExcellentUFD}) in the statement of the theorem, we know that $A[\fa t,t^{-1}]_{\mathfrak{M}}$ satisfies the Strong Canonical Element Property and so this is a contradiction, completing this step.

	\textbf{Step 8: Proving an analogue of the CE Property  for $A[\fa t,t^{-1}]_{ht}$ and the sequence $\mathbf{z}$.} Now, we claim that the CE Property mentioned in display (\ref{CEforfAtht}) for the possibly inhomogeneous  sequence $f_1,\ldots,f_d$ in $A[\fa t,t^{-1}]_{ht}$ implies the similar CE Property for the homogeneous sequence $\mathbf{z}$ in $A[\fa t,t^{-1}]_{ht}$. Namely, we consider an $A[\fa t,t^{-1}]_{ht}$-graded free resolution $(P_\bullet,\partial^{P_\bullet}_\bullet)$ of $$A[\fa t,t^{-1}]_{ht}/(\fm_A,a_1t/ht,\ldots,a_gt/ht)$$ such that there is no shift involved in each graded free module in the graded free resolution $P_\bullet$. This no-shift-property is accessible because $A[\fa t,t^{-1}]_{ht}$ has an invertible element of degree $1$ and all syzygies are generated in degree $0$. 
	
	 Let  $\Psi_\bullet:K_\bullet(\mathbf{z};A[\fa t,t^{-1}]_{ht})\rightarrow P_\bullet$ be an arbitrary  homogeneous  chain map lifting the natural epimorphism $A[\fa t,t^{-1}]_{ht}/(\mathbf{z})\rightarrow A[\fa t,t^{-1}]_{ht}/(\fm_A,a_1t/ht,\ldots,a_gt/ht)$. We   show that
	\begin{equation}
	\label{CEforz}
	\forall\ v\in \mathbb{N},\ z_1^{v-1}\ldots z_d^{v-1}\Psi_d(1)+\text{im}({\partial^{P_\bullet}_{d+1}})\notin (z_1^v,\ldots,z_d^v) \text{syz}_d^{P_\bullet}\big(A[\fa t,t^{-1}]_{ht}/(\fm_A,a_1t/ht,\ldots,a_gt/ht)\big).
	\end{equation}
	
	If not, there is some $v$ such that display (\ref{CEforz}) is violated. Let $\Psi'_\bullet$ be the composition of the  chain maps  $$K_\bullet(\mathbf{z}^v;A[\fa t,t^{-1}]_{ht})\overset{\eta_{v}}{\longrightarrow} K_\bullet(\mathbf{z};A[\fa t,t^{-1}]_{ht})\overset{\Psi_\bullet}{\longrightarrow}P_\bullet$$ where 
	$\eta_v$ is as in Lemma \ref{LemmaInducedMapOnKoszulHomolgoies}(\ref{itm:PowerOfX}), so $\Psi'_d(1)=z_1^{v-1}\cdots z_d^{v-1}\Psi_d(1)$ by Lemma \ref{LemmaInducedMapOnKoszulHomolgoies}(\ref{itm:PowerOfX}).
  Then, by our hypothesis we get $$\Psi'_d(1)+\text{im}({\partial^{P_\bullet}_{d+1}})\in (z_1^{v},\ldots,z_d^{v})\text{syz}_d^{P_\bullet}\big(A[\fa t,t^{-1}]_{ht}/(\fm_A,a_1t/ht,\ldots,a_gt/ht)\big).$$  Hence,
	$$\Psi'_d(1)\overset{\im\ \partial^{P_\bullet}_{d+1}}{\equiv}\sum\limits_{i=1}^d z_i^v\alpha_i$$ 
	for some $\alpha_1,\ldots,\alpha_d\in P_d$. 
	
	Set $\Psi''_i:=\Psi'_i$ for each $i< d-1,$ $\Psi''_d:=0$ and assume that $\Psi_{d-1}''$ is given by  $$\Psi''_{d-1}(e_1\wedge\cdots\wedge \widehat{e_i}\wedge\cdots \wedge e_d)=\Psi'_{d-1}(e_1\wedge\cdots\wedge \widehat{e_i}\wedge\cdots \wedge e_d)-\partial^{P_\bullet}_{d}\big((-1)^{i+1}\alpha_i\big)
	.$$  

	By definition, it is clear that $$\partial^{P_\bullet}_{d-1}\circ \Psi''_{d-1}=\partial^{P_\bullet}_{d-1}\circ \Psi'_{d-1}=\Psi'_{d-2}\circ \partial^{K_\bullet(\mathbf{z}^v;A[\fa t,t^{-1}]_{ht})}_{d-1}=\Psi''_{d-2}\circ \partial^{K_\bullet(\mathbf{z}^v;A[\fa t,t^{-1}]_{ht})}_{d-1}.$$  It is also easily verified that $\Psi''_{d-1}\circ \partial^{K_\bullet(\mathbf{z}^v;A[\fa t,t^{-1}]_{ht})}_{d}=0$. Consequently, we get another chain map $$\Psi_\bullet'':K_\bullet(\mathbf{z}^v;A[\fa t,t^{-1}]_{ht})\rightarrow P_\bullet$$  lifting  $A[\fa t,t^{-1}]_{ht}/(\mathbf{z}^v)\overset{\text{nat. epi.}}{\twoheadrightarrow} A[\fa t,t^{-1}]_{ht}/(\fm_A,a_1t/ht,\ldots,a_gt/ht)$ such that $\Psi''_d=0$. 
	From display (\ref{GoodHomogeneousSystemOfParameters}) and the fact that $t^{-1}=h/ht\in (\fm_A,a_1t/ht,\ldots,a_gt/ht)$ we deduce that $$ (f_1,\ldots,f_d)A[\fa t,t^{-1}]_{ht}\subseteq (\mathfrak{P})= (\fm_A,a_1t/ht,\ldots,a_gt/ht)=\sqrt{(\mathbf{z})},$$
	implying that $(f_1^w,\ldots,f_d^w)\subseteq (z_1^v,\ldots,z_d^v)$ for some $w$. Thus  
	by Lemma \ref{LemmaInducedMapOnKoszulHomolgoies}(\ref{itm:GeneralSubIdealSequence}) there exists a chain map of Koszul complexes $$\lambda_\bullet:K_\bullet(f_1^w,\ldots,f_d^w;A[\fa t,t^{-1}]_{ht})\rightarrow K_\bullet(z_1^v,\ldots,z_d^v;A[\fa t,t^{-1}]_{ht})$$ lifting the natural epimorphism on the $0$-th homologies. Then the composition $$\Theta_\bullet:=\Psi''_\bullet\circ \lambda_\bullet:K_\bullet(f_1^w,\ldots,f_d^w;A[\fa t,t^{-1}]_{ht})\rightarrow P_\bullet$$ lifts the natural epimorphism on the $0$-homologies while $\Theta_d(1)=0$. Then since $P_\bullet$ and $(F_\bullet)_{ht}$ are both free resolutions of $A[\fa t,t^{-1}]_{ht}/(\mathfrak{P})$, so there is a chain map $\pi_\bullet:P_\bullet\rightarrow (F_\bullet)_{ht}$ lifting the identity map on $A[\fa t ,t^{-1}]_{ht}/(\mathfrak{P})$, by which we get a chain map $$\delta_\bullet:=\pi_\bullet\circ \Theta_\bullet:K_\bullet(f_1^w,\ldots,f_d^w;A[\fa t,t^{-1}]_{ht})\rightarrow (F_\bullet)_{ht}$$ which lifts the natural epimorphism on the $0$-homologies and $\delta_d(1)=0$. However,  the composited chain map     (the chain map $(\phi_\bullet)_{ht}$ is as in Step 7) $${_w}\Phi_\bullet:K_\bullet(f_1^{w},\ldots,f_d^{w};A[\fa t,t^{-1}]_{ht})\overset{{_{w}^{}\zeta^{}}}{\longrightarrow} K_\bullet(\mathbf{f};A[\fa t,t^{-1}]_{ht})\overset{(\phi_\bullet)_{ht}}{\longrightarrow}(F_\bullet)_{ht}$$ also lifts  $$A[\fa t,t^{-1}]_{ht}/(f_1^w,\ldots,f_d^w)\overset{\text{nat. epi.}}{\twoheadrightarrow} A[\fa t,t^{-1}]_{ht}/(\mathfrak{P})$$ where ${_{w}^{}\zeta^{}}$ is as in Lemma \ref{LemmaInducedMapOnKoszulHomolgoies}(\ref{itm:PowerOfX}).  Then, we get   
	 \begin{align*}
	   f_1^{w-1}\ldots f_d^{w-1}\phi_d(1)/1+\big(\text{im}(\partial_{d+1}^{F_\bullet})\big)_{_{ht}}
	 &\overset{}{=} 
	 {_w}\Phi_d(1)+\big(\text{im}(\partial_{d+1}^{{F_\bullet}})\big)_{ht} && (\text{Lemma \ref{LemmaInducedMapOnKoszulHomolgoies}(\ref{itm:PowerOfX})}) 
	 \\&=
	 {_w}\Phi_d(1)-\underset{=0}{\underbrace{\delta_{d}(1)}}+\big(\text{im}(\partial_{d+1}^{{F_\bullet}})\big)_{ht} 
	 &\\&\in 
	 (f_1^w,\ldots,f_d^w) \text{syz}^{{F_\bullet}_{ht}}_{d}\big(A[\fa t,t^{-1}]_{ht}/(\mathfrak{P})\big) && (\text{Lemma \ref{LemmaBasicFactInHomologicalAlgebra}(\ref{itm:RelationInSyz})}).
	 \end{align*}
	This contradicts  display  (\ref{CEforfAtht}).
	
	\textbf{Step 9: Proving an analogue of the CE Property  for $(A[\fa t,t^{-1}]_{ht})_{[0]}$ and the sequence $\mathbf{z}$.} We set $Q:=(A[\fa t,t^{-1}]_{ht})_{[0]}$. 
	Since there is no shift involved in each graded free module in the graded free resolution $P_\bullet$, so the degree zero part of each graded free module in the free resolution $P_\bullet$ is a finite direct sum of $Q$, hence a free $Q$-module. Consequently,  $(P_{\bullet})_{[0]}$ is a free resolution of  $Q/(\fm_A,a_1t/ht,\ldots,a_gt/ht)$.
	
	  Furthermore, $\big(K_\bullet(\mathbf{z};A[\fa t,t^{-1}]_{ht})\big)_{[0]}=K_\bullet(\mathbf{z};Q)$. Thus, $\Psi_\bullet$ induces  $(\Psi_{\bullet})_{[0]}:K_\bullet(\mathbf{z};Q)\rightarrow (P_\bullet)_{[0]}$.

	Moreover, the analogue to display (\ref{CEforz}) in $Q$ holds. Namely, $(\Psi_\bullet)_{[0]}:K_\bullet(\mathbf{z};Q)\rightarrow (P_\bullet)_{[0]}$ is a chain map lifting the natural surjection on $0$-th homologies such that 
	\begin{equation}
	\label{EquationLastCEPropertySeemingly}
	\forall\ v\in \mathbb{N},\ \ \ z_1^{v-1}\ldots z_d^{v-1}\underset{=\Psi_d(1)}{\underbrace{(\Psi_d)_{[0]}(1)}}+\text{im}({\partial^{(P_\bullet)_{[0]}}_{d+1}})\notin (z_1^v,\ldots,z_d^v)\text{syz}^{(P_\bullet)_{[0]}}_d\big(Q/(\fm_A,a_1t/ht,\ldots,a_gt/ht)\big).
	\end{equation}

	\textbf{Step 10: Proving the CE Property for $R$.} 
	 From Proposition \ref{StructureProposition}(\ref{itm:HomogeneousMaximalRegularSequenceOfEReesAlgebra}), it is easily deduced that $a_1t/ht,\ldots,a_gt/ht$ is a regular sequence of  $Q$. Moreover, from \cite[Theorem 3.5]{HerzogSimisVasconcelosKoszul}\footnote{Theorem 3.6 of the researchgate edition of \cite{HerzogSimisVasconcelosKoszul}.} in conjunction with Lemma  \ref{GoodRegularSequenceCorollary}(\ref{itm:DSequencePrimeACI}) we deduce that $\fa$ is an ideal of linear type, in other words, the  Rees algebra $A[\fa t]$ of $\mathfrak{a}$ coincides with the symmetric algebra $\text{Sym}_A(\fa)$ of $\mathfrak{a}$. Thus recalling our notation in display (\ref{EquationNotationForGenericLinkedACIObtained}), 
	from 
	\begin{align*}
	  Q/(a_1t/ht,\ldots,a_gt/ht)
	  &=(A[\fa t,t^{-1}]_{ht})_{[0]}/(a_1t/ht,\ldots,a_gt/ht)
	  &\\&\cong
	  (A[\fa t]_{ht})_{[0]}/(a_1t/ht,\ldots,a_gt/ht)  && A[\fa t]_{ht}\rightarrow A[\fa t,t^{-1}]_{ht}\text{\ is an isomorphism}
	  &\\&\cong
	  (\text{Sym}_{A}(\fa)_{ht})_{[0]}/(a_1t/ht,\ldots,a_gt/ht)  &&\fa\ \text{\ is an ideal of linear type}
	  &\\&\cong 
  	  A/(\fc:h)  &&  (\text{}\text{Lemma \ref{LemmaEquationOfSymmetricAlgebra}(\ref{itm:ColonIdealAsaSpecilizationOfTheSymmetricAlgebra})})
	  &\\& =
	  A/\fb,
	\end{align*}
	  we see that $Q$ is a (possibly non-local) deformation of $A/\fb$. Also, $A/\fb=B[\mathbf{X}]_{(\fm_B,\mathbf{X})}/\fb'B[\mathbf{X}]_{(\fm_B,\mathbf{X})}\cong R[\mathbf{X}]_{(\fm_R,\mathbf{X})}$ is a trivial deformation of $R$ by the image of the indeterminates $\mathbf{X}$. Consequently, $Q$ is a (possibly non-local) deformation of $R$.
	  
	  In view of the above display, $(\fm_A,a_1t/ht,\ldots,a_gt/ht)Q$ is a maximal ideal of $Q$. Moreover, we have $\sqrt{\mathbf{z}Q}=(\fm_A,a_1t/ht,\ldots,a_gt/ht)Q$ by  display (\ref{EquationRadical}).  In view of these facts as well as   display (\ref{EquationLastCEPropertySeemingly}),	  
	   we are in the situation of Lemma \ref{DeformationOFCEC}(\ref{itm:CETandSpecialization}). Hence $R$ satisfies the CE Property by Lemma \ref{DeformationOFCEC}(\ref{itm:CETandSpecialization}) (Alternatively, one may first conclude from  Lemma \ref{DeformationOFCEC}(\ref{itm:CETandSpecialization}) as well as the above display that $A/\fb$ satisfies the CE Property. Then,  our desired conclusion follows from a second use of Lemma \ref{DeformationOFCEC}(\ref{itm:CETandSpecialization})  this time with respect to the deformation $A/\fb$ of $R=(A/\fb)/(\mathbf{X})$, by replacing $Q$ (respectively, $Q/\mathbf{x}Q$) in the statement of the Lemma \ref{DeformationOFCEC}(\ref{itm:CETandSpecialization})  with $A/\fb$ (respectively, $R$)).
\end{proof}

For the notion  of   Hochster's modification module, mentioned in the statement of the next corollary, we refer to \cite[8.3 Modications and non degeneracy]{BrunsHerzogCohenMacaulay}. It turns out that  Hochster's modification module is a big Cohen-Macaulay module precisely when it remains non-zero after tensoring with the residue field;  this non-zero property follows from the existence of any balanced big Cohen-Macaulay module. Via the next corollary we show that  concluding the big Cohen-Macaulayness of  Hochster's modification module from the existence of a maximal Cohen-Macaulay complex, if possible, establishes the Balanced Big Cohen-Macaulay Module Conjecture by a characteristic free proof in general (see \cite[4. MCM complexes]{IyengarMaSchwedeWalker}  for the notion of maximal Cohen-Macaulay  complexes). Alternatively, one may ask whether the  existence of a maximal Cohen-Macaulay complex implies the  existence of a closure operation satisfying the axioms of \cite[Axioms 1.1]{DietzACharacterization}. See Remark \ref{RemarkWhyLocalizingBBCMModules}(\ref{itm:WhyHochsterModificationModule}), for more comments on the assumptions in the statement of the next corollary.

\begin{cor}\label{CorollaryBigCohenMacaulay}
  If the big Cohen-Macaulayness of  Hochster's modification module can be deduced from the existence of a maximal Cohen-Macaulay complex  by a characteristic free proof, then the (Balanced) Big Cohen-Macaulay Module Conjecture 	can be settled by a characteristic free proof.
  \begin{proof} 
     We show that our hypothesis on deducing  the big Cohen-Macaulayness of  Hochster's modification module from the existence of a maximal Cohen-Macaulay complex implies that the Canonical Element Conjecture holds  by a characteristic free proof, while the latter  is equivalent to the existence of a maximal Cohen-Macaulay complex for any complete local ring by \cite[Proposition 4.3]{IyengarMaSchwedeWalker} and \cite[Theorem (4.3)]{HochsterCanonical}. Then a second use of our hypothesis in conjunction with Remark \ref{RemarkBalancedBCMofCompletion}(\ref{itm:CompletionBigCM}) and  \cite[Proposition 4.3]{IyengarMaSchwedeWalker} settles the statement  by a characteristic free proof.
     
     To conclude the validity of the  Canonical Element Conjecture by a characteristic free proof, in view of Theorem \ref{MonomialTheoremCharacteristicFree}(\ref{itm:SufficentConditionLocalization}), it suffices to show that any localization of a normal excellent local domain  $R$ satisfies the Canonical Element Conjecture provided  $R$ does. So let $R$ be  an excellent normal local domain which satisfies the Canonical Element Conjecture.  Then $\widehat{R}$  satisfies the Canonical Element Conjecture (\cite[Proposition (3.18)(b)]{HochsterCanonical}). Thus $\widehat{R}$ admits a maximal Cohen-Macaulay complex by \cite[Proposition 4.3]{IyengarMaSchwedeWalker} and \cite[Theorem (4.3)]{HochsterCanonical}, and so  $\widehat{R}$, and thence $R$, admits  a balanced big Cohen-Macaulay module by our hypothesis as well as Remark \ref{RemarkBalancedBCMofCompletion} (by a characteristic free proof). It then turns out that  any localization of $R$ admits a balanced big Cohen-Macaulay module by Theorem \ref{TheoremBalancedBigCMModuleLocalizes}, a fortiori any localization of $R$ satisfies the Canonical Element Conjecture as was to be proved.
  \end{proof}
\end{cor}

\begin{rem}\label{RemarkWhyLocalizingBBCMModules}\emph{
		It is perhaps appropriate to elucidate why we presented Theorem \ref{TheoremBalancedBigCMModuleLocalizes} with a  longer proof  than the proof of \cite[Proposition 2.11]{BaMaPaScTuWaWi21} and why we did not  simply refer to the already existing result \cite[Proposition 2.11]{BaMaPaScTuWaWi21} in the literature? We reason this as follows:}
\end{rem}	\vspace{-4mm}
	\begin{enumerate}[(i)]
		\item
			Our Theorem \ref{TheoremBalancedBigCMModuleLocalizes} is a new result showing that the  algebra structure of balanced big Cohen-Macaulay algebras over complete normal (or excellent normal) local   domains is not required for their localizing property. 			
		\item\label{itm:WhyHochsterModificationModule} 
			The main result of Section \ref{SectionNewVariant}, i.e. Corollary  \ref{CorollaryVariantHolds}, can be deduced from  \cite[Proposition 2.11]{BaMaPaScTuWaWi21} in place of Theorem \ref{TheoremBalancedBigCMModuleLocalizes}. But, in contrast,    Corollary \ref{CorollaryBigCohenMacaulay} essentially requires Theorem \ref{TheoremBalancedBigCMModuleLocalizes} and does not follow from \cite[Proposition 2.11]{BaMaPaScTuWaWi21}. This is because in  the statement of Corollary \ref{CorollaryBigCohenMacaulay}, we mention that   Hochster's modification module  is big Cohen-Macaulay (not  Hochster's modification algebra), so we would need in its proof the localizing property of balanced big Cohen-Macaulay modules rather than balanced big Cohen-Macaulay algebras.  For more comments, recall that the existence of  a maximal Cohen-Macaulay complex  (without an algebra structure) is equivalent to the Canonical Element Conjecture. Since such maximal Cohen-Macaulay complexes are not known to carry a DG-algebra structure (or any algebra structure), it   seems  unlikely (in the eyes of the author) to deduce the big Cohen-Macaulayness of Hochster's modification algebra  from the existence of such a maximal Cohen-Macaulay complex. This is the reason that in the statement of Corollary \ref{CorollaryBigCohenMacaulay} we assumed the deduction of big Cohen-Macaulayness of Hochster's modification module (not algebra) from the existence of a maximal Cohen-Macaulay complex. Here, it seems also necessary  to bring to the reader's attention that  algebras of finite type over a field of characteristic zero admit a maximal Cohen-Macaulay complex carrying a DG-algebra structure (see \cite[Subsection 4.3]{IyengarMaSchwedeWalker}). 
			\item As a main result of the present paper we showed  that for obtaining a characteristic free proof of the Canonical Element Theorem it suffices to show that it is stable under the localization over  excellent factorial (or excellent normal) domains. So  perhaps providing other results  in this direction as we did in Theorem \ref{TheoremBalancedBigCMModuleLocalizes}  is promising that one can also establish the stability of the CE Property under the localization. As an immediate question in this direction, one may ask whether the proof of Theorem \ref{TheoremBalancedBigCMModuleLocalizes} can be modified appropriately to show that the maximal Cohen-Macaulay complexes behave well under the localization. An affirmative answer to this question  establishes  the Canonical Element Theorem with a characteristic free proof.
	\end{enumerate}

\begin{rem} \emph{In \cite{TavanfarAnnihilators}, we showed that the  Monomial Theorem is equivalent to the assertion that for any (possibly non-unmixed) almost complete intersection ring $(R,\mathfrak{m})$ and any system of parameters $\mathbf{x}$ of $R$,  $$\big((\mathbf{x}):_R\fm\big)H_1(\mathbf{x};R)=0$$ (see \cite[Proposition 2.11]{TavanfarAnnihilators}). Then a question (not a conjecture) is  proposed in \cite{TavanfarAnnihilators}. Are all positive Koszul homologies $H_{i\ge 1}(\mathbf{x};R)$   killed by $(\mathbf{x}):\fm$ for  any system of parameters $\mathbf{x}$  of any almost complete intersection $(R,\fm)$ (see \cite[Question 1.1]{TavanfarAnnihilators} and \cite[Question 1.2]{TavanfarAnnihilators})? We conclude with  two comments regarding this question:}
		
		\begin{itemize}
			\item  \emph{
				We are aware of an example showing that the almost complete intersection hypothesis in   \cite[Proposition 2.11]{TavanfarAnnihilators} (and so in \cite[Question 1.1]{TavanfarAnnihilators}) is necessary and can not be relaxed. More precisely, such an example violates the inclusion $\big((\mathbf{x}):_R\fm\big)\subseteq 0:_RH_1(\mathbf{x};R)$     while this example still fulfills $(\mathbf{x})\subsetneq  0:_RH_1(\mathbf{x};R)$ (here $\mathbf{x}$ is a system of parameters for $R$ where $R$ is indeed not an almost complete intersection).
			 }
			
			\item \emph{As a positive result concerning \cite[Question 1.1 and Question 1.2]{TavanfarAnnihilators}, in \cite{TavanfarResolving}\footnote{The   project \cite{TavanfarResolving} is in its very early stage and its   ultimate version, if/when it is available  in the future, can have a  different title or/and  can have more authors.} we showed that any almost complete intersection $(R,\fm)$ admits a system of parameters $\mathbf{x}$ as well as some $z\in \big((\mathbf{x}):_R\fm\big)\backslash (\mathbf{x})$ such that $zH_{i\ge 1}(\mathbf{x};R)=0$. In other words, any  almost complete intersection   $(R,\fm)$ admits  a system of parameters $\mathbf{x}$ and  an  $\mathbf{x}R$-socle element $z$ (as above) for which  the residual approximation complex $\mathcal{Z}^+_\bullet(\mathbf{x},z;R)$ (which is a non-free finite complex)  is acyclic with $H_0\big(\mathcal{Z}^+_\bullet(\mathbf{x},z;R)\big)=R/\fm$.}
	  \end{itemize}
\end{rem}

\section*{Acknowledgement}

We wish to express our  deep appreciation to  Raymond Heitmann for lots of fruitful  and accurate comments on our paper as well as for providing   certain   alternative proofs for some of our results.  We are grateful to Simon H\"aberli for his    many   comments which improved the presentation of the paper. 
We are also grateful to Linquan Ma and S. Hamid Hassanzadeh  for  their  valuable
comments
 on this work.


\end{document}